\documentclass[12pt,a4paper]{article}
\usepackage{amssymb}
\usepackage{amsmath}
\usepackage{amsthm}
\usepackage{thmtools}
\usepackage{enumitem}
\usepackage[all]{xy}
\usepackage{framed}
\usepackage{hyperref}
\usepackage{parskip}
\usepackage[a4paper,margin=2cm]{geometry}
\usepackage{multicol}
\usepackage{marginnote}
\usepackage{xcolor}
\usepackage[normalem]{ulem}
\usepackage{bm}

\begingroup
    \makeatletter
    \@for\theoremstyle:=definition,remark,plain\do{%
        \expandafter\g@addto@macro\csname th@\theoremstyle\endcsname{%
            \addtolength\thm@preskip\parskip
            }%
        }
\endgroup

\makeatletter
\xydef@\txt@ii#1{\vbox{\vspace*{-5pt}%
 \let\\=\cr
 \tabskip=\z@skip \halign{\relax\hfil\txtline@@{##}\hfil\cr\leavevmode#1\crcr}}}
\makeatother

\theoremstyle{definition}
\newtheorem{thm}{Theorem}[section]
\newtheorem{lem}[thm]{Lemma}

\newtheorem{cor}[thm]{Corollary}
\newtheorem{defn}[thm]{Definition}
\newtheorem{propn}[thm]{Proposition}
\newtheorem*{thm*}{Theorem}

\newtheorem*{qn*}{Question}
\newtheorem{conj}[thm]{Conjecture}
\newtheorem{props}[thm]{Properties}

\newcounter{lthm}
\newtheorem{letterthm}[lthm]{Theorem}

\theoremstyle{remark}
\newtheorem{rk}[thm]{Remark}
\newtheorem*{rk*}{Remark}
\newtheorem{rks}[thm]{Remarks}

\newcommand{\gr}{\mathrm{gr}}
\newcommand{\Aut}{\mathrm{Aut}}
\newcommand{\id}{\mathrm{id}}

\newcommand{\Out}{\mathrm{Out}}
\newcommand{\supp}{\mathrm{supp}}
\newcommand{\chr}{\mathrm{char}}

\renewcommand{\O}{\mathcal{O}}

\newcommand\blfootnote[1]{%
  \begingroup
  \renewcommand\thefootnote{}\footnote{#1}%
  \addtocounter{footnote}{-1}%
  \endgroup
}

\begin{document}

\numberwithin{equation}{section}
\binoppenalty=\maxdimen
\relpenalty=\maxdimen

\title{Skew power series rings with automorphisms of finite inner order}
\author{Adam Jones and William Woods}
\date{\today}
\maketitle
\begin{abstract}

\blfootnote{\emph{2010 Mathematics Subject Classification}: 16W60, 16W70, 16W80, 16S35, 16S36, 16K40}

\noindent We investigate the algebraic properties of the bounded skew power series ring $Q^+[[x;\sigma,\delta]]$ over a (complete, simple) \emph{standard} filtered artinian algebra $Q$ of positive characteristic. Here we are assuming that $(\sigma,\delta)$ is a commuting skew derivation of $Q$, where $\delta$ is inner, satisfying the appropriate compatibility conditions. In a previous work of the authors, it was proved that $Q^+[[x;\sigma,\delta]]$ is a simple ring whenever $\sigma$ has infinite inner order. We now extend this result to the case when $\sigma$ has finite inner order, proving that this ring is often simple, and always prime in cases of interest. This solves an important special case of an open question of Letzter, and yields important consequences for the classification of prime ideals in Iwasawa algebras of solvable groups.

\end{abstract}
\tableofcontents

\section{Introduction and motivation}

This article is the fourth in a series of papers by the authors \cite{jones-woods-1,jones-woods-2,jones-woods-3}, focusing on the algebraic properties of large classes of (topological) skew power series rings.

\subsection{Skew power series rings}

Let $R$ be a ring and $(\sigma, \delta)$ a skew derivation on $R$: for us, this will mean that $\sigma$ is an automorphism of $R$ and $\delta$ is a left $\sigma$-derivation on $R$, i.e.\ $\delta(ab) = \delta(a)b + \sigma(a)\delta(b)$ for all $a, b\in R$. This allows us to define the \emph{skew polynomial ring}
$$R[x; \sigma, \delta] = \left\{\sum_{n=0}^N r_n x^n : r_n \in R, N\in\mathbb{N}\right\},$$
a ring isomorphic as a left and right $R$-module to $R[x]$, but whose multiplication is given by the family of relations $xr = \sigma(r)x + \delta(r)$ for all $r\in R$.

Under certain appropriate compatibility conditions between $(\sigma, \delta)$ and the topology on $R$, we can further define the \emph{skew power series ring}
\begin{align}\label{eqn: unbounded SPSR}
R[[x; \sigma, \delta]] = \left\{\sum_{n\geq 0} r_n x^n: r_n\in R\right\},
\end{align}
subject to the same multiplication rule. The \emph{automorphic} case ($\delta = 0$) is well-known classically \cite[\S 1.4]{MR}, and the case where $\delta$ is \emph{nilpotent} or \emph{locally nilpotent} has been studied recently by many authors, including Bergen and Grzeszczuk \cite{bergen-grzeszczuk-skew}.

We call $\delta$ an \emph{inner} $\sigma$-derivation if there exists $t\in R$ such that $\delta(r) = tr - \sigma(r)t$ for all $r\in R$. In this case, it is especially easy to handle the corresponding skew polynomial ring, but the skew power series ring is much more subtle, and work is ongoing: see \cite[\S 1.2]{jones-woods-2} for a discussion.

More recently, due to its applications in the field of Iwasawa algebras (see \S \ref{subsec: iwasawa algebras} below), there is growing interest in the more general case where $\delta$ satisfies a kind of \emph{topological} nilpotency property. In these cases, it can be delicate to prove that products in $R[[x; \sigma, \delta]]$ converge. Early papers \cite{venjakob, SchVen06, letzter-noeth-skew} restricted themselves to the case where $R$ is complete with respect to some $I$-adic filtration and $(\sigma, \delta)$ satisfies some rather rigid compatibility conditions with respect to this filtration.

However, this setup is too restrictive for our purposes, as we detail in \S \ref{subsec: filtered localisation}. For this reason, we require the significantly weaker conditions on the filtration on $R$ and $(\sigma, \delta)$ introduced in \cite[\S\S 3.1--3.2]{jones-woods-3}.

\subsection{Iwasawa algebras as skew power series rings}\label{subsec: iwasawa algebras}

An ongoing project in representation theory is to understand and classify the prime ideal structure of \emph{Iwasawa algebras}: see e.g.\ the open question \cite[3.19]{letzter-noeth-skew}, and the many open questions of \cite{ardakovbrown}. Considerable progress has been made on this subject to date (see e.g.\ \cite{ardakovInv,reflexive,jones-abelian-by-procyclic,jones-primitive-ideals}), and we hope that this work will form another chapter in this investigation.

Let $p$ be a prime number, and let $G$ a compact $p$-adic Lie group, which may be most simply defined as a closed subgroup of $GL_n(\mathbb{Z}_p)$. Let $k=\mathbb{F}_p$ or $\mathbb{Z}_p$, and define the \emph{completed group algebra} (or \emph{Iwasawa algebra}) of $G$ over $k$ to be the algebra $kG := k[[G]]$ defined by $$kG=\underset{U\lhd_o G}{\varprojlim}k[G/U],$$ where the inverse limit ranges over all open normal subgroups of $G$, and $k[G/U]$ denotes the usual discrete group algebra of $G/U$ over $k$. This algebra is of fundamental importance to the representation theory of $G$, as modules over $kG$ characterise \emph{continuous} $k$-linear representations of $G$.

Suppose for simplicity that $G$ is a \emph{$p$-valuable} group \cite[Chapitre III, D\'efinition 3.1.6]{lazard}, i.e.\ a torsion-free compact $p$-adic Lie group equipped with a well-behaved group filtration $\omega: G\to \mathbb{N}\cup\{\infty\}$ \cite[Chapitre III, D\'efinition 2.1.2]{lazard}, which induces a ring filtration $w: kG \to \mathbb{N}\cup\{\infty\}$. If $H\lhd G$ is a closed normal subgroup satisfying $G/H \cong \mathbb{Z}_p$, it is known (cf.\ \cite[Example 2.3ff.]{venjakob}) that we can write $kG$ as a skew power series ring
\begin{align}\label{eqn: iwasawa as spsr}
kG = kH[[x; \sigma, \delta]]
\end{align}
for appropriate $x$, $\sigma$ and $\delta$. If $I$ is any $G$-invariant ideal of $kH$, then $IkG$ is a two-sided ideal of $kG$, and \eqref{eqn: iwasawa as spsr} descends to the quotient \cite[Lemma 3.14(iv)]{letzter-noeth-skew}:
$$kG/IkG = (kH/I)[[\overline{x};\overline{\sigma},\overline{\delta}]].$$
Indeed, in this case, we can always choose $\delta = \sigma - \id$: note that this is an inner $\sigma$-derivation defined with respect to the $\sigma$-invariant element $t = -1$.

A central motivating question for the authors has been the following:

\textbf{Question.} Suppose $P$ is a $G$-invariant prime ideal of $kH$. Is $PkG$ a (semi)prime ideal of $kG$?

(Equivalently: is $(kH/P)[[\overline{x}; \overline{\sigma}, \overline{\delta}]]$ a (semi)prime ring?) This deceptively simple question is a special case of \cite[3.19]{letzter-noeth-skew}, which considers both a broader class of skew power series rings and a weaker invariance condition on their ideals.

Appropriate direct analogues of this question for skew polynomial rings, group algebras, crossed products, skew power series rings of automorphic type, etc.\ tend to be very straightforward to prove (see \cite[Theorems 1.2.9(i)(iii), 1.4.5, 1.5.11]{MR}). Their proofs usually implicitly rely on the existence of a well-behaved filtration-like function, such as the \emph{degree} of a polynomial or the \emph{order} of a power series: however, the order function is badly-behaved for general skew power series rings, justifying our search for a better filtration.

\subsection{Filtered localisation}\label{subsec: filtered localisation}

An important classical tool in the study of (semi)prime Goldie rings $R$ is passing to their (semi)simple artinian rings of quotients $Q(R)$. Any skew derivation $(\sigma,\delta)$ of $R$ extends uniquely to $Q(R)$ \cite[Lemma 1.3]{goodearl-skew-poly-and-quantized}, and when studying skew polynomial rings, we have the crucial relation
\begin{align}\label{eqn: tensor product of skew polynomial rings}
Q(R)[x;\sigma,\delta]=Q(R)\otimes_R R[x;\sigma,\delta],
\end{align}
which allows us to simplify many arguments: see e.g.\ \cite[2.3(iv), \S 3]{goodearl-letzter}.

This will also be a key tool for us throughout. However, filtrations generally interact poorly with localisation, and if we localise naively, we lose many desirable properties. The resulting filtration on $Q(R)$ will of course no longer be adic, but in general it will also no longer be complete, and hence it may not be possible to define an appropriate notion of skew power series ring over $Q(R)$. Moreover, even if $R$ is a commutative domain, the completion of $Q(R)$ with respect to this filtration will in general fail to be semisimple artinian.

However, there is a small subclass of filtered simple artinian rings that are very well understood. Roughly: we define a noncommutative notion of complete \emph{discrete valuation ring} $D$ (Definition \ref{defn: DVR}), so that its Goldie ring of quotients $F$ is a complete discretely valued division ring and we may consider matrix rings $Q = M_s(F)$ with the canonically induced filtrations (Definition \ref{defn: matrix filtration}). Any $Q$ that can be constructed in such a way is called a \emph{standard} filtered artinian ring (Definition \ref{defn: standard filtrations}).

The properties of such filtered rings are very well understood, but at first glance, they might seem rather rare. Indeed, naively localising a filtered ring $R$ will produce something very far from a standard filtration on $Q(R)$ in general. However, the fundamental \emph{filtered localisation} theorem tells us that, under relatively mild conditions, we are able to pass from a prime Noetherian filtered ring $(R, w)$ to a \emph{standard} filtered artinian ring $(Q, u)$ in a controlled way. The first version of this theorem appeared as \cite[Theorem C]{ardakovInv}, and it has since been successively refined as \cite[Theorem 3.3]{jones-abelian-by-procyclic}, implicitly in \cite[\S 4]{jones-woods-1}, and as \cite[Theorem 4.1.1]{jones-woods-3}. We prove another small refinement as Theorem \ref{thm: filtered localisation} below, and use it crucially throughout what follows.

This $Q$ is the completion (under $u$) of the Goldie ring of quotients $Q(R)$, but the process of constructing $u$ from $w$ is intricate: see \cite[\S 4]{jones-woods-3}. In performing this construction, we will usually lose any strong compatibility relation between $w$ and $(\sigma, \delta)$. However, we retain the property of \emph{quasi-compatibility}: this is a much weaker condition, but it still allows us to define the bounded skew power series ring $Q^+[[x;\sigma,\delta]]$ (Definition \ref{defn: bounded skew power series rings}).

There is growing evidence that these rings are independently far better behaved than $R[[x; \sigma, \delta]]$ (see \cite[Theorem C]{jones-woods-2}, \cite[Theorems B--C]{jones-woods-3}), and that first understanding $Q^+[[x; \sigma, \delta]]$ is the right avenue for understanding properties of $R[[x; \sigma, \delta]]$ itself (\cite[Theorem 2.6]{jones-woods-1}, \cite[Theorem D]{jones-woods-2}, \cite[Theorems D--E]{jones-woods-3}). This will also be the focus of the current paper.

\subsection{The main conjecture}

Suppose $Q$ is a simple artinian $\mathbb{Z}_p$-algebra carrying a standard filtration $u$, and suppose $(\sigma, \delta)$ is a \emph{commuting} skew derivation (that is, $\sigma\delta = \delta\sigma$) on $Q$. If $(\sigma, \delta)$ is \emph{quasi-compatible} with $u$ (in the sense of \cite[Definition 3.1.1]{jones-woods-3}), then the bounded skew power series ring $Q^+[[x; \sigma, \delta]]$ is known to exist \cite[Theorem A]{jones-woods-3}. Our main conjecture regarding the algebraic structure of this ring can be simply stated:

\begin{conj}\label{conj: main}
$Q^+[[x;\sigma,\delta]]$ is a (semi)prime ring.
\end{conj}

Using the filtered localisation procedure described in \S \ref{subsec: filtered localisation}, it should follow from this conjecture that any skew power series ring defined over a complete, filtered prime ring is indeed (semi)prime, thus answering the question from \S \ref{subsec: iwasawa algebras} in the positive.

The existing partial results in support of Conjecture \ref{conj: main} in the literature usually rely on understanding a well-behaved subring $Q^+[[X_n; \Sigma_n, \Delta_n]] \subseteq Q^+[[x; \sigma, \delta]]$. In characteristic $p$, for example, we may always take $(X_n, \Sigma_n, \Delta_n) = (x^{p^n}; \sigma^{p^n}, \delta^{p^n})$. It is known under mild conditions that such subrings are prime for sufficiently large $n$ \cite[Theorem A(a)]{jones-woods-1}, but we currently have no way of passing this information back to the larger ring. In characteristic $0$, we are unable to define any such $(X_n, \Sigma_n, \Delta_n)$ in general.

The situation improves considerably when $\delta$ is inner, i.e.\ there exists $t\in Q$ such that $\delta(q) = tq - \sigma(q)t$. To obtain good results, we also seem to need the additional assumption that $\sigma(t) = t$, but this is mild in our setting \cite[\S 3.4]{jones-woods-3}, so we may assume it. Here, there is a consistent notion of $(X_n, \Sigma_n, \Delta_n)$ in arbitrary characteristic, and in light of results such as \cite[Proposition 4.4.5, Theorem 2.3.4]{jones-woods-3} we can usually assume that $Q^+[[X_n;\Sigma_n, \Delta_n]]$ is prime for large $n$.

In characteristic 0, and with the additional assumption that $u(t) \geq 0$, \cite[Theorem B]{jones-woods-3} shows that $Q^+[[x; \sigma, \delta]]$ is semiprime under these conditions. In characteristic $p$, in the case where no positive power of $\sigma$ is an inner automorphism of $Q$, it follows from \cite[Theorem C]{jones-woods-3} that $Q^+[[x;\sigma,\delta]]$ is prime; furthermore, if $u(t)=u(t^{-1})=0$, then it is in fact a \emph{simple} ring. Other special cases of Conjecture \ref{conj: main} are largely unanswered.

Focusing on the case where $\delta$ is inner and $\chr(Q)=p$, we assume from now on that $\sigma^k$ is an inner automorphism for some $k\geq 1$. The discussion now splits into cases according to the value of $u(t)$.

If $u(t)>0$, then the ring $Q^+[[x;\sigma,\delta]]$ can be realised as (a localisation of) a skew power series ring of \emph{automorphic type} by \cite[Theorem B]{jones-woods-2}, a classical and well-studied object (see e.g.\ \cite[\S 1.4]{MR}). The case $u(t) < 0$ still seems mysterious, and we are unable to address it at present. Instead, we focus on the intermediate case where $u(t) = u(t^{-1}) = 0$, and we will outline below how the results of this paper finally complete the proof of Conjecture \ref{conj: main} under this assumption in cases of interest.

\subsection{Main results of the current paper}\label{subsec: main results}

From now on, suppose that $Q$ is a simple artinian ring of characteristic $p$, that $u$ is a standard filtration on $Q$, and that $(\sigma, \delta)$ is a commuting skew derivation on $Q$. Assume also that we may fix some $m\in\mathbb{N}$ such that $(\sigma^{p^m}, \delta^{p^m})$ is compatible with $u$. We also assume that $\delta$ is inner, i.e.\ there exists $t\in Q$ such that $\delta(q)=tq-\sigma(q)t$ for all $q\in Q$, and that $\sigma(t) = t$.

Throughout the paper, we will denote by $[\sigma]$ the image of $\sigma$ under the canonical map $\Aut(Q) \to \Out(Q)$ (though we will normally remind the reader explicitly that $[\sigma] \in \Out(Q)$). We assume that $[\sigma]\in\Out(Q)$ has finite order $n$, so that there exists $a_0\in Q^\times$ such that $\sigma^n(q) = a_0qa_0^{-1}$.

We say that the triple $(\sigma, \delta, u)$ satisfies the \emph{strong finite order hypothesis} (SFOH) if
\begin{equation}
    \tag{SFOH}
    \text{\parbox{.78\textwidth}{
	there exist $a\in \O^\times$ and $\ell\in\mathbb{N}$ such that $\sigma^{p^\ell}(q) = aqa^{-1}$ for all $q\in Q$.
    }}
    \label{SFOH}
\end{equation}

\begin{letterthm}\label{letterthm: SFOH implies prime}
Assume that $(\sigma, \delta, u)$ satisfies \eqref{SFOH}, say there exist $a\in \O^\times$ and $\ell\in\mathbb{N}$ such that $\sigma^{p^\ell}(q) = aqa^{-1}$ for all $q\in Q$. Suppose also that one of the following two conditions holds:
\begin{enumerate}[label=(\alph*), noitemsep]
\item $u(t^{p^M}) > 0$ for some $M \geq \max\{\ell, m\}$.
\item $u(a^{p^{M-\ell}} + t^{p^M}) > 0$ for some $M \geq \max\{\ell, m\}$.
\end{enumerate}
Then $Q^+[[x; \sigma, \delta]]$ is a prime ring.\qed
\end{letterthm}

\textbf{Note:} unlike in the case where $[\sigma]\in\Out(Q)$ has infinite order, $Q^+[[x;\sigma,\delta]]$ need not be a simple ring with these assumptions. For example, if $a = 1$ and $t = -1$, then condition (b) is satisfied for $\ell = M = 0$, but $Q^+[[x;\sigma,\delta]] = Q^+[[x]]$.

We will complete the proof of Theorem \ref{letterthm: SFOH implies prime} at the end of \S \ref{section: crossed product decomposition}. The condition \eqref{SFOH} is quite restrictive, but without it the relationship between $Q^+[[x; \sigma, \delta]]$ and its subrings becomes increasingly difficult to control.

If we cannot find such a unit $a\in\O^\times$ satisfying the conditions of Theorem \ref{letterthm: SFOH implies prime}, we are still free to multiply the element $a_0\in Q^\times$ by any nonzero element $\zeta$ of the centre $Z$ of $Q$, and it will still hold that $\sigma^n(q)=(\zeta a_0)q(\zeta a_0)^{-1}$ for all $q\in Q$. However, there may exist no $\zeta\in Z$ such that $\zeta a_0\in\O^\times$. 

\textbf{Idea:} we extend $Z$ to a larger field $K$ which contains such a $\zeta$: in other words, beginning with $Q$ and a finite extension of fields $K/Z$, we will form the central scalar extension $Q_K := K\otimes_Z Q$.

In using this idea practically, some care is required: in order to learn anything about $Q^+[[x; \sigma, \delta]]$ by doing this, we must pass from the triple $(\sigma, \delta, u)$ of data on $Q$ to a triple $(\sigma_K, \delta_K, u_K)$ of data on $Q_K$, in such a way that we are able to understand $Q_K^+[[y; \sigma_K, \delta_K]]$ and derive useful information about $Q^+[[x; \sigma, \delta]]$ from it.

Recall that $(Q, u)$ is constructed from an associated complete discrete valuation ring $D$. We outline a few definitions that are important to this notion of scalar extension.
\begin{itemize}
\item $D$ is \emph{admissible} if its associated graded ring has finite rank as a module over its centre (Definition \ref{defn: admissible}). This is the crucial condition that will allow us to extend $u$ to $u_K$.
\item An \emph{admissible datum} on $Q$ is a triple $(\sigma, \delta, u)$ whose associated complete discrete valuation ring $D$ is admissible (plus a few extra mild hypotheses for convenience): see Definition \ref{defn: admissible data} and Remarks \ref{rks: AD hypotheses} for full details.
\item Given such triples $(\sigma, \delta, u)$ and $(\sigma_K, \delta_K, u_K)$ on $Q$ and $Q_K$ respectively, we say that $(\sigma_K, \delta_K, u_K)$ is an \emph{admissible extension} of $(\sigma, \delta, u)$ if (roughly) we may understand properties of $Q^+[[x; \sigma, \delta]]$ from those of $Q_K^+[[y; \sigma_K, \delta_K]]$: see Definition \ref{defn: admissible extension} and Remarks \ref{rks: AE hypotheses}.
\end{itemize}

With these definitions in mind, we say that the triple $(\sigma, \delta, u)$ satisfies the \emph{finite order hypothesis} (FOH) if it is an \emph{admissible} datum, and if
\begin{equation}
    \tag{FOH}
    \text{\parbox{.78\textwidth}{
	there exist a finite extension of fields $K/Z$ and an admissible datum $(\sigma_K, \delta_K, u_K)$ on $Q_K$ which is an admissible extension of $(\sigma, \delta, u)$ and satisfies \eqref{SFOH}.
    }}
    \label{FOH}
\end{equation}

\textbf{Note:} it is essential in the definition of an admissible datum that the automorphism $\sigma$ act by the identity on $Z$. Fortunately, with our assumptions on $(\sigma,\delta)$, it follows from Lemma \ref{lem: identical on centre} below that $\sigma|_Z$ has order $p^s$ for some $s\in\mathbb{N}$, so this is only a finite reduction which is easy to remove (see Corollary \ref{cor: when sigma does not fix Z part 2}).

On the surface, the finite order hypothesis may not seem much more general than \eqref{SFOH} itself. But surprisingly, it turns out to be all we need in our cases of interest:

\begin{letterthm}\label{letterthm: FOH implies prime}
Assume that the discrete valuation ring $D$ associated to $Q$ is admissible. Suppose also that $u(t^N)=0$ for all $N\in\mathbb{N}$. Then setting $p^s$ as the order of $\sigma|_Z$:
\begin{enumerate}[label=(\alph*),noitemsep]
\item If $(\sigma^{p^s},\delta^{p^s},u)$ satisfies \eqref{FOH}, then $Q^+[[x;\sigma,\delta]]$ is a prime ring.
\item If $(\sigma^{p^s},\delta^{p^s},u)$ does not satisfy \eqref{FOH}, then $Q^+[[x;\sigma,\delta]]$ is a simple ring.
\end{enumerate}
In particular, $Q^+[[x;\sigma,\delta]]$ is always a prime ring.\qed
\end{letterthm}
\color{black}

This result, together with \cite[Theorem C]{jones-woods-3}, completes the proof of a crucial special case of Conjecture \ref{conj: main}. Moreover, it demonstrates that despite how restrictive \eqref{FOH} may seem, it is in fact the only condition which allows $Q^+[[x;\sigma,\delta]]$ to have any two-sided ideals at all.

We prove Theorem \ref{letterthm: FOH implies prime} in \S \ref{sec: FOH}, where we demonstrate that any non-trivial two-sided ideal in $Q^+[[x;\sigma,\delta]]$ gives rise to a finite extension satisfying the conditions of \eqref{SFOH}.

The proofs of Theorem \ref{letterthm: SFOH implies prime} and Theorem \ref{letterthm: FOH implies prime} are very technical, and this paper is principally dedicated to their proof. In \S \ref{subsec: proof of remaining theorems}, we prove the following consequences, which we hope will be of use both in the study of Iwasawa algebras, and of skew power series rings more generally. The following results assume some knowledge of the results of \cite{jones-woods-3}.

\begin{letterthm}\label{letterthm: compatible Iwasawa SPSR is prime}
Let $R$ be a prime $\mathbb{F}_p$-algebra with a complete, positive filtration $w_0$. Assume that $\gr_{w_0}(R)$ is Noetherian and finitely generated as a module over a central, graded subring $A$ containing a non-nilpotent element of positive degree. 

Let $(\sigma,\delta)$ be a commuting skew derivation on $R$, compatible with $w_0$, and assume there exists $t\in R^{\times}$ such that $\delta(r)=tr-\sigma(r)t$ and $\sigma(t) = t$ for all $r\in R$. Then $R[[x;\sigma,\delta]]$ is a prime ring.\qed
\end{letterthm}

We deduce an immediate consequence of this result
which generalises an important result of Ardakov \cite[Theorem 8.6]{ardakovInv}. In the following, $k$ is any field of characteristic $p$. Recall that $Q\lhd kH$ is \emph{$G$-prime} if, given any $G$-invariant ideals $A$ and $B$ of $kH/Q$, we have $AB \neq 0$. Equivalently \cite[Remarks 4* and 5*]{GolMic74}, $Q = T_1 \cap \dots \cap T_t$, where $\{T_1, \dots, T_t\}$ is a $G$-orbit of prime ideals of $kH$. For instance, if $P$ is a prime ideal of $kG$, then $P\cap kH$ is $G$-prime.

\begin{letterthm}\label{letterthm: primes in completely solvable Iwasawa algebras}
Let $G$ be a $p$-valuable group, and let $H$ be a closed normal subgroup of $G$ such that $G/H$ is completely solvable and torsion-free. If $Q$ is a $G$-prime ideal of the Iwasawa algebra $kH$, then $QkG$ is a prime ideal of $kG$.
\qed
\end{letterthm}

\textbf{Acknowledgements:} The first author is very grateful to the Heilbronn Institute for Mathematical Research for funding and supporting this research. The second author would like to thank Konstantin Ardakov for a very productive research meeting in Oxford, resulting in a significantly shorter proof of Theorem \ref{thm: conditions for primality}.

\section{Preliminaries}\label{sec: preliminaries}

\subsection{Filtrations and filtration topologies}\label{subsec: filtrations}

We will adopt the same conventions as in \cite[\S 2.1]{jones-woods-3}, which we briefly outline here for the convenience of the reader.

\begin{defn}\label{defn: filtrations} $ $

\begin{enumerate}[label=(\roman*)]
\item 
Let $(A, +)$ be an abelian group. If the function $u: A\to \mathbb{Z}\cup\{\infty\}$ satisfies both $u(a+b) \geq \min\{u(a), u(b)\}$ and ($u(a) = \infty \Leftrightarrow a = 0$) for all $a,b\in A$, then we say that $(A,u)$ is a \emph{filtered abelian group}, and that $u$ is a \emph{filtration (of abelian groups)} on $A$.

Let $R$ be a ring. If $u: R\to \mathbb{Z}\cup\{\infty\}$ is a filtration of abelian groups on $(R, +)$, and also $u(1) = 0$ and $u(rs) \geq u(r) + u(s)$ for all $r,s\in R$, then $(R,u)$ is a \emph{filtered ring}, and $u$ is a \emph{(ring) filtration} on $R$. If $u(rs) = u(r) + u(s)$, then $(R,u)$ will be called a \emph{valued} ring, and $u$ a \emph{valuation} on $R$.

Let $(R,u)$ be a filtered ring, and let $M$ be a left $R$-module. If $v: M\to\mathbb{Z}\cup\{\infty\}$ is a filtration of abelian groups on $(M, +)$, and also $v(rm) \geq u(r) + v(m)$ for all $m\in M$ and $r\in R$, then $(M, v)$ is a \emph{filtered left $R$-module}, and $v$ is a \emph{(left $R$-module) filtration} on $M$.

\item If $(A,u)$ is a filtered abelian group (resp.\ filtered ring, filtered module) and $n\in\mathbb{Z}$, the \emph{$n$th level set} of $(A,u)$ is the subset $\{a\in A: u(a) \geq n\}$. We will often denote this by $F_nA$ or $F^u_nA$, and will write $FA = \{F_nA\}_{n\in\mathbb{Z}}$ for the sequence of level sets.
\item Suppose $(A,u)$ and $(B,v)$ are filtered abelian groups and $f: A\to B$ is a group homomorphism. If there exists $d\in\mathbb{Z}$ such that, for all $n\in \mathbb{Z}$, $f(F_n A) \subseteq F_{n+d} B$, or in other words such that $v(f(a)) \geq u(a) + d$ for all $a\in A$, then we will say that $f$ is \emph{filtered} (with respect to $u$ and $v$). The maximal such $d$ will be called the \emph{degree} of $f$, written $\deg(f)$ or $\deg_{u,v}(f)$. In the special case where $(A,u) = (B,v)$, we will write this degree as $\deg_u(f)$.

If $f(F_n A) = f(A) \cap F_n B$, then we will say that $f$ is \emph{strictly filtered}: of course, this implies that $\deg(f) = 0$, and if $f$ is an isomorphism then $f^{-1}$ is also strictly filtered.
\end{enumerate}
\end{defn}

\begin{defn}\label{defn:associated graded} $ $
\begin{enumerate}[label=(\roman*)]
\item If $(R,u)$ is a filtered ring, with level sets $FR$, we define the \emph{associated graded ring} of $(R,u)$ to be the abelian group $$\gr_u(R):=\underset{n\in\mathbb{N}}{\bigoplus}{F_nR/F_{n+1}R}$$ with a graded ring structure defined by $(r+F_{n+1}R)\cdot(s+F_{m+1}R)=rs+F_{n+m+1}R$. If $u(r)=n$ then we denote by $\gr(r)$ the element $r+F_{n+1}R$ of $\gr_u(R)$.
\item If $(M,f)$ is a filtered left $R$-module, we define the \emph{associated graded module} of $(M,f)$ to be the abelian group $$\gr_f(M):=\underset{n\in\mathbb{N}}{\bigoplus}{F_nM/F_{n+1}M}$$ which is a $\gr_u(R)$-module via $(r+F_{n+1}R)\cdot (m+F_{m+1}M)=rm+F_{n+m+1}M$.
\end{enumerate}
\end{defn}

\begin{lem}\label{lem: u-regular elements}
Suppose $(R,u)$ is a filtered ring, $r\in R$ is arbitrary, and $z\in R^\times$ satisfies $u(z^{-1}) = -u(z)$. Then $u(zr) = u(z)+u(r)$.
\end{lem}

\begin{proof}
Using filtration properties, we can calculate
$$u(r) = u(z^{-1}zr) \geq u(z^{-1}) + u(zr) \geq u(z^{-1}) + u(z) + u(r).$$
Since the right-hand side is equal to $u(r)$, these are all in fact equalities, and the result follows immediately.
\end{proof}

\begin{lem}\label{lem: graded zero divisors}
If $(R,u)$ is a filtered ring, with associated graded ring $\gr_u(R)$, then for any $r\in R$, $\gr(r)$ is not a zero divisor in $\gr_u(R)$ if and only if $u(rs)=u(r)+u(s)$ for all $s\in R$.
\end{lem}

\begin{proof}
Since $\gr(r)$ is homogeneous, it is a zero divisor if and only if $\gr(r)m = 0$ for some homogeneous $m\in \gr_u(R)$, which is equivalent to saying that $\gr(r)\gr(s)=0$ for some $s\in R$, i.e.\ $(r+F_{u(r)+1}R)\cdot (s+F_{u(s)+1}R)=0$.

But $(r+F_{u(r)+1}R)\cdot (s+F_{u(s)+1}R)=rs+F_{u(r)+u(s)+1}R=0$ if and only if $u(rs)>u(r)+u(s)$, so $\gr(r)$ is a zero divisor if and only if there exists $s\in R$ such that $u(rs)>u(r)+u(s)$. Since $u(rs)\geq u(r)+u(s)$ for all $s\in R$, the result follows immediately.\end{proof}

Many of the rings $Q$ we will consider are simple artinian, and so may be viewed as matrix rings over division rings, $Q\cong M_s(F)$. In previous papers, we have needed to keep track of the isomorphism explicitly \cite[Hypothesis (H1)]{jones-woods-2}, but in this paper, we will usually be able to assume without loss of generality that our simple artinian rings $Q$ come with a predetermined canonical set of matrix units, and we will identify $Q = M_s(F)$. We now describe the natural way of inducing a filtration from a ring $R$ to a matrix ring over $R$.

\begin{defn}\label{defn: matrix filtration}
Let $(R, u)$ be a filtered ring and $s\geq 1$ an integer. Write $\{e_{ij}\}_{1\leq i,j\leq s}$ for the standard matrix units of $M_s(R)$. Then the filtration $M_s(u): M_s(R) \to \mathbb{Z}\cup\{\infty\}$ is defined to be the map sending $\sum_{i,j} a_{ij} e_{ij}$ to $\min_{i,j} \{u(a_{ij})\}$, and is called a \emph{matrix filtration}. The associated graded ring of $M_s(u)$ is $M_s(\gr_u(R))$.
\end{defn}

Filtrations give rise to a natural topology. The condition $u(a+b) \geq \min \{u(a), u(b)\}$ acts as an ultrametric inequality, and the resulting topological space is nonarchimedean in this sense.

\begin{defn}\label{defn: filtration topology} $ $

\begin{enumerate}[label=(\roman*)]
\item If $(A,u)$ is a filtered abelian group (resp.\ filtered ring, filtered module), we will always take it together with its \emph{filtration topology}: a subset $U\subseteq A$ is defined to be open if and only if, for all $x\in U$, there exists $n\in\mathbb{Z}$ such that $x + F_nA \subseteq U$.

If $(A,u)$ is a filtered abelian group (resp.\ filtered ring, filtered module), its filtration topology is an \emph{abelian group topology} (resp.\ \emph{ring topology}, \emph{module topology}) in the sense of \cite[Definitions 1.1, 2.1]{warner}: see e.g.\ \cite[Chapter I, Property 3.1(e)]{LVO}. In particular, many natural algebraic operations such as addition, multiplication and negation are continuous.
\item Let $(A,u)$ be a filtered abelian group, and $(a_i)_{i\in\mathbb{N}}$ a sequence of elements of $A$. We say that
\begin{itemize}
\item 
$(a_i)$ \emph{converges to} $a\in A$ if, for all $M\in\mathbb{Z}$, there exists $N\in\mathbb{N}$ such that $i\geq N$ implies $u(a_i - a) \geq M$.
\item
$(a_i)$ is \emph{Cauchy} if, for all $M\in\mathbb{Z}$, there exists $N\in\mathbb{N}$ such that $i,j\geq N$ implies $u(a_i - a_j) \geq M$.
\item $(A,u)$ is \emph{complete} if all Cauchy sequences in $A$ converge.
\end{itemize}
\end{enumerate}
\end{defn}

\begin{lem}\label{lem: continuous homomorphisms}
\cite[Lemma 2.1.6]{jones-woods-3}
Let $(A,u)$ and $(B,v)$ be filtered abelian groups, with collections of level sets $FA$ and $FB$ respectively. If $f: A\to B$ is a linear map, then it is continuous if and only if, for all $q\in\mathbb{Z}$, there exists $p\in\mathbb{Z}$ such that $F_pA \subseteq f^{-1}(F_q B)$.\qed
\end{lem}

\begin{defn}\label{defn: equivalent filtrations}\cite[Chapter I, \S 3.2]{LVO}
Two filtrations $u$ and $v$ on a module $M$ are \emph{topologically equivalent} if, for all $p,q\in\mathbb{Z}$, there exist $m,n\in\mathbb{Z}$ such that $F^u_m M \subseteq F^v_p M$ and $F^v_n M \subseteq F^u_q M$. Equivalently, by Lemma \ref{lem: continuous homomorphisms}, the identity map $M\to M$ must be continuous both as a map $(M,u)\to (M,v)$ and as a map $(M,v)\to (M,u)$. 
\end{defn}

Next, we consider the appropriate analogue of \emph{free} modules. In our context, we will only ever need free modules of finite rank, so we immediately restrict to this slightly simpler case.

\begin{defn}\label{defn: filt-free}\cite[Chapter I, Definition 6.1]{LVO}
Let $(R, u)$ be a filtered ring, and let $(M, f)$ be a filtered, free right $R$-module of finite rank $\ell$. Then $M$ is \emph{filt-free} if there is a right $R$-module basis $\{\alpha_1, \dots, \alpha_\ell\}$ and a collection of integers $k_1, \dots, k_\ell$ such that $f(\alpha_i) = k_i$ for all $i$ and
$$F_n M = \alpha_1 (F_{n-k_1} R) \oplus \dots \oplus \alpha_\ell (F_{n-k_\ell} R).$$
In other words, for any $m\in M$, there exists a unique choice of elements $r_1, \dots, r_\ell\in R$ such that $m = \alpha_1 r_1 + \dots + \alpha_\ell r_\ell$, and we may calculate
$$f(m) = \min_{1\leq i\leq \ell}\{f(\alpha_i) + u(r_i)\} = \min_{1\leq i\leq \ell}\{k_i + u(r_i)\}.$$
We will say that $\{\alpha_1, \dots, \alpha_\ell\}$ is a \emph{filt-basis} for $M$. (Left modules can be handled similarly.)
\end{defn}

It is not difficult to show that $M$ is filt-free if and only if $\gr_f(M)$ is a free $\gr_u(R)$-module of finite rank with a basis of homogeneous elements \cite[Chapter I, Lemma 6.2(2--3)]{LVO}.

We record the following easy property for reference later.

\begin{lem}\label{lem: transitivity of filt-free extensions}
Let $(R,u)$, $(S,v)$ and $(T,w)$ be filtered rings such that $R\subseteq S\subseteq T$, and suppose $(S,v)$ is a filt-free right $(R,u)$-module with filt-basis $\{\beta_1, \dots, \beta_m\}$ and $(T,w)$ is a filt-free right $(S,v)$-module with filt-basis $\{\alpha_1, \dots, \alpha_\ell\}$. Then $(T,w)$ is a filt-free right $(R,u)$-module with filt-basis $\{\alpha_i \beta_j : 1\leq i\leq \ell, 1\leq j\leq m\}$.\qed
\end{lem}

\subsection{The filtered tensor product}\label{subsec: central extension of scalars}

We begin by recalling some well-known facts about tensor products of rings. In all the below results, $k$ is a field, and $A$ and $B$ are arbitrary $k$-algebras. This means that $A$ is embedded in $A\otimes_k B$ by the injective map $a \mapsto a\otimes 1$, and likewise $B$ is embedded in $A\otimes_k B$, as taking the tensor product with a free module is an exact functor.

\begin{lem}\label{lem: centre of tensor product}
\cite[Corollary 5.4.4]{Coh03a} The centre of $A\otimes_k B$ is $Z(A) \otimes_k Z(B)$.\qed
\end{lem}

As tensor products of modules commute with direct sums, we also get:

\begin{lem}\label{lem: basis for K extends to basis for F_K}
Let $A$ and $B$ be $k$-algebras, and suppose that $B$ has a finite basis $\{\beta_1,\dots,\beta_d\}$ over $k$. Then $\{1\otimes \beta_1, \dots, 1\otimes \beta_d\}$ is a basis for $A\otimes_k B$ as a free $B$-module.\qed
\end{lem}

In the below lemmas, we will fix a field extension $K/k$, and denote by $A_K$ the tensor product $K\otimes_k A$. Recall that, if $A$ is a $k$-algebra, then $M_n(A)$ is also a $k$-algebra via the diagonal map.

\begin{lem}\label{lem: Q_K = M_n(F_K)}\cite[Theorem 5.4.1 and (5.4.5)]{Coh03a}
$A_K\otimes_k M_n(k) \cong K\otimes_k M_n(A) \cong M_n(A_K)$.
\end{lem}

\begin{propn}\label{propn: simplicity under tensor product}
If $A_K$ is prime (resp.\ simple), then $A$ is prime (resp.\ simple). Moreover, if $Z(A)=k$ and $A$ is simple, then $A_K$ is simple.
\end{propn}

\begin{proof}
Given any two-sided ideal $J\lhd A$, we will write $J_K = K\otimes_k J$ for the two-sided ideal of $A_K$. Note that $J_K\cap A = J$ (see e.g.\ \cite[proof of Theorem 5.1.2]{Coh03b}).

Assume $A_K$ is simple, and take an arbitrary nonzero ideal $J\lhd A$. Then $J_K$ is a nonzero ideal of $A_K$, and so it must be equal to $A_K$. But then $J = J_K\cap A = A_K \cap A = A$, so $A$ is simple.

Next, assume $A_K$ is prime, and take two nonzero ideals $J, J'\lhd A$. Then $J_K, J'_K\lhd A_K$ are nonzero ideals, and hence $J_KJ'_K \neq 0$. But $J_KJ'_K = (K\otimes_k J)(K\otimes_k J') = (JJ')_K$, and so $JJ' \neq 0$. Hence $A$ is prime.

Finally, if $A$ is simple and $Z(A) = k$, then the claim follows from standard results about (not necessarily finite-dimensional) central simple algebras \cite[Theorem 5.1.2]{Coh03b}.\end{proof}

Next, we will turn to tensor products of filtered algebras. Much of the following can be done in far greater generality: see \cite[Chapter I, \S 2 and \S 6]{LVO}. However, again, we quickly restrict to the case of interest to us.

\begin{defn}\label{def: filtered tensor product} \cite[Chapter I, \S 6, p.\ 57]{LVO}
Let $(R,u)$ be a filtered ring, $(M,f)$ be a filtered right $R$-module, and $(N,g)$ be a filtered left $R$-module. Then we can define their \emph{filtered tensor product}, an abelian group $M\otimes_R N$, whose $\ell$th level set $F_\ell(M\otimes_R N)$ is equal to the subgroup of $M\otimes_R N$ generated by all elements $m\otimes n$, where $m\in F_s M$ and $n\in F_t N$, such that $s+t \geq \ell$.

It will be useful to have the following explicit rephrasing: writing the filtration on $M\otimes_R N$ as $f\otimes g$, an element $x\in M\otimes_R N$ satisfies $(f\otimes g)(x) \geq \ell$ precisely when there exists a representation $x = \sum_{i=1}^r m_i \otimes n_i$ such that $f(m_i) + g(n_i) \geq \ell$ for all $1\leq i\leq r$.

Put yet another way: writing $\pi: M\otimes_\mathbb{Z} N\to M\otimes_R N$ for the natural surjective map (of abelian groups), we can define the $\ell$th level set as an abelian group: $$F_\ell(M\otimes_R N) = \sum_{s+t \geq \ell} \pi(F_s M \otimes_\mathbb{Z} F_t N).$$
\end{defn}

The following is now clear from the definitions.

\begin{lem}
Let $(R,u)$, $(A,f)$ and $(B,g)$ be filtered rings such that $(A,f)$ is a filtered right $R$-module and $(B,g)$ is a filtered left $R$-module. Then the tensor product filtration $f\otimes g$ on $A\otimes_R B$ is a ring filtration.\qed
\end{lem}

In the following two lemmas, let $k$ be a field, and $A$ and $B$ two $k$-algebras. For ease of notation, we will assume that $k\subseteq A$ and $k\subseteq B$. Suppose $A$ and $B$ are given ring filtrations $f$ and $g$ respectively, with the property that $f|_k = g|_k$, and write $\varphi = f|_k$ for the induced filtration on $k$.

\begin{lem}\label{lem: filt-free module under tensor product}
\cite[Chapter I, Lemma 6.15]{LVO}
If $(A, f)$ is filt-free as a right $(k, \varphi)$-module with filt-basis $\{e_1, \dots, e_r\}$, then $(A\otimes_k B, f\otimes g)$ is filt-free of finite rank as a right $(B, g)$-module with filt-basis $\{e_1\otimes 1, \dots, e_r\otimes 1\}$. Moreover, $(f\otimes g)(e_i\otimes 1) = f(e_i)$ for all $i$. In particular, if $B$ is complete, then $A\otimes_k B$ is also complete.
\end{lem}

\begin{proof}
By the definition of $\varphi = f|_k$, the inclusion map $k\to A$ is a filtered ring homomorphism of degree 0, and so we apply \cite[Chapter I, Lemma 6.15]{LVO}.
\end{proof}

\begin{lem}\label{lem: associated graded of filt-free}
\cite[Chapter I, Lemma 6.2(2)]{LVO} Suppose that $(A, f)$ is filt-free as a right $(k, \varphi)$-module  with filt-basis $\{e_1, \dots, e_r\}$. Then $\gr_f(A)$ is free as a right $\gr_\varphi(k)$-module with homogeneous basis $\{\gr(e_1), \dots, \gr(e_r)\}$.
\end{lem}

\pushQED{\qed}
\begin{lem}\label{lem: associated graded of tensor product}
\cite[Chapter I, Lemma 6.14]{LVO} Suppose that $(A, f)$ is filt-free as a right $(k, \varphi)$-module. Then
\[
\gr_{f\otimes g}(A\otimes_k B) \cong \gr_f(A) \otimes_{\gr_\varphi(k)} \gr_g(B).\qedhere
\]
\end{lem}
\popQED

\subsection{Discrete valuation rings and standard filtrations}\label{subsec: DVRs}

\begin{defn}\label{defn: DVR}
As in \cite[\S 3.6]{ardakovInv} and \cite[\S 2.2]{jones-woods-3}, a \emph{discrete valuation ring} is a Noetherian domain $D$ such that, for every nonzero $x\in Q(D)$ (its Goldie division ring of quotients), we have either $x\in D$ or $x^{-1}\in D$.
\end{defn}

\begin{props}\label{props: DVRs}
Suppose $D$ is a discrete valuation ring, and let $F = Q(D)$ be its Goldie division ring of quotients. Most of the following facts were proved in \cite[\S 2.5]{jones-woods-2} and \cite[\S 2.2]{jones-woods-3}.
\begin{enumerate}[label=(\roman*)]
\item $D$ is a (scalar) local ring. This implies that $D^\times = D\setminus J(D)$, and so in particular, if $x\in F$ but $x\not\in D$, then $x^{-1}\in J(D)$.
\item The Jacobson radical is given by $J(D) = \pi D$ for some normal element $\pi$, and every nonzero left or right ideal has the form $\pi^n D=J(D)^n$ for some $n\geq 0$. It follows that $\bigcap_{n\in\mathbb{N}} \pi^n D = \{0\}$. Indeed, we can take $\pi$ to be any element of $J(D)\backslash J(D)^2$, and we call such an element $\pi$ a \emph{uniformiser} for $D$.
\item We will always assume that a discrete valuation ring $D$ comes equipped with its natural $J(D)$-adic valuation, i.e.\ the valuation $v: D\to \mathbb{N}\cup\{\infty\}$ defined by $$v(x)=\sup\{n\in\mathbb{N}:x\in\pi^nD\}.$$
Since $\pi$ is normal and $D$ is a domain, $\gr_v(D)$ is isomorphic to a skew polynomial ring $(D/J(D))[\overline{\pi}; \alpha]$ for some automorphism $\alpha$, while $\gr_v(F)$ is isomorphic to the skew Laurent polynomial ring $(D/J(D))[\overline{\pi},\overline{\pi}^{-1}; \alpha]$.
\item $J(D)$ is invertible, with inverse $J(D)^{-1} = \pi^{-1}D \subseteq F$. Hence the sets $\{\pi^n D\}_{n\in\mathbb{Z}}$ are the level sets of a filtration on $F$. We will call this the filtration \emph{induced} by $v$ on $F$, and will often continue to denote it by $v$ if there is no danger of confusion.
\end{enumerate}
\end{props}

In \cite[Introduction, Definition 2.2.3]{jones-woods-3}, we introduced the following terminology for certain particularly well-behaved filtered rings.

\begin{defn}\label{defn: standard filtrations}
Suppose $Q$ is a simple artinian ring with filtration $u$. We will say that $u$ is a \emph{standard} filtration on $Q$, or that $(Q,u)$ is a \emph{standard filtered artinian ring}, if: 
\begin{enumerate}[label=(\roman*)]
\item there exists an identification $Q = M_s(F)$ such that $u = M_s(v)$ for some integer $s\geq 1$ and some valued division ring $(F,v)$, and
\item the valuation $v$ is induced from a complete discrete valuation ring $D\subseteq F$ satisfying $F = Q(D)$.
\end{enumerate}
In this case, the complete DVR $D$ and its (valued) division ring of fractions $F$ are uniquely defined, and we may say that they are \emph{associated} to $Q$. There is also a canonically associated maximal order $\O := u^{-1}([0,\infty]) = M_s(D)$, the valuation ring of $(Q,u)$, and the level sets of the filtration $u$ are $\{J(\O)^n : n\in\mathbb{Z}\}$.
\end{defn}

We recall a basic fact about standard filtrations: see \cite[Lemma 4.5]{jones-woods-1}, or \cite[Lemma 4.2.1]{jones-woods-3}.

\begin{lem}\label{lem: criterion for delta positive degree}
Suppose $(Q, u)$ is a standard filtered artinian ring with associated maximal order $\O$, and let $(\sigma, \delta)$ be any skew derivation on $Q$. If $\sigma(\O) = \O$ and $\delta(\O) \subseteq J(\O)^2$, then $\deg_u(\delta) \geq 1$. Likewise, if $(\sigma - \id)(\O) \subseteq J(\O)^2$, then $\deg_u(\sigma - \id) \geq 1$. \qed
\end{lem}

\begin{propn}\label{propn: gr(Q)}
\cite[Proposition 3.13]{ardakovInv} Suppose $(Q, u)$ is a standard filtered artinian ring. Then $\gr_u(Q)$ is isomorphic as a graded ring to $C[Y, Y^{-1}; \alpha]$, a skew Laurent polynomial ring over a finite-dimensional central simple algebra $C$ with grading given by the natural $\deg_Y$ function.\qed
\end{propn}

\begin{rk*}
This was proven as \cite[Proposition 3.13]{ardakovInv} for \cite[Theorem C(c)]{ardakovInv}. Note that there is a small error later in the proof: writing $\pi$ for a uniformiser of $D$, and viewing it as an element of $\O$ under the diagonal embedding, the ring automorphism $\alpha:C\to C$ is defined to be $q+J(\O)\mapsto \pi q\pi^{-1}+J(\O)$. But this is \emph{not} a $Z(C)$-algebra homomorphism in general, so the Skolem-Noether theorem cannot be applied as it is in the proof of \cite[Theorem 3.13]{ardakovInv}. However, this error does not affect our results.
\end{rk*}

\subsection{The filtered localisation procedure}

When considering a prime filtered ring $(R,w_0)$, we want to canonically extend the filtration to the Goldie ring of quotients $Q(R)$, since filtrations of simple artinian rings are often easier to work with in practice, as demonstrated in \S \ref{subsec: DVRs}.

A procedure for constructing a canonical, standard filtration on $Q(R)$ was first given in \cite[\S 3]{ardakovInv}, and this was subsequently expanded upon and refined in \cite[\S\S 3.1--3.2]{jones-abelian-by-procyclic}, \cite[\S 4.1]{jones-woods-1} and \cite[\S 4.1]{jones-woods-3}. In this paper, we will principally rely on results proved in these works, and the procedure itself will not play a significant role, but there are some technical details that have not been proved elsewhere, and which will have fundamental importance in \S \ref{section: extension of scalars}, when we extend standard filtrations to scalar extensions. We will now record these results here.

\begin{thm}\label{thm: filtered localisation}
Let $R$ be a prime, Noetherian $\mathbb{Z}_p$-algebra carrying a complete, separated filtration $w_0:R\to\mathbb{Z}\cup\{\infty\}$, such that $\gr_{w_0}(R)$ is Noetherian and finitely generated as a module over a central, graded subring $A$ containing a non-nilpotent element of positive degree. Then:

\begin{enumerate}
\item The Goldie ring of quotients $Q(R)$ carries a filtration $u:Q(R)\to\mathbb{Z}\cup\{\infty\}$ such that

\begin{enumerate}[label=(\roman*), noitemsep]
\item the inclusion $(R,w_0)\to (Q(R),u)$ is continuous,

\item if $w_0(r)\geq 0$ then $u(r)\geq 0$ for any $r\in R$,

\item the completion $Q := \widehat{Q(R)}$ is a simple artinian ring,

\item the unique extension of $u$ from $Q(R)$ to $Q$ makes $(Q,u)$ into a standard filtered artinian ring,

\item the degree-0 part $\gr_u(Q)_0 = \O / J(\O)$ of the associated graded ring is a central simple algebra, where $\O$ is the maximal order associated to $(Q, u)$.
\end{enumerate}

\item If $R$ is an $\mathbb{F}_p$-algebra, then $\gr_u(Q)$ is finitely generated as a module over a central, graded subdomain.

\item If we assume further that $R$ is artinian and $A$ contains no zero divisors in $\gr_{w_0}(R)$, then $R=Q(R)=Q$, and $u$ is topologically equivalent to $w_0$.
\end{enumerate}
\end{thm}

\begin{rk*}
For parts 2 and 3 of the proof, we will need to assume some familiarity with the construction of the filtration $u$, as given in e.g.\ \cite{ardakovInv, jones-abelian-by-procyclic, jones-woods-3, jones-woods-1}. We briefly recap the construction to set up our notation:

\begin{itemize}
\item Choose a minimal prime ideal $\mathfrak{q}$ of $A$ not containing the positive part $w_0^{-1}((0,\infty])$ of $A$. (We know such a $\mathfrak{q}$ exists by our assumption that $A$ contains a non-nilpotent element of positive degree.) Set $S:=\{r\in R:\gr_{w_0}(r)\in A\setminus\mathfrak{q}\}$. This is a localisable subset of $R$, and $S^{-1}R=Q(R)$ by the proof of \cite[Lemma 3.3]{ardakovInv}.
\item Define a filtration $w:Q(R)\to\mathbb{Z}\cup\{\infty\}$ from $w_0$ using \cite[Proposition 2.3]{li-ore-sets}. This satisfies $w(x)=\max\{w_0(r)-w_0(s)\}$, where this maximum is taken over all $r\in R$ and $s\in S$ such that $x=s^{-1}r$. (See e.g.\ \cite[Proposition 1.14]{jones-woods-1} for more details.)
\item \cite[\S 3.4]{ardakovInv} Form the completion $Q'$ of $Q(R)$ with respect to $w$. This is an artinian ring, but will in general not be semisimple, so choose a maximal ideal $M$ of $Q'$, and set $Q:=Q'/M$. Define the noetherian subring $U=\{q\in Q':w(q)\geq 0\}$, and let $V = (U + M)/M$ be the image of $U$ in $Q$.
\item \cite[Lemma 3.2]{jones-abelian-by-procyclic} There exists an element $z\in J(U)$ which is regular and normal in $U$, such that $U$ is $z$-adically complete. Write $\overline{z}$ for its image in $V$.
\item Choose a maximal order $\mathcal{O}$ in $Q$, containing and equivalent to $V$, which exists by \cite[Theorem 3.11]{ardakovInv}. The desired filtration $u$ is the $J(\O)$-adic filtration on $Q$ \cite[Theorem 3.6, Proposition 3.13]{ardakovInv}.
\end{itemize}
\end{rk*}

\begin{proof}
$ $

\begin{enumerate}
\item Parts (i)--(iv) are given by \cite[Theorem 3.3]{jones-abelian-by-procyclic}, and part (v) follows from \cite[Theorem 3.6 and Theorem 3.13]{ardakovInv}: see also \cite[Theorem 4.1.1 and \S 4.1]{jones-woods-3} for further details.

\item Since $Q$ has characteristic $p$, we see using the proof of \cite[Proposition 3.5]{jones-abelian-by-procyclic} that $t := \overline{z}^{p^m}$ is normal in $\O$ for sufficiently high $m$, and $u(tqt^{-1} - q)>u(q)$ for all $q\in Q$, which implies that $\gr_u(t)$ is central in $\gr_u(Q)$.

Now $\gr_u(Q)\cong C[Y,Y^{-1};\alpha]$, a skew Laurent polynomial ring over a finite-dimensional central simple algebra $C$, by Proposition \ref{propn: gr(Q)}. Hence $\gr_u(t) = cY^k =: T$ for some $c\in C$ and $k\in\mathbb{N}$, and so since it is central, this implies that $\alpha^k$ is inner, given by conjugation by $c^{-1}$.

As $T$ is clearly a unit in $\gr_u(Q)$, we have a central graded subring $Z(C)[T,T^{-1}]$ of $\gr_u(Q)$. Since $C$ is finitely generated over $Z(C)$, it follows that $\gr_u(Q)$ is finitely generated over the central, graded subdomain $Z(C)[T,T^{-1}]$ as required.

\item Since $R$ is artinian and prime, it is simple, and hence $R=Q(R)$. We will first prove that the filtration $w$ on $R$ coincides with $w_0$.

Since $A$ contains no zero divisors, it must be a domain, so the minimal prime ideal $\mathfrak{q}$ is zero, and thus $S=\{r\in R:\gr(r)\in A$ and $\gr(r)\neq 0\}$. It follows that for all $s\in S$, $\gr(s)\in A$ cannot be a zero divisor in $\gr_{w_0}(R)$, which means that $w_0(rs)=w_0(sr)=w_0(s)+w_0(r)$ for all $r\in R$ by Lemma \ref{lem: graded zero divisors}. Therefore, for any $x\in R$, if $s\in S$ and $r:=sx$, then $w_0(r)=w_0(sx)=w_0(s)+w_0(x)$, and thus $w_0(r)-w_0(s)=w_0(x)$. Since our choice of $s$ was arbitary, we see that:
\begin{align*}
w(x) &=\max\{w_0(r)-w_0(s):r\in R,s\in S, x=s^{-1}r\}\\
&=\max\{w_0(r)-w_0(s):r\in R,s\in S, sx=r\}\\
& =\max\{w_0(x):r\in R,s\in S, sx=r\}=w_0(x)
\end{align*}
Since our choice of $x\in R$ was arbitrary, it follows that $w=w_0$ as required.

As we are assuming that $R=Q(R)$ is complete with respect to $w_0$, this means that $Q'=Q(R)$. Moreover, since $Q(R)$ is simple, the maximal ideal $M$ of $Q'=Q(R)$ must be zero, so $$\widehat{Q(R)}=Q=Q'/M=Q(R)/M=Q(R).$$ 
Finally, we know $u$ is topologically equivalent to the $\overline{z}V$-adic filtration on $Q$ and $w$ is topologically equivalent to the $zU$-adic filtration on $Q(R)$ by the construction \cite[Proposition 4.1.3]{jones-woods-3}. But since the ideal $M$ is zero, $U=V$ and $\overline{z}=z$, so it follows that $w_0=w$ is topologically equivalent to $u$ as required.\qedhere
\end{enumerate}
\end{proof}

\begin{rk}\label{rk: examples satisfying filt hypotheses}
Specific examples of rings satisfying the conditions of this theorem are abundant: see \cite[\S 1.2]{jones-woods-3} and \cite[Examples 4.12--4.13]{jones-woods-1}. For instance, if $\mathbb{F}_pG$ is the Iwasawa algebra of a soluble \emph{uniform} \cite[\S 4.1]{DDMS} pro-$p$ group (a condition only slightly more restrictive than $p$-valuable), and $P$ is a non-maximal prime ideal, then the ring $R = \mathbb{F}_pG/P$ inherits a natural induced filtration satisfying these hypotheses.

We will make use of this example in the proof of Theorem \ref{letterthm: primes in completely solvable Iwasawa algebras}, but the construction itself will play a fundamental role in our proof of Theorem \ref{letterthm: compatible Iwasawa SPSR is prime}, and most importantly in extending standard filtrations to scalar extensions in \S\ref{section: extension of scalars}.
\end{rk}

\subsection{Skew derivations and bounded skew power series rings}\label{subsec: bounded skew power series rings}

\begin{defn}\label{defn: skew derivations}
Let $S$ be a ring.

\begin{enumerate}[label=(\roman*)]
\item A \emph{skew derivation} on $S$ is a pair $(\sigma, \delta)$, where $\sigma$ is an automorphism of $S$ and $\delta$ is a \emph{left $\sigma$-derivation}, i.e.\ an additive map satisfying $\delta(st) = \delta(s)t + \sigma(s)\delta(t)$ for all $s,t\in S$. (Note that some authors allow $\sigma$ to be a more general \emph{endomorphism}, but we will not need this generality.)  
\item The skew derivation $(\sigma, \delta)$ is \emph{commuting} (also known elsewhere as \emph{1-skew}) if $\sigma\delta = \delta\sigma$.
\item Given a skew derivation $(\sigma, \delta)$, we may define the \emph{skew polynomial ring} $S[x;\sigma,\delta]$. As a left $S$-module, this is isomorphic to $S[x]$ in the obvious way, and its ring structure is uniquely determined by the family of relations $xs = \sigma(s)x + \delta(s)$ for all $s\in S$.
\end{enumerate}
\end{defn}

Most of the following definition is well known, though part 4 is rather rare in the literature.

\begin{defn}\label{defn: sigma-invariant and sigma-prime etc}
Let $S$ be a ring and $(\sigma, \delta)$ a skew derivation on $S$.
\begin{enumerate}
\item An ideal $I$ is called \emph{$\sigma$-invariant} if $\sigma(I) \subseteq I$. (If $S$ is noetherian, this implies $\sigma(I) = I$.)
\item $S$ is \emph{$\sigma$-prime} if, for any two nonzero $\sigma$-invariant ideals $I$ and $J$, we have $IJ \neq 0$.
\item An ideal $I$ is $\sigma$-prime if the quotient ring $S/I$ is $\sigma$-prime.
\item $S$ is \emph{$\sigma$-simple} if its only $\sigma$-invariant ideals are $0$ and $S$.
\end{enumerate}
\end{defn}

As in \cite{jones-woods-1}, our skew derivations will not always be well-behaved, and we will need to replace them with related, better-behaved skew derivations. We note the following for later use:

\begin{props}\label{props: Sigma, Delta, crossed product}
Let $S$ be a ring of characteristic $p$, and let $(\sigma, \delta)$ be a commuting skew derivation on $S$.

\begin{enumerate}
\item \cite[Remark 1.2(i)]{jones-woods-1} \cite[Definition 3.3.1]{jones-woods-3} Each $(\sigma^{p^n}, \delta^{p^n})$ is a commuting skew derivation on $S$, and $S[x^{p^n}; \sigma^{p^n}, \delta^{p^n}]$ is a subring of $S[x; \sigma, \delta]$.
\item \cite[Lemma 3.4.3]{jones-woods-3} If there exists $t\in S$ such that $\delta(s) = ts - \sigma(s)t$ for all $s\in S$ and $\sigma(t) = t$, then $\delta^{p^n}(s) = t^{p^n} s - \sigma^{p^n}(s) t^{p^n}$ for all $s\in S$ and all $n\in\mathbb{N}$.
\end{enumerate}
\end{props}

(In characteristic 0, something analogous but slightly more complicated can be done: see \cite[\S 3.4]{jones-woods-3}.)

If there exists $t\in S$ such that $\delta(s)=ts-\sigma(s)t$ for all $s\in S$, as in Property \ref{props: Sigma, Delta, crossed product}.2, we say that $\delta$ is an \emph{inner $\sigma$-derivation}.

Next, recall the following compatibility conditions between skew derivations and ring filtrations, which are necessary for us to define the bounded skew power series ring.

\begin{defn}\label{defn: quasi-compatibility}
Suppose $(S,w)$ is a filtered ring and $(\sigma, \delta)$ is a commuting skew derivation on $S$.
\begin{enumerate}
\item \cite[Definition 1.8]{jones-woods-1} We say that $(\sigma,\delta)$ is \emph{(strongly) compatible} with $w$ if $\sigma$ and $\delta$ are filtered, with $\deg_w(\sigma - \id) > 0$ and $\deg_w(\delta) > 0$.
\item \cite[Definition 2.4]{jones-woods-2} We say that $(\sigma,\delta)$ is \emph{weakly compatible} with $w$ if $\sigma$, $\sigma^{-1}$ and $\delta$ are filtered, with $\deg_w(\sigma) = \deg_w(\sigma^{-1}) = 0$ and $\deg_w(\delta) > 0$.
\item \cite[Definition 3.1.1]{jones-woods-3} We say that $(\sigma,\delta)$ is \emph{quasi-compatible} with $w$ if:
\begin{enumerate}[label=(\roman*)]
\item $\sigma^i \delta^j$ is filtered for all $i\in\mathbb{Z}$ and $j\in\mathbb{N}$, and
\item there exists a common lower bound for $\deg_w(\sigma^i \delta^j)$, i.e.\ there exists $B\in\mathbb{Z}$ such that $\deg_w(\sigma^i \delta^j) \geq B$ for all $i\in\mathbb{Z}$ and all $j\in\mathbb{N}$, and
\item there exists some $N\in\mathbb{N}$ such that $\deg_w(\delta^N) > 0$.
\end{enumerate}
\end{enumerate}
It is easy to see that (strong) compatibility implies weak compatibility, and that weak compatibility implies quasi-compatibility (taking $B = 0$ and $N = 1$). We will not need the notion of weak compatibility again in this paper, and will simply write ``compatible" for ``strongly compatible", but some of our arguments are true in the more general setting.
\end{defn}

\begin{defn}\label{defn: bounded skew power series rings}
\cite[Definition 3.2.7]{jones-woods-3}
Let $(S,w)$ be a complete filtered ring, and let $(\sigma, \delta)$ be a commuting skew derivation on $S$ which is quasi-compatible with $w$. Then we can define the \emph{bounded skew power series ring} (with respect to $w$)
$$S^+[[x;\sigma, \delta]] = \left\{\underset{n\in\mathbb{N}}{\sum}{s_nx^n}:s_n\in S,\text{ there exists }M\in\mathbb{Z}\text{ such that }w(s_n)\geq M\text{ for all }n\in\mathbb{N}\right\}.$$
(We may say that a sequence of coefficients $(s_n)_{n\in\mathbb{N}}$ is ($w$-)\emph{bounded below} as shorthand for saying that there exists $M\in\mathbb{Z}$ such that $w(s_n) \geq M$ for all $n\in\mathbb{N}$.)

If $\{w(s)\}_{s\in S}$ is bounded below, e.g.\ if $w(s) \geq 0$ for all $s\in S$, then this is the set of \textit{all} power series over $S$, and we simply write it as $S[[x;\sigma,\delta]]$, the (unbounded) skew power series ring: see \eqref{eqn: unbounded SPSR}.

By \cite[Theorem 3.2.6]{jones-woods-3}, this is a topological ring with continuous multiplication which contains $S[x;\sigma, \delta]$ as a dense subring. Multiplication in this ring is given explicitly by
\begin{align}\label{eqn: multiplication in bounded SPSRs}
\left( \sum_{i\geq 0} a_i x^i\right)\left( \sum_{j\geq 0} b_jx^j\right) = \sum_{k\geq 0} \left(\sum_{0\leq e\leq k} \sum_{i\geq e} \binom{i}{e} a_i \sigma^e \delta^{i-e}(b_{k-e}) \right)x^k.
\end{align}
\end{defn}

We state carefully the following analogue of Property \ref{props: Sigma, Delta, crossed product}.1 in the skew power series case:

\begin{lem}\label{lem: going up compatibility}
\cite[Proposition 3.3.2]{jones-woods-3}
Suppose $(S, w)$ is a complete, filtered ring of characteristic $p$ and $(\sigma,\delta)$ is a commuting skew derivation on $S$, and fix arbitrary $n\in \mathbb{N}$. Then $(\sigma,\delta)$ is quasi-compatible with $w$ if and only if $(\sigma^{p^n},\delta^{p^n})$ is quasi-compatible with $w$. In this case, $S^+[[x;\sigma,\delta]]$ and $S^+[[y; \sigma^{p^n}, \delta^{p^n}]]$ are both defined, and moreover, inside the ring $S^+[[x; \sigma, \delta]]$, we have $x^{p^n}s = \sigma^{p^n}(s) x^{p^n} + \delta^{p^n}(s)$ for all $s\in S$, and so $S^+[[x; \sigma, \delta]]$ contains the subring $S^+[[x^{p^n}; \sigma^{p^n}, \delta^{p^n}]]$.\qed
\end{lem}

\begin{rk*}
We typically assume the stronger condition that $(\sigma^{p^m},\delta^{p^m})$ is strongly compatible with $w$ for some $m\in\mathbb{N}$, so it follows from this lemma that $(\sigma,\delta)$ is quasi-compatible in this case.
\end{rk*}

\begin{lem}\label{lem: extend automorphisms to SPSR}
Let $(S,w)$ be a complete filtered ring, and let $(\sigma, \delta)$ be a commuting skew derivation on $S$ which is quasi-compatible with $w$. Let $\alpha$ be a filtered automorphism of $S$ such that $\alpha^{-1}$ is also filtered, and assume that $\alpha\sigma = \sigma\alpha$ and $\alpha\delta = \delta\alpha$. Then we may extend $\alpha$ to an automorphism $\widetilde{\alpha}$ of the whole bounded skew power series ring $S^+[[x; \sigma, \delta]]$ as follows:
$$\widetilde{\alpha}: \sum_{n\in\mathbb{N}} s_n x^n \mapsto \sum_{n\in\mathbb{N}} \alpha(s_n) x^n,$$
with inverse $\widetilde{\alpha^{-1}}$ defined in the same way. Further, if $\varepsilon$ is a filtered $\alpha$-derivation of $S$ such that $\varepsilon\sigma = \sigma\varepsilon$ and $\varepsilon\delta = \delta\varepsilon$, then we may extend $\varepsilon$ to an $\widetilde{\alpha}$-derivation $\widetilde{\varepsilon}$ of $S^+[[x; \sigma, \delta]]$ in the same way:
$$\widetilde{\varepsilon}: \sum_{n\in\mathbb{N}} s_n x^n \mapsto \sum_{n\in\mathbb{N}} \varepsilon(s_n) x^n.$$
\end{lem}

\begin{proof}
The maps $\widetilde{\alpha}$, $\widetilde{\alpha^{-1}}$ and $\widetilde{\varepsilon}$ are well-defined: if $(s_n)_{n\in\mathbb{N}}$ is $w$-bounded below then $(\alpha(s_n))_{n\in\mathbb{N}}$ is $w$-bounded below, as $\alpha$ is filtered with respect to $w$, and likewise for $\alpha^{-1}$ and $\varepsilon$. It is clear that $\widetilde{\alpha}$ and $\widetilde{\alpha^{-1}}$ are mutually inverse. It remains to show that $\widetilde{\alpha}$ is a ring homomorphism and $\widetilde{\varepsilon}$ is a $\widetilde{\alpha}$-derivation. Showing that they are additive maps is easy; the only difficulty is to show that $\widetilde{\alpha}(fg) = \widetilde{\alpha}(f)\widetilde{\alpha}(g)$ and $\widetilde{\varepsilon}(fg) = \widetilde{\varepsilon}(f)g + \widetilde{\alpha}(f) \widetilde{\varepsilon}(g)$. But writing $f = \sum_{i\geq 0} a_i x^i$ and $g = \sum_{j\geq 0} b_j x^j$ and applying the multiplication formula \eqref{eqn: multiplication in bounded SPSRs}, this comes down to showing that 
$$\alpha(a_i \sigma^e \delta^{i-e}(b_{k-e})) = \alpha(a_i) \sigma^e \delta^{i-e}(\alpha(b_{k-e}))$$
and
$$\varepsilon(a_i \sigma^e \delta^{i-e}(b_{k-e})) = \varepsilon(a_i) \sigma^e \delta^{i-e}(b_{k-e}) + \alpha(a_i) \sigma^e \delta^{i-e}(\varepsilon(b_{k-e}))$$
for all $i$, $k$ and $e$, which follow trivially from the assumption that $\alpha$ and $\varepsilon$ commute pairwise with $\sigma$ and $\delta$.
\end{proof}

Of course, we could take $(\alpha,\varepsilon)=(\sigma,\delta)$ in this lemma, which we will often during the paper.

\subsection{Skew power series over standard filtered artinian rings}

\textbf{Throughout this subsection}, we take $(Q, u)$ to be a standard filtered artinian ring as in Definition \ref{defn: standard filtrations}, with commuting skew derivation $(\sigma, \delta)$. We also define $(D, v)$ to be the associated complete discrete valuation ring (Definition \ref{defn: DVR}) with valued division ring of fractions $(F, v)$ and associated maximal order $\O = u^{-1}([0,\infty])$. In particular, there exists $r \geq 1$ such that $Q = M_r(F)$, $\O = M_r(D)$ and $u = M_r(v)$.

The following results are now routine, and demonstrate the utility of standard filtrations. For the following lemma, compare the very similar \cite[Definition 2.8, Corollary 2.13(ii)]{jones-woods-2}.

\begin{lem}\label{lem: restricted SPSR over Q}
Suppose $(\sigma, \delta)$ is compatible with $u$.

\begin{enumerate}[label=(\roman*)]
\item The function
\begin{align*}
f_Q: Q^+[[x; \sigma, \delta]] &\to \mathbb{Z}\cup\{\infty\}\\
\sum_{n\in\mathbb{N}} q_n x^n &\mapsto \min_n \{u(q_n) + n\}
\end{align*}
is a ring filtration.
\item Give $\O[[x;\sigma,\delta]]$ the filtration $f_\O := f_Q|_{\O[[x;\sigma,\delta]]}$, and form the tensor product $Q\otimes_\O \O[[x;\sigma,\delta]]$ of $(Q, u)$ and $\O[[x; \sigma, \delta]], f_\O)$ with the tensor product filtration. Then the canonical multiplication map
$$Q\otimes_\O \O[[x;\sigma,\delta]] \to Q^+[[x; \sigma, \delta]]$$
is an isomorphism of filtered rings.\qed
\end{enumerate}
\end{lem}

In the case when $(\sigma, \delta)$ \emph{is} compatible with $u$, we have the following result (compare \cite[Theorem A]{jones-woods-1}), which is the first step towards our desired goal.

\begin{thm}\label{thm: simple artinian with compatible skew derivation}
If $(\sigma, \delta)$ is compatible with $u$, then the skew power series ring $Q^+[[x; \sigma, \delta]]$ is prime.
\end{thm}

\begin{proof}
Since $Q^+[[x; \sigma, \delta]] = Q\otimes_\O \O[[x; \sigma, \delta]]$ by Lemma \ref{lem: restricted SPSR over Q}, it is enough to prove that $\O[[x; \sigma, \delta]]$ is prime. So let $I$ and $J$ be nonzero ideals of $\O[[x; \sigma, \delta]]$. Since $I\cap \O[x; \sigma, \delta]\subseteq I$ and $J\cap \O[x; \sigma, \delta]\subseteq J$, it will suffice to prove that $(I\cap \O[x; \sigma, \delta])(J\cap \O[x; \sigma, \delta])$ is nonzero. But both $I\cap \O[x; \sigma, \delta]$ and $J\cap \O[x; \sigma, \delta]$ are nonzero by \cite[Theorem C]{jones-woods-2}, and $\O[x; \sigma, \delta]$ is a prime ring by \cite[Theorem 1.2.9(iii)]{MR}, so we are done.
\end{proof}

If $(\sigma, \delta)$ is \emph{not} compatible with $u$, we will usually be able to replace $(\sigma, \delta)$ by something better-behaved.

\begin{lem}
Suppose $Q$ has characteristic $p$, and there exists $t\in Q$ such that $\delta(q) = tq - \sigma(q)t$ for all $q\in Q$ and $\sigma(t) = t$. Then the following are equivalent:
\begin{enumerate}[label=(\roman*)]
\item there exists $N\in\mathbb{N}$ such that $(\sigma^{p^N}, \delta^{p^N})$ is compatible with $u$,
\item $(\sigma^{p^m}, \delta^{p^m})$ is compatible with $u$ for all sufficiently large $m\in\mathbb{N}$,
\item there exists $M\in\mathbb{N}$ such that $(\sigma^{p^M} - \id)(\O) \subseteq J(\O)^2$ and $\delta^{p^M}(\O) \subseteq J(\O)^2$.
\end{enumerate}
\end{lem}

\begin{proof}
It is clear that (ii)$\implies$(i). Conversely, assume (i), and suppose $(\sigma^{p^N}, \delta^{p^N})$ is compatible with $u$: then
\begin{align*}
\deg_u(\sigma^{p^{N+1}} - \id) &= \deg_u((\sigma^{p^N} - \id)^p) \geq p \geq 1,\\ 
\deg_u(\delta^{p^{N+1}}) &= \deg_u((\delta^{p^N})^p) \geq p \geq 1,
\end{align*}
and so $(\sigma^{p^{N+1}}, \delta^{p^{N+1}})$ is compatible with $u$: then it follows that $(\sigma^{p^m}, \delta^{p^m})$ is compatible with $u$ for all $m\geq N$ by induction, which shows (ii).

Now assume (i), so that $\deg_u(\sigma^{p^N} - \id) \geq 1$ and $\deg_u(\delta^{p^N}) \geq 1$. In this case, we certainly have 
$$\deg_u(\sigma^{p^{N+1}} - \id) = \deg_u((\sigma^{p^N} - \id)^p) \geq p \geq 2,$$
and likewise $\deg_u(\delta^{p^{N+1}}) \geq 2$, so we have verified (iii) for $M = N + 1$. Conversely, assume (iii), so that $(\sigma^{p^M} - \id)(\O) \subseteq J(\O)^2$ and $\delta^{p^M}(\O) \subseteq J(\O)^2$: it then follows from Lemma \ref{lem: criterion for delta positive degree} that $\deg_u(\sigma^{p^M} - \id) \geq 1$ and $\deg_u(\delta^{p^M}) \geq 1$, so we have established (i) for $N = M$.
\end{proof}

A slightly more general form of this was proved as \cite[Lemma 5.2.2]{jones-woods-3}, but this simpler version will be enough for our current purposes. In practice, it is often more convenient to check $d(\O) \subseteq J(\O)^2$ rather than $d(J(\O)^r) \subseteq J(\O)^{r+1}$ for all $r\in\mathbb{Z}$, but we will use the equivalence of these conditions throughout the paper, often without explicit mention.

The basic relationship between $Q^+[[x; \sigma, \delta]]$ and the subrings $Q^+[[x^{p^m}; \sigma^{p^m}, \delta^{p^m}]]$ is given in the following theorem, which says that -- after localisation -- the larger ring is a \emph{crossed product} over the smaller ring, in the sense of \cite{passmanICP}.

\begin{thm}\label{thm: crossed product after localisation}
\cite[Theorem B]{jones-woods-3}
Suppose $Q$ has characteristic $p$, and $(\sigma, \delta)$ is quasi-compatible with $u$. Suppose further that there exists $t\in Q$ such that $\delta(q) = tq - \sigma(q)t$ for all $q\in Q$ and $\sigma(t) = t$, and that $(\sigma^{p^n}, \delta^{p^n})$ is compatible with $u$ for sufficiently large $n$. Then $x-t$ is 
a regular normal element of $Q^+[[x; \sigma, \delta]]$, and for any $m\in\mathbb{N}$, we have
$$Q^+[[x; \sigma, \delta]]_{(x-t)} = \bigoplus_{i=0}^{p^m-1} Q^+[[x^{p^m}; \sigma^{p^m}, \delta^{p^m}]]_{(x^{p^m} - t^{p^m})} g^i,$$
a crossed product, where $g = x - t$. If $g$ is already a unit in $Q^+[[x; \sigma, \delta]]$, then this gives a crossed product
$$Q^+[[x; \sigma, \delta]] = \bigoplus_{i=0}^{p^m-1} Q^+[[x^{p^m}; \sigma^{p^m}, \delta^{p^m}]] g^i.$$
In either case, we have $gx = xg$ and $gq = \sigma(q)g$ for all $q\in Q$.\qed
\end{thm}

\section{The crossed product decomposition}\label{section: crossed product decomposition}

\textbf{Setup.} Throughout this section, we will adopt a number of hypotheses, which we state upfront for the convenience of the reader.
\begin{itemize}
\item Let $(Q, u)$ be a standard filtered simple artinian ring of characteristic $p$, with centre $Z$. Write $\O$ for its associated maximal order.
\item Take a commuting skew derivation $(\sigma, \delta)$ on $Q$, and assume that there exists some $N\in\mathbb{N}$ such that $(\sigma^{p^n}, \delta^{p^n})$ is compatible with $u$ for all $n\geq N$. Then it follows from Lemma \ref{lem: going up compatibility} that $(\sigma, \delta)$ is quasi-compatible with $u$, and so in particular, we may form the bounded skew power series ring $Q^+[[x; \sigma, \delta]]$. Applying \cite[Corollary 3.5.11]{jones-woods-3} with $S = Q$ and $A = \O$ tells us that $Q^+[[x; \sigma, \delta]]$ is Noetherian.
\item Assume also that there exists $t\in Q$ such that $\delta(q) = tq - \sigma(q)t$ for all $q\in Q$, and $\sigma(t) = t$. Now Theorem \ref{thm: crossed product after localisation} gives us a crossed product decomposition for any $m\in\mathbb{N}$: $$Q^+[[x;\sigma,\delta]]_{(x-t)} = \bigoplus_{i=0}^{p^m-1} Q^+[[x^{p^m};\sigma^{p^m},\delta^{p^m}]] g^i,$$ where $g=x-t$, and the action of $g$ on $Q^+[[x^{p^m};\sigma^{p^m},\delta^{p^m}]]$ is given by $gx = xg$ and $gq = \sigma(q)g$ for all $q\in Q$.
\end{itemize}

Recall the \emph{strong finite order hypothesis} \eqref{SFOH} from \S\ref{subsec: main results}: in this section, we will examine the crossed product decomposition more closely, and use it to prove that $Q^+[[x;\sigma,\delta]]$ is a prime ring in the case where \eqref{SFOH} is satisfied, thus deducing Theorem \ref{letterthm: SFOH implies prime}.

\subsection{The action of \(\sigma\) on the centre}\label{subsec: new galois}

In this section, we explore what happens to the skew power series ring in the case where the restriction of $\sigma$ to the centre $Z$ has finite order $p^s$.

\begin{lem}\label{lem: identical on centre}
Suppose $[\sigma]\in \Out(Q)$ has finite order $d$. Then $\sigma|_Z$ has order $p^s$ for some $s\in\mathbb{N}$.
\end{lem}

\begin{proof}
By our compatibility assumption, we know that for all $d\in\mathbb{N}$ there exists $N$ such that $\deg_u(\sigma^{p^m} - \id) \geq d$ for all $m\geq N$. In particular, for any $q\in Q$, the sequence $(\sigma^{p^\ell}(q))_{\ell\in\mathbb{N}}$ converges to $q$ as $\ell\to\infty$. But since $\sigma^d$ is inner, $\sigma|_Z$ has finite order at most $d$, and so the set $\{\sigma^{p^\ell}(z): \ell\in\mathbb{N}\}$ must contain at most $d$ elements for any $z\in Z$. 

So fix $z\in Z$. Since $z$ is the unique limit point of the set $\{\sigma^{p^\ell}(z): \ell\in\mathbb{N}\}$, it follows that $\sigma^{p^k}(z) = z$ for some $k$ (which may depend on $z$). But since also $\sigma^d(z) = \sigma^{p^k}(z) = z$, we must have $\sigma^e(z) = z$ where $e = \gcd(p^k, d)$. Writing $d=p^rb$ where $p\nmid b$, we must have that $e\mid p^r$, so $\sigma^{p^r}(z)=z$ for all $z\in Z$, and hence $\sigma|_{Z}$ has order $p^s$ for some $s\leq r$.\end{proof}

So \textbf{assume throughout this subsection} that $\sigma|_Z$ has order $p^s$ for some fixed $s\in\mathbb{N}$. We will quickly be able to reduce to the case $s = 0$, i.e.\ the case where $\sigma$ acts trivially on $Z$.

We first recall some basic Galois theory. If $k$ is a field and $G$ is a group of automorphisms of $k$, we will write $k^G = \{x\in k : g(x) = x \text{ for all } g\in G\}$ for the corresponding fixed field. When $G = \langle \tau \rangle$, we will write simply $k^\tau$, and so ``$\sigma$ acts trivially on $Z$" may be written as $Z^\sigma = Z$.

\begin{lem}\label{lem: Galois over invariance}
\cite[Theorem 18.20]{Isa93}
Let $k$ be a field and $G$ a finite group of automorphisms of $k$. Then $k/k^G$ is a finite Galois extension with Galois group $G$.\qed
\end{lem}

\begin{lem}\label{lem: artin-schreier}
Let $k$ be a field of characteristic $p$, and let $K/k$ be a Galois extension of degree $p$. Then there exists $\alpha\in K$ such that $K = k(\alpha)$ and $\alpha$ satisfies an equation $X^p - X - a = 0$ for some $a\in k$.\qed
\end{lem}

Such extensions $K/k$ are called \emph{Artin-Schreier} extensions. It is easy to see that the equation $X^p - X - a = 0$ is irreducible in $k[X]$, $K$ is its splitting field, and the roots are $\{\alpha + i : i\in\mathbb{F}_p\}$.

Now, using Theorem \ref{thm: crossed product after localisation}, $x-t$ is a regular normal element of $Q^+[[x;\sigma,\delta]]$, and every element of the localisation $Q^+[[x; \sigma, \delta]]_{(x-t)}$ can be written as
$$f(x) = (x-t)^{-r} \sum_{n\in\mathbb{N}} q_n x^n \in Q^+[[x; \sigma, \delta]]$$
for some integer $r$ and some $q_n\in Q$ such that $(q_n)_{n\in\mathbb{N}}$ is $u$-bounded below. We will define the \emph{formal derivative} $\frac{\text{d}}{\text{d}x}f(x)$ of this element to be
$$\frac{\text{d}}{\text{d}x}f(x) = -r(x-t)^{-r-1}\sum_{n\in\mathbb{N}} q_n x^n + (x-t)^{-r} \sum_{n\geq 1} n q_n x^{n-1}.$$
It is easy to see that this definition makes sense as an element of $Q^+[[x; \sigma, \delta]]_{(x-t)}$. Indeed, suppose that $\min_{n\in\mathbb{N}}\{u(q_n)\} \geq B$, and since $(\sigma, \delta)$ is assumed quasi-compatible with $u$, choose $A$ such that $\min \{\deg_w(\sigma^i \delta^j) : i\in\mathbb{Z}, j\in\mathbb{N}\} \geq A$. We can rearrange this expression for $\frac{\text{d}}{\text{d}x}f(x)$ to get
$$\frac{\text{d}}{\text{d}x}f(x)= (x-t)^{-r-1}\sum_{n\in \mathbb{N}}\left(n\sigma(q_n)-rq_n - (n+1)\sigma(q_{n+1})t\right)x^n,$$
and we may bound the $u$-value of the coefficients as follows:
\begin{align*}
\min_{n\in\mathbb{N}} \{u(n\sigma(q_n)-rq_n - (n+1)\sigma(q_{n+1})t)\} &\geq \min_{n\in\mathbb{N}} \{\min\{u(n\sigma(q_n)), u(rq_n), u((n+1)\sigma(q_{n+1})t)\}\}\\
&\geq \min_{n\in\mathbb{N}} \{\min\{u(\sigma(q_n)), u(q_n), u(\sigma(q_{n+1})) + u(t)\}\}\\
& \geq \min\{ A+B, B, A+B+u(t)\},
\end{align*}
i.e.\ the coefficients in this sum are still $u$-bounded below.

\begin{propn}\label{propn: derivative invariant}
Assume that $Z^\sigma \neq Z$. Let $I$ be a two-sided ideal in the localised skew power series ring $Q^+[[x;\sigma,\delta]]_{(x-t)}$. Then, for all $f(x)\in I$, the formal derivative $\frac{\text{d}}{\text{d}x}f(x)$ lies in $I$.
\end{propn}

\begin{proof}
By Lemma \ref{lem: Galois over invariance}, $Z/Z^{\sigma}$ is a Galois extension of degree $p^s$, and so it contains the nontrivial Galois subextension $F/Z^\sigma$ of degree $p$, where $F = Z^{\sigma^p}$. Hence, by Lemma \ref{lem: artin-schreier}, there exists an element $\alpha\in F\setminus Z^\sigma$ satisfying $X^p - X - a = 0$ for some $a\in Z^\sigma$ such that $F = Z^\sigma(\alpha)$, and $\sigma$ must permute the roots $\{\alpha+i : i\in\mathbb{F}_p\}$ of this polynomial.

Fix $i\in\mathbb{F}_p^\times$ such that $\sigma(\alpha) = \alpha + i$. In the ring $Q^+[[x; \sigma, \delta]]$ we have $x\alpha = \sigma(\alpha)x + \delta(\alpha)$, which is equal to $\alpha x + i(x-t)$ (using the assumptions on $\delta$, and the fact that $\alpha$ is central in $Q$). From here, it is easy to prove by induction that 
\begin{align}\label{eqn: alpha 1}
x^n \alpha = \alpha x^n + inx^{n-1}(x-t)
\end{align}
and
\begin{align}\label{eqn: alpha 2}
(x-t)^n \alpha = \alpha (x-t)^n + in(x-t)^n.
\end{align} 
Since both $i$ and $x-t$ are units in $Q^+[[x; \sigma, \delta]]_{(x-t)}$, and $i$ is central, we may rephrase \eqref{eqn: alpha 1} in $Q^+[[x; \sigma, \delta]]_{(x-t)}$ as
\begin{align}\label{eqn: alpha 3}
(x^n \alpha - \alpha x^n)i^{-1}(x-t)^{-1} = nx^{n-1}.
\end{align}
But clearly $\alpha$ commutes with every element of $Q$, as $\alpha\in Z$. So take an arbitrary element $g = \sum_{n\in\mathbb{N}} q_n x^n$ of $Q^+[[x; \sigma, \delta]]$ (which we view as a subring of $Q^+[[x; \sigma, \delta]]_{(x-t)}$), with coefficients $q_n\in Q$. We get
\begin{align*}
(g\alpha - \alpha g)i^{-1} (x-t)^{-1} &= \left(\left(\sum q_n x^n\right) \alpha - \alpha \left(\sum q_n x^n\right)\right)i^{-1}(x-t)^{-1}\\
&= \sum q_n x^n \alpha i^{-1}(x-t)^{-1} - \sum q_n \alpha x^n i^{-1}(x-t)^{-1}\\
&= \sum q_n \left((x^n \alpha - \alpha x^n)i^{-1}(x-t)^{-1}\right)\\
&= \sum q_n (nx^{n-1}) = \frac{\text{d}}{\text{d}x}g,
\end{align*}
where the final equality comes from \eqref{eqn: alpha 3}. It follows immediately that if $g\in I\cap Q^+[[x; \sigma, \delta]]$, then we have $(g\alpha - \alpha g)i^{-1} (x-t)^{-1} = \frac{\text{d}}{\text{d}x}g \in I\cap Q^+[[x; \sigma, \delta]]$.

Finally, an arbitrary element $f \in Q^+[[x;\sigma,\delta]]_{(x-t)}$ can be written as $f = (x-t)^{-r} g$, where $g\in Q^+[[x;\sigma,\delta]]$ as above. Write $X = x-t$ for ease of notation, and note that \eqref{eqn: alpha 2} can be rephrased as $\alpha X^n = X^n \alpha - inX^n$. (Inside $Q^+[[x;\sigma,\delta]]_{(x-t)}$, this holds for all $n\in\mathbb{Z}$: indeed, if $n = -m$ is negative, then multiplying the equation $X^m \alpha = \alpha X^m - i(-m)X^m$ on both sides by $X^{-m}$ gives the desired equation.) Then
\begin{align*}
(f\alpha - \alpha f)i^{-1}X^{-1} &= (X^{-r} g\alpha - \alpha X^{-r} g) i^{-1}X^{-1}\\
&= (X^{-r} g\alpha - X^{-r} \alpha g - irX^{-r} g) i^{-1}X^{-1}\\
&= X^{-r}(g\alpha - \alpha g)i^{-1} X^{-1} - rX^{-r-1}g\\
&= X^{-r}\frac{\text{d}}{\text{d}x}g - rX^{-r-1}g = \frac{\text{d}}{\text{d}x}f
\end{align*}
as required. Hence, if $I$ is any two-sided ideal in $Q^+[[x; \sigma, \delta]]_{(x-t)}$ and $f\in I$, then $\frac{\text{d}}{\text{d}x}f\in I$.
\end{proof}

In the following results, we extend $(\sigma, \delta)$ to a skew derivation $(\widetilde{\sigma},\widetilde{\delta})$ on the bounded skew power series ring $Q^+[[x; \sigma, \delta]]$ (see Lemma \ref{lem: extend automorphisms to SPSR}), and hence to $Q^+[[x; \sigma, \delta]]_{(x-t)}$ \cite[Lemma 1.3]{goodearl-skew-poly-and-quantized}.

\begin{propn}\label{propn: properties of extended skew ders}
Let $r\in Q^+[[x; \sigma, \delta]]_{(x-t)}$ be arbitrary. Then:
\begin{enumerate}[label=(\roman*)]
\item $\widetilde{\sigma}(x) = x$ and $\widetilde{\sigma}(x-t) = x-t$.
\item $\widetilde{\delta}(x) = 0$ and $\widetilde{\delta}(x-t) = 0$.
\item $xr = \widetilde{\sigma}(r)x + \widetilde{\delta}(r)$.
\item $\widetilde{\delta}(r) = tr - \widetilde{\sigma}(r)t$.
\end{enumerate}
\end{propn}

\begin{proof}
Parts (i)--(ii) are an easy calculation. For parts (iii)--(iv), write the arbitrary element $r$ explicitly as $r = (x-t)^{-m} \sum_{n\in\mathbb{N}} q_n x^n \in Q^+[[x; \sigma, \delta]]_{(x-t)}$.

\begin{enumerate}[label=(\roman*)]
\setcounter{enumi}{2}
\item 
$\begin{aligned}[t]
xr &= (x-t)^{-m}\sum_{n\in\mathbb{N}} xq_n x^n\\
&= (x-t)^{-m}\sum_{n\in\mathbb{N}} (\sigma(q_n)x + \delta(q_n))x^n\\
&= \widetilde{\sigma}(r)x + \widetilde{\delta}(r).
\end{aligned}$
\item 
$\begin{aligned}[t]
\widetilde{\delta}(r) &= \widetilde{\delta}((x-t)^{-m})\sum_{n\in\mathbb{N}} q_n x^n + \widetilde{\sigma}((x-t)^{-m}) \widetilde{\delta}\left(\sum_{n\in\mathbb{N}} q_n x^n\right)\\
&= (x-t)^{-m}\sum_{n\in\mathbb{N}} \delta(q_n) x^n = (x-t)^{-m}\sum_{n\in\mathbb{N}} (tq_n - \sigma(q_n)t)x^n\\
&= t\left((x-t)^{-m}\sum_{n\in\mathbb{N}} q_n x^n\right) - (x-t)^{-m} \left(\sum_{n\in\mathbb{N}} \sigma(q_n)x^n\right)t\\
&= tr - \widetilde{\sigma}(r)t.
\end{aligned}$
\end{enumerate}
\end{proof}

From now on, we will simply write $\sigma$ and $\delta$ instead of $\widetilde{\sigma}$ and $\widetilde{\delta}$.

\begin{thm}\label{thm: suff identical on centre}
Assume that $Z^\sigma \neq Z$. Then there is a one-to-one correspondence between two-sided ideals $I \lhd Q^+[[x;\sigma,\delta]]_{(x-t)}$ and $\sigma$-invariant two-sided ideals $J\lhd Q^+[[x^p; \sigma^p, \delta^p]]_{(x^p-t^p)}$, given by the two mutually inverse mappings

\centerline{
\xymatrix@R=2px{
{\left\{\begin{array}{c}
\text{two-sided ideals of }\\
Q^+[[x;\sigma,\delta]]_{(x-t)}
\end{array}\right\}}
\ar@/^/@{->}[r]^-\Phi\ar@/_/@{<-}[r]_-\Psi&
{\left\{\begin{array}{c}
\sigma\text{-invariant ideals of }\\
Q^+[[x^p;\sigma^p,\delta^p]]_{(x^p-t^p)}
\end{array}\right\}}
}
}
defined by contraction $\Phi(I) = I\cap Q^+[[x^p;\sigma^p,\delta^p]]$ and extension $\Psi(J) = JQ^+[[x;\sigma,\delta]]$.

Moreover, suppose $\alpha$ is any automorphism of $Q$, commuting with $\sigma$ and $\delta$. Then $\Phi$ and $\Psi$ restrict to mutually inverse mappings

\centerline{
\xymatrix@R=2px{
{\left\{\begin{array}{c}
\alpha \text{-invariant ideals of }\\
Q^+[[x;\sigma,\delta]]_{(x-t)}
\end{array}\right\}}
\ar@/^/@{->}[r]^-\Phi\ar@/_/@{<-}[r]_-\Psi&
{\left\{\begin{array}{c}
(\sigma,\alpha) \text{-invariant ideals of }\\
Q^+[[x^p;\sigma^p,\delta^p]]_{(x^p-t^p)}
\end{array}\right\}}.
}
}
\end{thm}

\begin{proof}
For convenience, write $R = Q^+[[x;\sigma,\delta]]_{(x-t)} \supseteq Q^+[[x^p;\sigma^p,\delta^p]]_{(x^p-t^p)} = S$.

First, we show that $\Phi$ is well-defined. Take a two-sided ideal $I\lhd R$, and note that $I$ is $\sigma$-invariant, since $\sigma(r)=(x-t)^{-1}r(x-t)\in I$ for all $r\in I$ by Proposition \ref{propn: properties of extended skew ders}(iii--iv). So $\Phi(I) = I\cap S$ is a $\sigma$-invariant ideal of $S$. 

Next, we show that $\Psi$ is well-defined. Clearly $\Psi(J) = JR$ is a \emph{right} ideal of $R$. To show it is a left ideal, we first need to show that it is invariant under left multiplication by $x$. So take $s = (x-t)^{-pm} \sum_n q_n x^{pn} \in J$: then, by Proposition \ref{propn: properties of extended skew ders}(iii--iv), we have $xs = \sigma(s)x + ts - \sigma(s)t$. Since $J$ is $\sigma$-invariant, we have $\sigma(s)\in J$, and so $xs\in JR$. Next, we need to show that left multiplication by elements of $R$ preserves boundedness of coefficients: that is, given arbitrary $r\in R$ and $s\in J$, we must be able to express $rs$ as $(x-t)^{-k} \sum_{\ell\in\mathbb{N}} a_\ell x^\ell$, where the $u(a_\ell)$ are bounded below. But this follows trivially from the fact that $R$ is a ring (Theorem \ref{thm: crossed product after localisation}) and $J\subseteq R$.

We must now show that these mappings are mutually inverse. Firstly, since $R$ is a free left $S$-module by Theorem \ref{thm: crossed product after localisation}, it is in particular faithfully flat, and so the map $S/J \to S/J \otimes_S R$ sending $s$ to $s\otimes 1$ is injective \cite[Lemma 7.2.5]{MR}. That is, the natural map $S/J \to R/JR$ is injective, and so its kernel $(JR\cap S)/J$ must equal zero. Hence $J = JR\cap S = \Phi(\Psi(J))$.

In the other direction, take $I\lhd R$, and note that the statement is clear if $I = 0$, so we may assume that $I\neq 0$. Then given any $0\neq f(x)\in I$, left-multiply by a sufficiently large power of $x-t$ to ensure it is in $I\cap Q^+[[x; \sigma, \delta]]$, and write $f(x)=f_0(x^p)+f_1(x^p)x+\dots+f_{r}(x^p)x^{r}$, where $r\leq p-1$, $f_i(x^p)\in S$, and $f_r\neq 0$. If $r = 0$, then clearly $f(x) = f_0(x^p) \in I$, so assume that $r>0$. Note that the formal derivative (with respect to $x$) of $g(x^p)$ is 0 for \emph{any} left power series $g\in S$, so we see that
$$\dfrac{\mathrm d}{\mathrm{d}x} g(x^p) x^i = ig(x^p)x^{i-1}$$
for any $0\leq i\leq p-1$. In particular, the $r$'th formal derivative of $f(x)$ is $r!f_r(x^p)$, an element of $I$ by Proposition \ref{propn: derivative invariant}, and since $r!\not\equiv 0$ (mod $p$) we have that $f_r(x^p) \in I$. An easy inductive argument shows further that $f_{r-1}(x^p), \dots, f_1(x^p), f_0(x^p)\in I$. As we now have $f_i(x^p) \in I\cap S = J$ for all $i$, we must have $f(x) \in JR = \Psi(\Phi(I))$.

The ``$\alpha$-invariant" form of this claim follows immediately after noting that $\alpha(R) = R$ and $\alpha(S) = S$, so if $J\lhd S$ is $\alpha$-invariant, then $\alpha(JR) = \alpha(J)\alpha(R) \subseteq JR$, which completes the proof.
\end{proof}

\begin{cor}\label{cor: Phi and Psi properties}
Assume that $Z^\sigma \neq Z$, and let $\Phi$ and $\Psi$ be as in Theorem \ref{thm: suff identical on centre}. Then $\Phi$ and $\Psi$ respect inclusions and preserve products. That is, given ideals $I,I'\lhd Q^+[[x;\sigma,\delta]]_{(x-t)}$ and $\sigma$-invariant ideals $J, J'\lhd Q^+[[x^p; \sigma^p, \delta^p]]_{(x^p-t^p)}$, we have
$$\Phi(II') = \Phi(I)\Phi(I'), \qquad \Psi(JJ') = \Psi(J)\Psi(J'),$$
$$I\subseteq I' \iff \Phi(I) \subseteq \Phi(I'), \qquad J\subseteq J' \iff \Psi(J) \subseteq \Psi(J').$$
\end{cor}

\begin{proof}
Recalling that $\Psi(J)$ is a \emph{two-sided} ideal, i.e.\ $JR = RJR$, the following standard facts about contractions and extensions are routine to prove, and we omit their details:
$$\Phi(II') \supseteq \Phi(I)\Phi(I'), \qquad \Psi(JJ') = \Psi(J)\Psi(J'),$$
$$I\subseteq I' \implies \Phi(I) \subseteq \Phi(I'), \qquad J\subseteq J' \implies \Psi(J) \subseteq \Psi(J').$$
Since $\Phi$ and $\Psi$ are mutually inverse, we also get:
\begin{align*}
\Phi(I) \subseteq \Phi(I') \implies \Psi(\Phi(I)) \subseteq \Psi(\Phi(I')) \implies I \subseteq I',
\end{align*}
and similarly $\Psi(J) \subseteq \Psi(J') \implies J \subseteq J'$. Finally,
$$II' = \Psi(J) \Psi(J') = \Psi(JJ') = \Psi(\Phi(I) \Phi(I')) \subseteq \Psi(\Phi(II')) = II',$$
so $\Psi(\Phi(I) \Phi(I')) = \Psi(\Phi(II'))$, and hence $\Phi(I) \Phi(I') = \Phi(II')$.
\end{proof}

\begin{cor}\label{cor: when sigma does not fix Z large ring is prime iff small ring is sigma-prime}
Fix any $1\leq \ell\leq s$. Then $Q^+[[x; \sigma, \delta]]_{(x-t)}$ is prime (resp.\ simple) if and only if $Q^+[[x^{p^\ell}; \sigma^{p^\ell}, \delta^{p^\ell}]]_{(x^{p^\ell}-t^{p^\ell})}$ is $\sigma$-prime (resp.\ $\sigma$-simple).
\end{cor}

\begin{proof}
Take some $0\leq i\leq \ell-1$. As $\sigma^{p^i}$ does not act trivially on $Z$, we may apply Theorem \ref{thm: suff identical on centre} to the rings $S_i := Q^+[[x^{p^i}; \sigma^{p^i}, \delta^{p^i}]]_{(x^{p^i}-t^{p^i})}$ and $\displaystyle S_{i+1} := Q^+[[x^{p^{i+1}}; \sigma^{p^{i+1}}, \delta^{p^{i+1}}]]_{(x^{p^{i+1}}-t^{p^{i+1}})}$.  Taking $\alpha = \sigma$ gives a one-to-one correspondence

\centerline{
\xymatrix@R=2px{
\left\{ \sigma \text{-invariant ideals of } S_i \right\}
\ar@/^/@{->}[r]^-\Phi\ar@/_/@{<-}[r]_-\Psi&
\left\{\sigma \text{-invariant ideals of } S_{i+1}\right\},
}
}

and so it is clear that $S_i$ is $\sigma$-simple if and only if $S_{i+1}$ is $\sigma$-simple.

Now suppose that $S_{i+1}$ is $\sigma$-prime, and take two nonzero $\sigma$-invariant ideals $I, I'\lhd S_i$. Then (using Corollary \ref{cor: Phi and Psi properties} implicitly throughout) $II' = \Psi(\Phi(II')) = \Psi(\Phi(I)\Phi(I'))$, and since $\Phi(I)\Phi(I')$ is the product of two nonzero $\sigma$-invariant ideals in $S_{i+1}$, it must be nonzero, so $II'$ is nonzero. It follows that $S_i$ is $\sigma$-prime, as $I, I'$ were arbitrary. Conversely, suppose that $S_i$ is $\sigma$-prime, and take two nonzero $\sigma$-invariant ideals $J, J'\lhd S_{i+1}$: then $\Psi(JJ') = \Psi(J)\Psi(J')$ is the product of two $\sigma$-invariant ideals of $S_i$, so it must be nonzero, and hence $JJ' = \Phi(\Psi(JJ'))$ is also nonzero.

By arguing inductively on $i$, we have shown that $S_0 = Q^+[[x; \sigma, \delta]]_{(x-t)}$ is $\sigma$-prime (resp.\ $\sigma$-simple) if and only if $S_s = Q^+[[x^{p^s}; \sigma^{p^s}, \delta^{p^s}]]_{(x^{p^s} - t^{p^s})}$ is $\sigma$-prime (resp.\ $\sigma$-simple). But since all two-sided ideals of $Q^+[[x; \sigma, \delta]]_{(x-t)}$ are $\sigma$-invariant, this completes the proof.
\end{proof}

When $t$ is invertible, and $u(t) = u(t^{-1}) = 0$, the element $x-t$ is already a unit in $Q^+[[x; \sigma, \delta]]$, so localising at $(x-t)$ does not change the ring. 

\begin{cor}\label{cor: when sigma does not fix Z part 2}
Assume that $u(t)=u(t^{-1})=0$. Fix any $1\leq \ell\leq s$, then $Q^+[[x; \sigma, \delta]]$ is prime (resp.\ simple) if and only if $Q^+[[x^{p^\ell}; \sigma^{p^\ell}, \delta^{p^\ell}]]$ is $\sigma$-prime (resp.\ $\sigma$-simple).\qed
\end{cor}

Note that, if $Q^+[[x^{p^\ell}; \sigma^{p^\ell}, \delta^{p^\ell}]]$ is prime, then it is $\sigma$-prime. Hence in proving that $Q^+[[x; \sigma, \delta]]$ is prime, we may safely replace $\sigma$ with $\sigma^{p^s}$, which will allow us to assume without loss of generality that $\sigma$ acts trivially on $Z$.

\subsection{The case where \((\sigma, \delta)\) is compatible}\label{subsec: using compatibility}

\textbf{Assume throughout this subsection} (in addition to the setup of this section) that $(\sigma,\delta)$ is compatible with $u$, and that $[\sigma]\in \Out(Q)$ has order $d\in\mathbb{N}$. Fix any $b\in Q^{\times}$ such that $\sigma^d(q)=bqb^{-1}$ for all $q\in Q$.

As Theorem \ref{thm: simple artinian with compatible skew derivation} demonstrates, the case where our skew derivation $(\sigma,\delta)$ is compatible with the standard filtration on $Q$ is undoubtedly the easiest case to deal with in practice. The following result uncovers restrictions on the ideal structure of $Q^+[[x; \sigma, \delta]]_{(x - t)}$ in this case, and can be contrasted with \cite[Theorem C]{jones-woods-3}, which demonstrated that $Q^+[[x;\sigma,\delta]]_{(x-t)}$ is simple when $[\sigma]\in \Out(Q)$ has infinite order.

\begin{propn}\label{propn: central polynomial}
Let $I$ be a nonzero, two-sided ideal of $Q^+[[x;\sigma,\delta]]_{(x-t)}$, and let $X:=x-t$. Then there exists a monic polynomial $f(Y) \in Z[Y]$ such that $f(b^{-1}X^d)\in I$.
\end{propn}

\begin{proof}
We know that we can identify $Q^+[[x;\sigma,\delta]] = Q\otimes_{\O}\O[[x;\sigma,\delta]]$ by Lemma \ref{lem: restricted SPSR over Q}(ii), so since $I\neq 0$, it follows that $I\cap Q^+[[x;\sigma,\delta]] \neq 0$ and hence $J:=I\cap\O[[x;\sigma,\delta]]\neq 0$. Therefore, since $(\sigma,\delta)$ is compatible with $u$, we can apply \cite[Theorem C]{jones-woods-2} to see that $J\cap Q[x;\sigma,\delta]\neq 0$, so there exists a polynomial $0\neq g\in J\cap Q[x;\sigma,\delta]\subseteq I\cap Q[x;\sigma,\delta]$. We will assume that $g$ has minimal degree $r \geq 0$ among such polynomials.

Write $g$ as a polynomial in $X$: take $q_0,\dots,q_r\in Q$ with $q_r\neq 0$ such that $g=q_0+q_1X+\dots+q_rX^r$. Note that $q_0\neq 0$: indeed, if $q_0 = 0$, then $r \geq 1$, and $g = g'X$ where $g'$ is a polynomial of smaller degree $r - 1$. But $X$ is invertible in $Q^+[[x; \sigma, \delta]]_{(x-t)}$, so $g' = gX^{-1} \in I$, contradicting our minimality assumption.

Now let $A:=\{q\in Q:c_0+c_1X+\dots+c_{r-1}X^{r-1}+qX^r\in I$ for some $c_i\in Q\}$. Using the fact that $Xq=\sigma(q)X$ for all $q\in Q$, it is easy to show that $A$ is a two-sided ideal in the simple ring $Q$. So since $0\neq q_r\in A$, it follows that $A=Q$, and hence $1\in A$. Therefore, we may assume that $q_r=1$, i.e.\ $g$ is a monic polynomial.

Now, we saw in the proof of Theorem \ref{thm: suff identical on centre} that all ideals of $Q^+[[x;\sigma,\delta]]_{(x-t)}$ are $\sigma$-invariant. In particular, since $\sigma(X)=X$:
$$\sigma(g)=\sigma(q_0)+\sigma(q_1)X+\dots+\sigma(q_{r-1})X^{r-1}+X^r\in I,$$
so $\sigma(g)-g=(\sigma(q_0)-q_0)+(\sigma(q_1)-q_1)X+\dots+(\sigma(q_{r-1})-q_{r-1})X^{r-1}$ is a polynomial of degree less than $r$ in $I$. Hence, by minimality, it must be zero, and hence $\sigma(q_i)=q_i$ for each $0\leq i\leq r-1$.

Also, given $s\in Q$, we know that $X^is=\sigma^i(s)X^i$ for all $i$, so
$$\sigma^r(s)g-gs=(\sigma^r(s)q_0-q_0s)+(\sigma^r(s)q_1-q_1\sigma(s))X+\dots+(\sigma^r(s)q_{r-1}-q_{r-1}\sigma^{r-1}(s))X^{r-1}$$
is another polynomial element of $I$ of degree less than $r$. As above, this means it must be zero, which proves for all $s\in Q$ and all $0\leq i\leq r-1$:
\begin{equation}\label{eqn: tau and g}
\sigma^r(s)q_i=q_i\sigma^i(s).
\end{equation}
Thus each $q_i$ is a normal element in the simple ring $Q$, so it is either zero or a unit. Fix any $i$ so that $q_i \neq 0$. Since $\sigma(q_i)=q_i$, it follows from (\ref{eqn: tau and g}) that $$\sigma^{r-i}(s) = \sigma^{-i}(q_i)s\sigma^{-i}(q_i)^{-1}=q_isq_i^{-1}$$ for all $s\in Q$, and so $\sigma^{r-i}$ is inner. But since $d$ is the order of $[\sigma]\in \Out(Q)$, this means that $d$ divides $r-i$. In particular, we know that $q_0 \neq 0$, so this argument also shows that $d$ divides $r$, and it follows that $d$ divides $i$ whenever $q_i\neq 0$.

Altogether, we have proved that $g$ has the form $$g=q_0+q_dX^d+q_{2d}X^{2d}+\dots+q_{(\ell-1)d}X^{(\ell-1)d}+X^{\ell d}$$ Whenever $q_{jd}\neq 0$, since $q_{jd}sq_{jd}^{-1}=\sigma^{r-jd}(s)=\sigma^{(\ell-j)d}(s)=b^{\ell-j}sb^{-(\ell-j)}$ for all $s\in Q$, it follows that $q_{jd}^{-1}b^{\ell-j}$ is central, i.e.\ $q_{jd}=z_j b^{\ell-j}$ for some $z_j\in Z$.

Altogether, we can rewrite $g\in I$ as $$g=z_0b^{\ell}+z_1b^{\ell-1}X^d+\dots+z_{\ell-1}bX^{(\ell-1)d}+X^{\ell d}.$$
Finally, let $f(Y):=z_0+z_1Y+\dots+z_{\ell-1}Y^{\ell-1}+Y^{\ell}\in Z[Y]$, then $$f(b^{-1}X^d)=z_0+z_1b^{-1}X^d+\dots+z_{\ell-1}b^{-(\ell-1)}X^{(\ell-1)d}+b^{-\ell}X^{\ell d}=b^{-\ell}g\in I,$$ which completes the proof.\end{proof}

In \S \ref{sec: FOH}, this result will form a key step in the proof of Theorem \ref{letterthm: FOH implies prime}.

\begin{lem}\label{lem: a regular}
For each $m\in\mathbb{N}$, fix some $a_m\in Q^{\times}$ such that $\sigma^{p^m d}(q)=a_mqa_m^{-1}$ for all $q\in Q$.
\begin{enumerate}[label=(\roman*)]
\item $\deg_u(\sigma^{p^m d} - \id) \geq p^m$. In particular, for any $R>0$, there exists $N>0$ such that $u(a_mqa_m^{-1}-q)>u(q)+R$ for all $q\in Q$ and all $m\geq N$.
\item For all $n\in\mathbb{Z}$, $m\in\mathbb{N}$, $u(a_m^n)=nu(a_m)$.
\end{enumerate}
\end{lem}

\begin{proof}
$ $

\begin{enumerate}[label=(\roman*)]
\item By the assumption of compatibility, we have $\deg_u(\sigma-\id) \geq 1$, and so (as we are in characteristic $p$) it follows that $\deg_u(\sigma^{p^m} - \id) = \deg_u((\sigma - \id)^{p^m}) \geq p^m$. Hence
$$\sigma^{p^m d} - \id = ((\sigma^{p^m} - \id) + \id)^d - \id = \sum_{i=1}^d \binom{d}{i} (\sigma^{p^m} - \id)^i$$
must also have degree at least $p^m$. So, given $R > 0$, take $N:=\lceil\log_p(R)\rceil$ to deduce the second statement.
\item If $u(a_mqa_m^{-1} - q) > u(q)$ for all $q\in Q$, then $a_mqa_m^{-1} - q\in\O$ for all $q\in\O$, so $a_m\O=\O a_m$. Let $\pi\in D$ be a uniformiser, which we will view as an element of $\O$ under the diagonal embedding: then, for sufficiently large $k \in\mathbb{Z}$, $\pi^ka_m\O$ is a two-sided ideal in $\O$, and thus we may choose $k$ such that $\pi^ka_m\O = J(\O)^m$, so that $a_m\O=J(\O)^{m-k}$. 

In particular, if $u(q)=r$ then $q\in J(\O)^r\backslash J(\O)^{r+1}$, and thus $qa_m\in J(\O)^{m+r-k}$, so $u(qa_m) \geq m+r-k$. But if $u(qa_m) \geq m+r-k+1$, i.e.\ $qa_m\in J(\O)^{m+r-k+1}$, then $qJ(\O)^{m-k}\subseteq J(\O)^{m+r-k+1}$, and so $q\pi^{m-k}\in J(\O)^{m+r-k+1}$, from which we would be able to conclude that $q\in J(\O)^{r+1}$ -- contradiction.

Thus $u(qa_m)=r+m-k=u(q)+u(a_m)$ for all $q\in Q$. In particular, $u(a_m^n)=nu(a_m)$ for all $n\in\mathbb{N}$. Also, $0 = u(1) = u(a_m a_m^{-1}) = u(a_m) + u(a_m^{-1})$, so $u(a_m^{-1}) = -u(a_m)$, and thus $u(a_m^n) = nu(a_m)$ for all $n\in\mathbb{Z}$.\qedhere
\end{enumerate}
\end{proof}

Using this lemma, we can now prove the following corollary to Proposition \ref{propn: central polynomial} that will be essential for completing the proof of Theorem \ref{letterthm: FOH implies prime} in the case where the centre of $Q$ is concentrated in degree 0.

\begin{cor}\label{cor: algebraic with small centre}
Fix an element $b\in Q^\times$ such that $\sigma^d(b) = bqb^{-1}$. Suppose further that $J(\O)\cap Z = \{0\}$ and $Q^+[[x;\sigma,\delta]]_{(x-t)}$ is not a simple ring. Then:
\begin{enumerate}[label=(\roman*)]
\item $u(b) = 0$ if and only if $b\in \O^\times$.
\item If $z\in Z^\times$ then $u(z) = 0$.
\item If $t\in \O$ then $b\in \O$.
\item If $t\in\O^{\times}$ then $b\in\O^{\times}$.
\end{enumerate}
\end{cor}

\begin{proof}
$ $

\begin{enumerate}[label=(\roman*)]
\item We know using Lemma \ref{lem: a regular}(ii) that $u(b^n)=nu(b)$ for all $n\in\mathbb{Z}$. Therefore, if $u(b)=0$ then $u(b^{-1})=0$, so $b\in\O^{\times}$.
\item As usual, we write $(D,v)$ for the complete discrete valuation ring associated to $Q$, and $(F,v)$ for its associated valued division ring of fractions. Then as $Q$ is a full matrix ring over $F$, we can identify $Z$ with $Z(F)$. Having made this identification, note that $u|_Z = v|_Z$, and $v$ is multiplicative, in the sense that $v(qq') = v(q) + v(q')$ for all $q, q'\in F$. Hence, for any $z\in Z^\times$, we have $u(z^{-1}) = -u(z)$. So if $u(z) \neq 0$, we have either $u(z) > 0$ or $u(z^{-1}) > 0$, contradicting the assumption that $J(\O) \cap Z = \{0\}$.
\end{enumerate}
Now since $Q^+[[x;\sigma,\delta]]_{(x-t)}$ is not simple, we can fix a proper, nonzero two-sided ideal $I$ of $Q^+[[x;\sigma,\delta]]_{(x-t)}$. By Proposition \ref{propn: central polynomial}, there exists a monic polynomial $0\neq f(Y)\in Z[Y]$ such that $f(b^{-1}X^d)\in I$ (where again we set $X=x-t$), and we will assume that $f$ has minimal degree among such polynomials. 

Write $f(X)=z_0+z_1X+\dots+z_{\ell-1}X^{\ell-1}+X^{\ell}$, and by minimality we can assume that $z_0\neq 0$, as in the proof of Proposition \ref{propn: central polynomial}. By (ii), we can now assume that either $z_i=0$ or $u(z_i)=0$ for each $0\leq i\leq \ell$, and so in particular, $z_i\in\O$ for all $0\leq i\leq \ell$.

\begin{enumerate}[label=(\roman*), resume]
\item By the assumption, we have $u(t) \geq 0$, and let us suppose for contradiction that $u(b)<0$. Then $u(b^{-1}) = -u(b) > 0$ by another application of Lemma \ref{lem: a regular}(ii), i.e.\ $b^{-1}\in J(\O)$. Expanding the element $f(b^{-1}X^d)$ as a polynomial in $x$:
\begin{align*}
f(b^{-1}X^d) &= z_0 + z_1 b^{-1} X^d + \dots + z_{\ell-1} b^{-(\ell-1)} X^{(\ell-1)d} + b^{-\ell} X^{\ell d}\\
&= \alpha_0 + \alpha_1 x + \dots + \alpha_{\ell d} x^{\ell d}.
\end{align*}
As $b^{-1}$, $t$ and all the $z_i$ are in $\O$, it is easy to see that $\alpha_0, \dots, \alpha_{\ell d} \in \O$. We can now calculate directly that the constant term is
$$\alpha_0 = z_0 - z_1b^{-1}t^d + z_2b^{-2}t^{2d} + \dots + (-1)^{\ell-1}z_{\ell-1}b^{-(\ell-1)}t^{(\ell-1)d} + (-1)^{\ell}b^{-\ell}t^{\ell d}.$$
Moreover, $u(z_ib^{-i}t^{i d})\geq u(t^{id})+u(z_i)+u(b^{-i})\geq iu(b^{-1})>0$ for each $i$, and $u(z_0)=0$. So $\alpha_0+J(\O)=z_0+J(\O)$, and it follows that $\alpha_0\in\O^{\times}$. Hence we can write $f(b^{-1}X^d) = \alpha_0 + hx$ for some $\alpha_0 \in \O^{\times}$ and $h\in\O[[x;\sigma,\delta]]$, and so this is a unit in $Q^+[[x;\sigma,\delta]]$. But $f(b^{-1}X^d)\in I$, which contradicts our assumption that $I$ is a proper ideal.
\item We are assuming now that $t\in\O^{\times}$, and so $u(b) \geq 0$ by (iii). Suppose for contradiction that $u(b)>0$. Multiply $f(b^{-1}X^d)$ by $b^{\ell}$ to get $$b^{\ell}f(b^{-1}X^d) = z_0b^{\ell} + z_1b^{\ell-1} X^d + \dots + z_{\ell-1}bX^{(\ell-1)d} + X^{\ell d}.$$ Expanding this as a polynomial in $x$ as above, we see that all coefficients lie in $\O$, that the constant term $\beta$ is 
$$t^{\ell d}+z_{\ell-1}bt^{(\ell-1)d}+\dots+z_1b^{\ell-1}t^d+z_0b^{\ell},$$
and that $u(z_i b^{\ell-i} t^{id}) > 0$ for each $0\leq i\leq \ell-1$. So $\beta+J(\O)=t^{\ell d}+J(\O)$, and it follows that $\beta\in\O^\times$, so we can apply the same argument as above to see that $b^{\ell}f(b^{-1}X^d)\in I$ is a unit in $Q^+[[x;\sigma,\delta]]$, contradicting our assumption again. So we conclude that $u(b)=0$, and hence $u(b^{-1})=-u(b)=0$, so $b\in\O^{\times}$.\qedhere
\end{enumerate}
\end{proof}

\subsection{Ideals in crossed products: minimal elements}

In general, our skew derivation $(\sigma, \delta)$ will \emph{not} be compatible with the standard filtration $u$ on $Q$ as in \S \ref{subsec: using compatibility}. However, we are assuming that some $(\sigma^{p^n}, \delta^{p^n})$ is compatible with $u$, and we will be able to form a crossed product as in Theorem \ref{thm: crossed product after localisation}. We want to exploit this crossed product decomposition to propagate results from a ``compatible" subring to the full skew power series ring.

We begin by taking a step back, and examining the ring structure of more general crossed products.

Let $A$ be a ring with automorphism $\sigma$, and suppose $R$ is a crossed product \cite{passmanICP} of $A$ with $\mathbb{Z}/n\mathbb{Z}$, which we will consider having the form $R = \bigoplus_{i=0}^{n-1} Ag^i,$ where $g\in R$ satisfies $ga = \sigma(a)g$ for all $a\in A$.

\begin{enumerate}
\item For notational convenience, elements of $R$ will be denoted by bold letters, and their entries will be denoted by the corresponding italic letters, so that $\mathbf{a} = a_0 + a_1g + \dots + a_{n-1}g^{n-1}$, and so on.
\item The \emph{support} of $\mathbf{a}\in R$ is the set $\supp(\mathbf{a}) = \{g^i : a_i \neq 0, 0\leq i\leq n-1\}$.
\end{enumerate}

\begin{defn}\label{defn: minimal elements}
Let $I$ be a nonzero ideal of $R$. A nonzero element $\mathbf{a}\in I$ is \emph{minimal for $I$} if $\mathbf{a}$ has smallest possible support among the elements of $I\setminus \{0\}$, i.e.\ $|\supp(\mathbf{a})| \leq |\supp(\mathbf{b})|$ for all $\mathbf{b}\in I\setminus \{0\}$.
\end{defn}

Minimality allows us to deduce near-uniqueness properties:

\begin{propn}\label{propn: minimal elements and centrality properties}
Let $\mathbf{a}$ be minimal for $I\lhd R$, and fix some $0\leq j\leq n-1$ such that $a_j \neq 0$.
\begin{enumerate}[label=(\roman*)]
\item If $\sigma(a_j) = a_j$, then $\sigma(a_i) = a_i$ for \emph{all} $0\leq i\leq n-1$.
\item Let $q\in A$ be arbitrary. If $qa_j = a_j\sigma^j(q)$, then $qa_i = a_i\sigma^i(q)$ for \emph{all} $0\leq i\leq n-1$.
\item If $\mathbf{b}$ is also minimal for $I$ with $\supp(\mathbf{a}) = \supp(\mathbf{b})$, and $b_j = qa_j$ for some $q\in A$, then $\mathbf{b} = q\mathbf{a}$.
\end{enumerate}
\end{propn}

\begin{proof}
In each of the three cases, set
\begin{enumerate}[label=(\roman*)]
\item 
\hfill$\begin{aligned}[t]
\mathbf{c} := g\mathbf{a}g^{-1} - \mathbf{a} 
= \sum_{i=0}^{n-1}(\sigma(a_i) - a_i)g^i,
\end{aligned}$\hfill\null
\item
\hfill$\begin{aligned}[t]
\mathbf{c} := q\mathbf{a} - \mathbf{a}q &= \sum_{i=0}^{n-1} (qa_i - a_i\sigma^i(q))g^i,
\end{aligned}$\hfill\null
\item 
\hfill$\begin{aligned}[t]
\mathbf{c} := \mathbf{b} - q\mathbf{a} &= \sum_{i=0}^{n-1}(b_i - qa_i)g^i.
\end{aligned}$\hfill\null
\end{enumerate}

In each case, it can be checked that $\mathbf{c}$ is also an element of $I$, that $\supp(\mathbf{c})\subseteq \supp(\mathbf{a})$, and that $\mathbf{c}$ has been chosen so that $c_j = 0$. So, by the minimality of $\mathbf{a}$, we must have $\mathbf{c} = 0$.
\end{proof}

\begin{cor}\label{cor: minimal and central}
Let $\mathbf{a}$ be minimal for $I$, and suppose that $a_0 \neq 0$. Then the following are equivalent:
\begin{enumerate}[label=(\roman*)]
\item $\mathbf{a}\in Z(R)$,
\item $a_0\in Z(A)^\sigma$,
\item for all $0\leq i,j\leq n-1$ and all $q\in A$, we have (a) $a_i\in A^\sigma$, (b) $qa_i = a_i\sigma^i(q)$, and (c) $a_ja_i = a_ia_j$.
\end{enumerate}
\end{cor}

\begin{proof}
An element $\mathbf{a}\in R$ is central if and only if (i) $\mathbf{a}$ centralises $g$ and (ii) $\mathbf{a}$ centralises all $q\in A$. Now apply Proposition \ref{propn: minimal elements and centrality properties}(i--ii).
\end{proof}

\subsection{Ideals in crossed products: positive characteristic}\label{subsec: minimal elements and positive characteristic}

We retain the assumptions of the previous subsection, so that $g^n \in A$ and $ga = \sigma(a)g$. From now on, we assume additionally that $n = p^m$, so that $R = A*(\mathbb{Z}/p^m\mathbb{Z})$, and that $\chr(A) = p$.

\begin{rk}\label{rk: if a and b commute, can raise to pth powers separately}
If $a, b\in A$ commute with each other, then $(a+b)^p = a^p+b^p$. If $a$ is $\sigma$-invariant, then $a$ commutes with $g$. In particular, if $\mathbf{a} = \sum_{i=0}^{p^m - 1} a_i g^i\in R$, and all of its coefficients $a_i$ are $\sigma$-invariant and commute pairwise, then $\mathbf{a}^p = \sum_{i=0}^{p^m - 1} a_i^p g^{ip}$. We will use this throughout.
\end{rk}

\begin{lem}\label{lem: raising to power p decreases supp}
Let $\mathbf{a}\in R$ be arbitrary such that its coefficients $a_i$ are $\sigma$-invariant and commute pairwise. Then $|\supp(\mathbf{a}^p)| \leq |\supp(\mathbf{a})|$.
\end{lem}

\begin{proof}
Write $\mathbf{a} = a_{n_1} g^{n_1} + \dots + a_{n_\ell} g^{n_\ell}$ for nonzero coefficients $a_{n_i}\in A$ and pairwise distinct $n_i \leq p^m - 1$, so that $|\supp(\mathbf{a})| = \ell$. Then $\mathbf{a}^p = a_{n_1}^p g^{pn_1} + \dots + a_{n_\ell}^p g^{pn_\ell}$, and so $|\supp(\mathbf{a}^p)| \leq \ell$.
\end{proof}

\begin{lem}\label{lem: minimal elements are p-nilpotent}
Let $I$ be an ideal of $R$ such that $I\cap A = \{0\}$, and that we are given some $\mathbf{a} \in I$ whose coefficients are $\sigma$-invariant and commute pairwise. Then $\mathbf{a}^{p^m} = 0$.
\end{lem}

\begin{proof}
Evaluate directly:
\begin{align*}
\mathbf{a}^{p^m} = \left(\sum_{i=0}^{p^m-1} a_ig^i\right)^{\!\!p^m} &= \sum_{i=0}^{p^m-1} a_i^{p^m}g^{p^mi},
\end{align*}
and notice that $g^{p^m}\in A$, so that this is an element of $I\cap A = \{0\}$.
\end{proof}

\begin{rk}\label{rk: p-nilpotent minimal elements}
We have not yet addressed the question of whether, given a nonzero ideal $I\lhd R$ satisfying $I\cap A = \{0\}$, there exists an element $\mathbf{a}\in I\cap Z(R)$ which is minimal for it. However, if we have such an element, then Corollary \ref{cor: minimal and central} tells us that all its coefficients $a_i$ are $\sigma$-invariant and commute pairwise, and so we may apply Lemma \ref{lem: raising to power p decreases supp}. It follows that either $|\supp(\mathbf{a}^p)| = |\supp(\mathbf{a})|$ or $\mathbf{a}^p = 0$. In the former case, $\mathbf{a}^p\in Z(R)$ must also be minimal for $I$: but, by Lemma \ref{lem: minimal elements are p-nilpotent}, there exists some minimal $i$ such that $\mathbf{a}^{p^i} = 0$, and so replacing $\mathbf{a}$ by $\mathbf{a}^{p^{i-1}}$, we can assume that $\mathbf{a}^p = 0$.
\end{rk}

We make one further simplification.

\begin{lem}\label{lem: can pass to a small subgroup}
Suppose $\mathbf{a}\in Z(R)$, with $\sigma$-invariant and pairwise commuting coefficients. Set $\mathbf{b} = a_0 + a_{p^{m-1}} g^{p^{m-1}} + \dots + a_{(p-1)p^{m-1}} g^{(p-1)p^{m-1}}$: that is, $b_i = a_i$ precisely when $i$ is divisible by $p^{m-1}$, and otherwise $b_i = 0$. If $\mathbf{a}^p = 0$, then $\mathbf{b}^p = 0$.
\end{lem}

\begin{proof}
$\mathbf{b}^p \in A$ is the $g^0$-coefficient of $\mathbf{a}^p$.
\end{proof}

Now consider the following property that the data $(A,\sigma)$ might satisfy:
\begin{equation}
    \tag{P1}
    \text{\parbox{.78\textwidth}{Every nonzero, $\sigma$-invariant ideal of $A$ contains a nonzero, $\sigma$-invariant central element.}}
    \label{P1}
\end{equation}

\begin{propn}\label{propn: there exists a central minimal element whose pth power is zero}
Let $I$ be a nonzero ideal of $R$, and suppose that $(A,\sigma)$ satisfies \eqref{P1}. Then there exists a nonzero central element $\mathbf{a}\in I$ which is minimal for $I$. Moreover, if $I\cap A = \{0\}$, then we may choose this element $\mathbf{a}$ so that $\mathbf{a}^p = 0$, $a_0\neq 0$, and $a_i = 0$ whenever $i$ is not divisible by $p^{m-1}$.
\end{propn}

\begin{proof}
Let $\mathbf{c}$ be minimal for $I$, so that some $c_i \neq 0$. Multiplying by $g$ as many times as necessary, we may ensure without loss of generality that $c_0 \neq 0$. Now define
$$C_0 = \{a\in A : \text{there exists } \mathbf{b}\in I \text{ such that } \supp(\mathbf{b}) \subseteq \supp(\mathbf{c})\text{ and } b_0 = a\}.$$
Now $C_0$ is a nonzero ideal of $A$, and moreover if $a\in C_0$ then choose $\mathbf{b}\in I \text{ such that } \supp(\mathbf{b}) \subseteq \supp(\mathbf{c})$ and $b_0 = a$. Then $g\mathbf{b}g^{-1}\in I$ has the same support as $\mathbf{b}$ and the $0$'th coefficient is $\sigma(a)$. Thus $\sigma(C_0)\subseteq C_0$, and by a similar argument $\sigma^{-1}(C_0)\subseteq C_0$, so we deduce that $C_0$ is $\sigma$-invariant.

Therefore, applying property (P1) we see that $C_0$ contains a nonzero $\sigma$-invariant central element $z$. Let $\mathbf{a}\in I$ be the corresponding element such that $a_0 = z$: then, by the choice of $\mathbf{c}$, we must have $|\supp(\mathbf{c})| = |\supp(\mathbf{a})|$, so that $\mathbf{a}$ is also minimal for $I$; and it follows from Corollary \ref{cor: minimal and central} that $\mathbf{a}$ is central, and that its coefficients $a_i$ are $\sigma$-invariant and commute pairwise. The remaining claims follow from Lemma \ref{lem: minimal elements are p-nilpotent}, Remark \ref{rk: p-nilpotent minimal elements} and Lemma \ref{lem: can pass to a small subgroup}.
\end{proof}

\subsection{Using the nilradical}\label{subsec: nilradical}

We make the same assumptions as in \S \ref{subsec: minimal elements and positive characteristic}, but we now assume further that $A$ is prime.

\begin{lem}\label{lem: prime nilradical}
$R=A\ast(\mathbb{Z}/p^m\mathbb{Z})$ has a unique minimal prime ideal $P$, and $P\cap A=\{0\}$.
\end{lem}

\begin{proof}
$R$ has a unique minimal prime ideal $P$ which is nilpotent by \cite[Proposition 16.4]{passmanICP}. The ideal $P\cap A$ is a nilpotent ideal of a prime ring, and so must be zero.\end{proof}

\begin{rk}\label{rk: can take m = 1}
Assume that $R$ is not prime, or equivalently that $P$ is nonzero. If we assume further that $A$ satisfies (\ref{P1}), then by Proposition \ref{propn: there exists a central minimal element whose pth power is zero}, we can find a nonzero central element $\mathbf{a}\in P$ which is minimal for $P$, which satisfies $a_0\neq 0$ and $\mathbf{a}^p = 0$, and which takes the form
$$\mathbf{a} = a_0 + a_{p^{m-1}} g^{p^{m-1}} + \dots + a_{(p-1)p^{m-1}} g^{(p-1)p^{m-1}}.$$
\end{rk}

We now begin to reintroduce some of our setup.

\begin{thm}\label{thm: conditions for primality}

Suppose that $A$ is a prime ring with $\chr(A)=p$ and an automorphism $\sigma$. Suppose further that $R = \bigoplus_{i=0}^{p^m-1} Ag^i$ is a crossed product of $A$ with $\mathbb{Z}/p^m\mathbb{Z}$ such that $ga = \sigma(a)g$ for all $a\in A$, and that the following conditions are satisfied:

\begin{itemize}
\item $(A,\sigma)$ satisfies (\ref{P1}).

\item $A=Q^+[[X]]$ or $Q^+[[X]][X^{-1}]$ for some standard filtered simple artinian ring $Q$, and $X\in Z(A)^{\sigma}$.

\item $\sigma$ restricts to an automorphism of $Q$ and $\sigma^{p^m}$ is an inner automorphism of $Q$, say there exists $\gamma\in Q^\times$ such that $\sigma(\gamma) = \gamma$ and $\sigma^{p^m}(q) = \gamma^{-1}q\gamma$ for all $q\in Q$.

\item $g^{p^m}\in A$ is a polynomial of degree 1 in $X$.
\end{itemize}

Then $R$ is a prime ring.

\end{thm}

\begin{proof} We will assume for contradiction that $R$ is not prime, so the unique minimal prime ideal $P$ of $R$ given by Lemma \ref{lem: prime nilradical} is nonzero. So by Proposition \ref{propn: there exists a central minimal element whose pth power is zero}, we can find a nonzero central element $\mathbf{a} = a_0+a_{p^{m-1}}g^{p^{m-1}}+\dots+a_{p^{m-1}(p-1)}g^{p^{m-1}(p-1)} \in P$ which is minimal for $P$, such that $a_0\neq 0$ and $\mathbf{a}^p = 0$. To simplify notation, throughout the rest of this proof, we will write $h:=g^{p^{m-1}}$, and $b_i:=a_{ip^{m-1}}$, so that $\mathbf{a} = b_0 + b_1 h + \dots + b_{p-1} h^{p-1}$.
Possibly after multiplying $\mathbf{a}$ by a high power of $X$ (which does not change minimality), we can write each $b_i$ as an element of $Q^+[[X]]$: $b_i = b_{i,0} + b_{i,1} X + b_{i,2} X^2 + \dots$ for some elements $b_{i,k}\in Q$.

Let $0\leq i,j \leq p-1$ and $k, \ell \geq 0$ be arbitrary. Using the fact that $\mathbf{a}$ is central, by Corollary \ref{cor: minimal and central}, each $b_i$ is $\sigma$-invariant, and so conjugating the above expression for $b_i$ by $g$ shows that 
\begin{align}\label{eqn: b_ik are sigma-invariant}
\text{$b_{i,k}$ is $\sigma$-invariant for all $i,k$.}
\end{align}
Also, Proposition \ref{propn: minimal elements and centrality properties}(ii) implies that $qb_i = qa_{ip^{m-1}}=a_{ip^{m-1}}\sigma^{ip^{m-1}}(q) = b_i\sigma^{ip^{m-1}}(q)$
for all $q\in A$. So in particular it follows that $b_{j,\ell} b_i = b_i b_{j,\ell}$. Comparing the coefficients of $X^k$ in the expressions $qb_i=b_i\sigma^{ip^{m-1}}(q)$ and $b_{j,\ell} b_i = b_i b_{j,\ell}$, we also have that
\begin{align}\label{eqn: the b_ij commute pairwise}
b_{j,\ell} b_{i,k} = b_{i,k} b_{j,\ell} \text{ for all } i, j, k, \ell,
\end{align}
and
\begin{align}\label{eqn: the b_ij act like powers of sigma}
qb_{i,k}=b_{i,k}\sigma^{ip^{m-1}}(q) \text{ for all } q\in A \text{ and all } i, k.
\end{align}
Note that \eqref{eqn: the b_ij act like powers of sigma} implies that the $b_{i,k}$ are \emph{normal} elements of $Q$, and hence each one must be either zero or a unit.

We also know that $\mathbf{a}\notin A$: indeed, as $A$ is a prime ring, it contains no central nilpotent elements other than $0$. Hence there exists $i>0$ such that $b_i\neq 0$, and hence $b_{i,k}\neq 0$ for some $k\geq 0$. This $b_{i,k}$ is a unit, and now \eqref{eqn: the b_ij act like powers of sigma} implies that $\sigma^{ip^{m-1}}$ is conjugation by $b_{i,k}$. It follows that $\sigma^{ip^{m-1}}$ is an inner automorphism of $Q$. But since $i<p$ and $\sigma^{p^m}$ is an inner automorphism of $Q$, it follows that $\sigma^{p^{m-1}}$ must also be inner as an automorphism of $Q$.

Now, for all $t\geq 0$, set $\mathbf{z}_t = b_{0,t} + b_{1,t} h + \dots + b_{p-1,t} h^{p-1}$, so that $\mathbf{a} = \mathbf{z}_0 + \mathbf{z}_1 X + \mathbf{z}_2 X^2 + \dots$. Then \eqref{eqn: b_ik are sigma-invariant} implies that $\mathbf{z}_t$ commutes with $g$ for all $t$, and \eqref{eqn: the b_ij act like powers of sigma} shows that $\mathbf{z}_t$ commutes with each $q\in A$. Therefore, $\mathbf{z}_t\in Z(R)$.

Since each $\mathbf{z}_t$ commutes with $X$, we also have that $\mathbf{a}^p = \mathbf{z}_0^p + \mathbf{z}_1^p X^p + \mathbf{z}_2^p X^{2p} + \dots = 0$.
But by \eqref{eqn: b_ik are sigma-invariant} and \eqref{eqn: the b_ij commute pairwise}, we get that $\mathbf{z}_t^p=b_{0,t}^p + b_{1,t}^p h^p + \dots + b_{p-1,t}^p h^{p(p-1)}$, so since $h^p=g^{p^m}$ has degree 1 as a polynomial in $X$, we can write $\mathbf{z}_t^p=c_{0,t} + c_{1,t} X + \dots + c_{p-1,t} X^{p-1}$, and hence
\begin{align*}
\mathbf{a}^p &= c_{0,0} + c_{1,0} X + \dots + c_{p-1,0} X^{p-1}&\\
&\phantom{mm} + c_{0,1} X^p + c_{1,1} X^{p+1} + \dots + c_{p-1,1} X^{2p-1}\\
&\phantom{mm} + c_{0,2} X^{2p} + + c_{1,2} X^{2p+1} + \dots + c_{p-1,2} X^{3p-1} + \dots = 0.
\end{align*}
It clearly follows that $\mathbf{z}_t^p = 0$ for all $t$.

In particular, as $\mathbf{a}\neq 0$, there is some $\mathbf{z} = \mathbf{z}_t$ such that $\mathbf{z}\neq 0$ but $\mathbf{z}^p = 0$. But $P$ is the nilradical of $R$, so since $\mathbf{z}$ is central, we must have that $\mathbf{z}\in P$. But $\varnothing \neq \supp(\mathbf{z}) \subseteq \supp(\mathbf{a})$, so by minimality of $\mathbf{a}$, we must have that $\mathbf{z}$ is minimal for $P$, and $z_0\neq 0$. 

Therefore, replacing $\mathbf{a}$ by $\mathbf{z}$, we may assume that each $b_i=a_{ip^{m-1}}$ lies in $Q$.

Now, since $\sigma^{p^{m-1}}$ is inner as an automorphism of $Q$, choose $\gamma\in Q^\times$ such that $\sigma^{p^{m-1}}(q) = \gamma^{-1} q\gamma$ for all $q\in Q$, and let $\ell:=\gamma h$. Then $\ell$ centralises $A$, so by changing basis, we see that the subring $R'=A\ast\langle h\rangle$ of $R$ is a twisted group ring $A^t[\mathbb{Z}/p\mathbb{Z}] = \bigoplus_{i=0}^{p-1} A\ell^i$.

In particular, if we write $\mathbf{a} = b_0' + b_1' \ell + \dots + b'_{p-1} \ell^{p-1}$, we have $b'_i = b_i \gamma^{-i}$ for all $i$. Defined this way, it follows that $b'_i\in K := Z(Q)^\sigma$, a field. Since $\ell^p \in K[X]$, the set $\bigoplus_{i=0}^{p-1} K[X]\ell^i$ is a subring of $A^t[\mathbb{Z}/p\mathbb{Z}]$ containing $\mathbf{a}$, and is in fact also a twisted group ring, so we will denote it $K[X]^t[\mathbb{Z}/p\mathbb{Z}]$.

On the other hand, let $K(X)$ be the field of fractions of $K[X]$, so that $K[X]^\times \subseteq K(X)^\times$: then \cite[Lemma 1.1]{passmanICP} implies that $K[X]^t[\mathbb{Z}/p\mathbb{Z}]$ extends uniquely to a twisted group ring $K(X)^t[\mathbb{Z}/p\mathbb{Z}]$ (compare \cite[proof of Proposition 12.4(i)]{passmanICP}). Now $\mathbf{a}$ is a central nilpotent element of $K(X)^t[\mathbb{Z}/p\mathbb{Z}]$, and so must lie in its Jacobson radical; and as $\mathbf{a} \neq 0$, we must have that $K(X)^t[\mathbb{Z}/p\mathbb{Z}]$ is not semisimple.

But this contradicts \cite[Proposition 1]{AljRob94}, as $\ell^p=\gamma^ph^p=\gamma^pg^{p^m}$ is a polynomial of degree $1$ in $X$, and so cannot be a $p$th power in $K(X)$.
\end{proof}

\subsection{Proof of Theorem \ref{letterthm: SFOH implies prime}}\label{subsec: minimal elements and Iwasawa-type rings}

\textbf{We now reassume our setup} as given at the start of \S \ref{section: crossed product decomposition}. Fix some $N\in\mathbb{N}$ such that $(\sigma^{p^N}, \delta^{p^N})$ is compatible with $u$, and note as a result that, for all $m\geq N$, we have that $(\sigma^{p^m},\delta^{p^m})$ is compatible with $u$ and hence $Q^+[[x^{p^m};\sigma^{p^m},\delta^{p^m}]]$ is prime by Theorem \ref{thm: simple artinian with compatible skew derivation}.

\begin{lem}\label{lem: t in Q^times}
For any $m\geq N$: if $u(t^{p^m})=0$, then $t\in Q^{\times}$ and $t^{p^m}\in\O^{\times}$.
\end{lem}

\begin{proof}

For any $q\in\O$, $\delta^{p^m}(q)=t^{p^m}q-\sigma^{p^m}(q)t^{p^m}\in J(\O)$ and $\sigma^{p^m}(q)-q\in J(\O)$ by compatibility. Therefore $$t^{p^m}q-qt^{p^m}=t^{p^m}q-\sigma^{p^m}(q)t^{p^m}+(\sigma^{p^m}(q)-1)t^{p^m}\in J(\O),$$ so we conclude that $t^{p^m}+J(\O)$ is central in $\O/J(\O)$, and it is nonzero since $u(t^{p^m})=0$. Since $u$ is a standard filtration, $Z(\O/J(\O))$ is a field, so there exists $v\in\O$ such that $t^{p^m}v\equiv 1$ (mod $J(\O))$. 

Writing $t^{p^m}v=1+y$ with $y\in J(\O)$, it follows from completeness of $u$ that $t^{p^m}v$ is a unit in $Q$, with inverse $1-y+y^2-y^3+\dots\in \O$, and hence $t^{p^m}$ is a unit in $\O$, from which we deduce that $t$ is a unit in $Q$.
\end{proof}

Ultimately, we aim to show that $Q^+[[x; \sigma, \delta]]$ is prime, and so in light of Corollary \ref{cor: when sigma does not fix Z part 2}, we may assume that $\sigma$ acts trivially on $Z$.

\begin{lem}\label{lem: sigma(a)=a when action on centre is trivial}
Assume that $\sigma$ acts trivially on $Z$, and fix any $k\in\mathbb{N}$. If there exists $a\in Q^{\times}$ such that $\sigma^{p^k}(q)=aqa^{-1}$ for all $q\in Q$, then $\sigma(a)=a$.
\end{lem}

\begin{proof}
First note that for any $q\in Q$, $$\sigma(a)a^{-1}q=\sigma(a)\sigma^{-p^k}(q)a^{-1}=\sigma(a\sigma^{-p^k-1}(q))a^{-1}=\sigma(a\sigma^{-p^k-1}(q)a^{-1})\sigma(a)a^{-1}=q\sigma(a)a^{-1}$$ i.e.\ $\sigma(a)a^{-1}\in Z$.

Therefore, since we are assuming $\sigma|_{Z}=\id$, it follows that $\sigma^i(\sigma(a)a^{-1})=\sigma(a)a^{-1}$ for all $i$, and we see that $$(\sigma(a)a^{-1})^{p^k}=\underset{0\leq i\leq p^k-1}{\prod}\sigma^i(\sigma(a)a^{-1})=\underset{0\leq i\leq p^k-1}{\prod}\sigma^{i+1}(a)\sigma^i(a)^{-1} = \sigma^{p^k}(a) a^{-1},$$ and since $\sigma^{p^k}(a)=a$, it follows that this product is equal to 1.

So since $Z$ is a field of characteristic $p$, and $\sigma(a)a^{-1}\in Z$ is a $p^k$'th root of unity, we must have that $\sigma(a)a^{-1}=1$, and hence $\sigma(a)=a$.
\end{proof}

As remarked in our setup at the beginning of \S \ref{section: crossed product decomposition}, we may localise the Noetherian rings $Q^+[[x;\sigma,\delta]]$ and $Q^+[[x^{p^m};\sigma^{p^m},\delta^{p^m}]]$ (for any $m\in\mathbb{N}$) to form the Noetherian rings 
$$R:=Q^+[[x;\sigma,\delta]]_{(x-t)}, \qquad A:=Q^+[[x^{p^m};\sigma^{p^m},\delta^{p^m}]]_{(x^{p^m}-t^{p^m})}$$ respectively. Extending $\sigma$ to $A$ using Proposition \ref{propn: properties of extended skew ders} and setting $g = x - t$, we get a crossed product $R=A\ast\mathbb{Z}/p^m\mathbb{Z}=\bigoplus_{i=0}^{p^m-1} Ag^i$, where the action of $g$ on $A$ is given by $ga = \sigma(a)g$ for all $a\in A$. We will also fix some $m \geq N$, so that $A$ is prime by Theorem \ref{thm: simple artinian with compatible skew derivation}, and we are in the situation of \S \ref{subsec: nilradical}.

\begin{propn}\label{propn: sub-skew power series ring satisfies P1}
$(A,\sigma)$ satisfies \eqref{P1}.
\end{propn}

\begin{proof}

If $\O$ is the maximal order associated to $(Q,u)$, then using Lemma \ref{lem: restricted SPSR over Q} and compatibility of $(\sigma^{p^m},\delta^{p^m})$, we see that $A_0:=Q^+[[x^{p^m};\sigma^{p^m},\delta^{p^m}]]=Q\otimes_{\O}\O[[x^{p^m};\sigma^{p^m},\delta^{p^m}]]$. Applying \cite[Theorem C]{jones-woods-2} we see that nonzero ideals of $\O[[x^{p^m};\sigma^{p^m},\delta^{p^m}]]$ must contain nonzero polynomial elements, and since any ideal of $A$ has nonzero intersection with $A_0$, and therefore nonzero intersection with $\O[[x^{p^m};\sigma^{p^m},\delta^{p^m}]]$, it follows that all two-sided ideals of $A$ contain nonzero polynomial elements.

Let $I$ be a nonzero, $\sigma$-invariant, two-sided ideal of $A$, let $Y=x^{p^m}-t^{p^m}$, and choose a nonzero polynomial element $z:=a_0+a_1Y+\dots+a_sY^s\in I\cap A_0$ with $a_i\in Q$ of minimal degree (so that $a_0\neq 0 \neq a_s$). Observing that $J:=\{q\in Q:q+b_1Y+\dots+b_sY^s\}$ is a two-sided ideal in the simple ring $Q$ containing $a_0\neq 0$, we conclude that $J=Q$. In particular, we may assume that $a_0=1$.

To prove that $z$ is $\sigma$-invariant and central, we use a similar argument as in Proposition \ref{propn: minimal elements and centrality properties}. Specifically, since $Y$ is $\sigma$-invariant, then for any element $q\in Q$, we can evaluate $$\sigma(z)-z=(\sigma(a_1)-a_1)Y+\dots+(\sigma(a_s)-a_s)Y^s$$ and $$qz-zq=(qa_1-a_1\sigma^{p^m(q)})Y+\dots+(qa_s-a_s\sigma^{sp^m}(q))Y^s$$

Since $I$ is a $\sigma$-invariant, two-sided ideal in $A$, we know that both these elements lie in $I$. But multiplying on the right by $Y^{-1}$ we get that they must both be zero, else this would contradict minimality of $s$. Thus $z$ is central and $\sigma$-invariant as required.\end{proof}

We are now prepared to complete the proof of Theorem \ref{letterthm: SFOH implies prime}. The following result is the final step in the proof.

\begin{thm}\label{thm: case where a power of sigma is inner and untwistable}
Let us suppose $\sigma$ acts trivially on $Z$, and that there exists $m\geq N$, $a\in\O^\times$ such that $\sigma^{p^m}(q)=aqa^{-1}$ for all $q\in Q$. Suppose further that one of the following two conditions is satisfied:

\begin{enumerate}[label=(\alph*)]
\item $u(t^{p^m})>0$.

\item $u(a+t^{p^m})>0$.
\end{enumerate}

Then $Q^+[[x;\sigma,\delta]]$ is a prime ring.
\end{thm}

\begin{proof}

Since $m\geq N$, we know using Theorem \ref{thm: simple artinian with compatible skew derivation} that $A$ is a prime ring, and $(A,\sigma)$ satisfies (\ref{P1}) by Proposition \ref{propn: sub-skew power series ring satisfies P1}. Moreover, if $R$ is prime then $Q^+[[x;\sigma,\delta]]$ is prime by \cite[Proposition 2.1.16]{MR}. We also know that $\sigma(a)=a$ by Lemma \ref{lem: sigma(a)=a when action on centre is trivial}. So to prove that $R=A\ast(\mathbb{Z}/p^m\mathbb{Z})$ is prime, we only need to prove that the remaining conditions of Theorem \ref{thm: conditions for primality} are satisfied in both cases: that is, it will suffice to find an element $X\in Z(A)^\sigma$ such that $A = Q^+[[X]]$ or $Q^+[[X]][X^{-1}]$ and $x^{p^m} - t^{p^m}$ is a polynomial of degree 1 in $X$.

\begin{enumerate}
\item Since $u(t^{p^m})>0$, if we define $Y:=x^{p^m}-t^{p^m}$, then by \cite[Theorem B]{jones-woods-2} (and Lemma \ref{lem: restricted SPSR over Q}), we see that
$$Q^+[[x^{p^m};\sigma^{p^m},\delta^{p^m}]] = Q\otimes_\O \O [[x^{p^m};\sigma^{p^m},\delta^{p^m}]] = Q\otimes_\O \O[[Y; \sigma^{p^m}]] = Q^+[[Y;\sigma^{p^m}]],$$
so $A=Q^+[[Y;\sigma^{p^m}]]_{(Y)}$. Moreover, since $a\in\O^\times$, we have that $u(a)=u(a^{-1})=0$, and if we set $X:=a^{-1}Y$, it follows from \cite[Theorem B]{jones-woods-3} again that $Q^+[[Y;\sigma^{p^m}]]=Q^+[[X]]$. As $a$ is already invertible, it follows that $A = Q^+[[X]]_{(Y)} = Q^+[[X]]_{(X)} = Q^+[[X]][X^{-1}]$.

This $X$ is central in $A$ by construction, and since $X=a^{-1}(x^{p^m}-t^{p^m})$ and $\sigma(t)=t$, it follows that $\sigma(X)=X$, so $X\in Z(A)^\sigma$.

Clearly $\sigma^{p^m}$ is an inner automorphism of $Q$, and $x^{p^m}-t^{p^m}=aX$ is a polynomial of degree 1 in $X$ as required.

\item Note that since $u(a)=0$, it follows from the assumption that $u(a+t^{p^m})>0$ that $u(t^{p^m})=0$, and hence $t^{p^m}\in \O^{\times}$ by Lemma \ref{lem: t in Q^times}. In particular, $u(t^{p^m})=u(t^{-p^m})=0$. 

Let $Y:=-t^{-p^m}x^{p^m}$, and let $\alpha$ be the inner automorphism of $Q$ defined by $\alpha(q) = t^{-p^m}qt^{p^m}$. Since $(\sigma^{p^m},\delta^{p^m})$ is compatible with $u$, it follows from \cite[Theorem B]{jones-woods-3} that $$Q^+[[x^{p^m};\sigma^{p^m},\delta^{p^m}]]=Q^+[[Y;\alpha\sigma^{p^m},-t^{-p^m}\delta^{p^m}]].$$ But $-t^{-p^m}\delta^{p^m}(q)=\alpha\sigma^{p^m}(q)-q$. So set $\tau:=\alpha\sigma^{p^m}$, and we see that $$Q^+[[x^{p^m};\sigma^{p^m},\delta^{p^m}]]=Q^+[[Y;\tau,\tau-\id]].$$

Moreover, $s:=-t^{-p^m}a \in\O^{\times}$, as both $t^{p^m} \in \O^\times$ and $a \in \O^\times$, and $\tau(q)=sqs^{-1}$. Also since $a+t^{p^m}\in J(\O)$, we see that $s-1=-t^{-p^m}(a+t^{p^m})\in J(\O)$, and hence $s^{-1} \equiv s \equiv 1 \bmod J(\O)$. Again, since $\sigma(t)=t$, and $\sigma(a)=a$ by Lemma \ref{lem: sigma(a)=a when action on centre is trivial}, it is clear that $\sigma(s)=s$.

So, let $X:=s^{-1}Y+(s^{-1}-1)$, and applying \cite[Theorem B]{jones-woods-3} again, we see that $A=Q^+[[Y;\tau,\tau-\id]]=Q^+[[X]]$. Clearly $X$ is central in $A$, and since $\sigma(s)=s$ and $\sigma(Y)=Y$, we have that $X\in Z(A)^{\sigma}$.

Finally, $x^{p^m}-t^{p^m}=-t^{p^m}(s^{-1}X+1)$, which is a polynomial of degree 1 in $X$ as required.
\end{enumerate}

So using Theorem \ref{thm: conditions for primality}, we see that $Q^+[[x;\sigma,\delta]]_{(x-t)}=R$ is prime, as required.\end{proof}

We now prove Theorem \ref{letterthm: SFOH implies prime}, and so we revert to the notation of the introduction: crucially, we have fixed some $m\in\mathbb{N}$ such that $(\sigma^{p^m}, \delta^{p^m})$ is compatible with $u$.

\begin{proof}[Proof of Theorem \ref{letterthm: SFOH implies prime}]
By Lemma \ref{lem: identical on centre}, $\sigma|_{Z}$ has order $p^s$ for some $s\leq\ell$. Choose $M \geq \max\{\ell, m\}$ such that either (a) $u(t^{p^M}) > 0$ or (b) $u(a^{p^{M-\ell}} + t^{p^M}) > 0$. In either case, set $b:= a^{p^{M-\ell}}$, and we now have that $\sigma^{p^M}(q) = bqb^{-1}$. Since $M \geq s$, we may apply the corresponding part of Theorem \ref{thm: case where a power of sigma is inner and untwistable} to see that $Q^+[[x^{p^s};\sigma^{p^s},\delta^{p^s}]]_{(x^{p^s}-t^{p^s})}$ is prime. Hence $Q^+[[x;\sigma,\delta]]_{(x-t)}$ and $Q^+[[x;\sigma,\delta]]$ are also prime by Corollary \ref{cor: when sigma does not fix Z large ring is prime iff small ring is sigma-prime}.\qedhere
\end{proof}

\section{Extension of scalars}\label{section: extension of scalars}

In \S \ref{section: crossed product decomposition} and Theorem \ref{letterthm: SFOH implies prime}, we demonstrated that, if the data $(\sigma,\delta, u)$ satisfies \eqref{SFOH} (plus certain technical conditions on the elements $a$ and $t$), then $Q^+[[x; \sigma, \delta]]$ is prime. However, \eqref{SFOH} is a rather restrictive condition, and we would like to be able to extend the scalars in $Q$ as in \eqref{FOH}: we will demonstrate in this section that, under mild conditions, such scalar extensions are well-behaved algebraically with respect to the skew power series rings. Once we have done this, in \S \ref{sec: FOH}, we will invoke a series of reductions to deduce \eqref{FOH} from some very general conditions.

\subsection{Discrete valuation rings and scalar extensions}\label{subsec: DVRs and scalar extensions}

Let $(D,v)$ be a complete discrete valuation ring with division ring of fractions $(F,v)$, and let $Z$ be the centre of $F$ (which is necessarily a field). Write $\varphi = v|_Z$, so that  $(Z, \varphi)$ is a complete discretely valued field, in the sense of \cite[Chapter 2, Definitions 1 and 6]{Rib99}.

(We adopt the following temporary notation. If $(A, f)$ is a complete discretely valued field, write $\O(A)$ for the valuation ring $f^{-1}([0,\infty])$, $\mathfrak{m}(A)$ for the maximal ideal $f^{-1}((0,\infty])$, and $k_A$ for the residue ring $\O(A)/\mathfrak{m}(A)$.)

Now the valuation ring of $(Z, \varphi)$ is $\O(Z) = Z(D)$, and the following well-known results in commutative valuation theory can be read off from \cite[Chapter 5.1, (A)]{Rib99} and \cite[\S 6.1, Theorem 1 and its proof, pp. 151--153]{Rib99}:

\begin{props}\label{props: extension of commutative valuations}
$ $

\begin{enumerate}[label=(\roman*)]
\item Let $K/Z$ be any finite extension of fields. Then $\varphi$ on $Z$ extends uniquely to a complete discrete valuation $\varphi_K$ on $K$. (This is not necessarily integer-valued, but we will see below that it can be rescaled to give an integer-valued valuation.)
\item The field extension $K/Z$ has basis $\{\alpha_i \pi^j : 1\leq i\leq f, 0\leq j\leq e-1\}$, where the images of $\alpha_1, \dots, \alpha_f$ form a basis for $k_K$ over $k_Z$, and $\pi$ is a uniformiser for the valuation ring $\O(K)$ of $K$. Here $e$ is the ramification index and $f$ is the inertial degree. (Note that, in the exceptional case where $Z = Z(D)$, there is no such uniformiser as the valuation on $Z$ is trivial; but in this case, we have $e = 1$ and $K = k_K$.)
\item For all $z_{ij} \in Z$, we have $\displaystyle \varphi_K\left(\sum_{i,j} z_{ij} \alpha_i \pi^j\right) = \min\left\{\varphi(z_{ij})+hj/e\right\}$, where $$h=\inf\{\varphi(z):z\in \mathfrak{m}(Z)\}\leq\infty.$$ (As above, if $h = \infty$ then $e = 1$.)

This means that $(K,\varphi_K)$ is filt-free over $(Z,\varphi)$ in the sense of Definition \ref{defn: filt-free}, with filt-basis $\{\alpha_i \pi^j\}$.
\end{enumerate}
\end{props}

Given the ramification index $e$ as above: if we replace the initial valuation $v$ on $D$ by $ev$, then the extended valuation $e\varphi_K$ is now integer valued. We are free to do this, since $v$ and $ev$ are algebraically (and hence topologically) equivalent.

\begin{defn}\label{defn: admissible}
Let $(D,v)$ be a discrete valuation ring. Then we say $D$ is \emph{admissible} if it is complete and $\gr_v(D)$ is finitely generated as a module over its centre.
\end{defn}

\begin{rk}\label{rk: why admissible DVRs}
In later subsections, beginning with a discrete valuation ring $(D, v)$ with division ring of fractions $(F, v)$ and a finite extension of fields $K/Z$ as above, we will attempt to extend the valuation $v$ to a standard filtration on $F_K := K\otimes_Z F$. To do this, we will apply Theorem \ref{thm: filtered localisation}.1, where this admissibility criterion is necessary.

Fortunately, admissible DVRs occur very naturally in our context. For example, let $(R, w_0)$ be a complete, filtered ring satisfying the hypotheses of Theorem \ref{thm: filtered localisation} (see Remark \ref{rk: examples satisfying filt hypotheses} for examples). Then by Theorem \ref{thm: filtered localisation}.1(iv), there exists a filtration $u$ on $Q(R)$ such that the extension of $u$ to its completion $Q$ is a standard filtration, and so it has an associated complete discrete valuation ring $D$ in the sense of Definition \ref{defn: standard filtrations}. It follows from Theorem \ref{thm: filtered localisation}.2 that $\gr_u(Q)$ is finitely generated over its centre, and hence that $D$ is admissible.
\end{rk}

Parts (i)--(ii) of the following lemma follow easily from more general facts about tensor products. Parts (iii) and (iv) are crucial to allow us to apply Theorem \ref{thm: filtered localisation} in the situation described above.

\begin{lem}\label{lem: tensor product filtration on F_K}
Let $(D,v)$ be a complete discrete valuation ring with division ring of fractions $(F,v)$. Write $(Z, \varphi)$ for the centre of $F$ (with the induced valuation), and let $(K, \varphi_K)$ be a finite extension of $Z$ (with the unique extended valuation).

Set $F_K=K\otimes_Z F$, and let $v^\otimes_K = \varphi_K \otimes v$ be the tensor product filtration on $F_K$ in the sense of Definition \ref{def: filtered tensor product}. Then

\begin{enumerate}[label=(\roman*)]
\item $F_K$ is a simple artinian ring.
\item $v^\otimes_K$ is complete.
\item $\gr_{v^\otimes_K}(F_K)$ is Noetherian, it is finitely and freely generated as a module over $\gr_{v}(F)$, and $\gr_{v}(F)$ contains no nonzero zero divisors in $\gr_{v^\otimes_K}(F_K)$.
\item If $D$ is admissible, then $\gr_{v^\otimes_K}(F_K)$ is finitely generated over $A:=Z(\gr_{v}(F))$.
\end{enumerate}
\end{lem}

\begin{proof}
$ $

\begin{enumerate}[label=(\roman*)]
\item $F$ is a simple $Z$-algebra, and hence $F_K$ is a simple $K$-algebra by Proposition \ref{propn: simplicity under tensor product}. Since $K$ is filt-free over $Z$ by Property \ref{props: extension of commutative valuations}(iii), it follows from Lemma \ref{lem: filt-free module under tensor product} that $F_K$ is filt-free (and so has finite dimension) as an $F$-module, and hence is also an artinian ring.
\item Follows from Lemma \ref{lem: filt-free module under tensor product}.
\item Property \ref{props: extension of commutative valuations}(iii) implies that $K$ is filt-free over $Z$. Hence we know that
$$\gr_{v^\otimes_K}(F_K)\cong \gr_{\varphi_K}(K)\otimes_{\gr_\varphi(Z)}\gr_{v}(F)$$
by Lemma \ref{lem: associated graded of tensor product}. We also know that $\gr_{\varphi_K}(K)$ has a finite homogeneous basis as a free $\gr_\varphi(Z)$-module by Lemma \ref{lem: associated graded of filt-free}, and hence $\gr_{v^\otimes_K}(F_K)$ is finitely generated over $\gr_v(F)$ by Lemma \ref{lem: basis for K extends to basis for F_K}. Since $\gr_{v}(F)$ is Noetherian, it follows that $\gr_{v^\otimes_K}(F_K)$ is also Noetherian. 

Moreover, write $k = D/J(D)$, so that $\gr_{v}(D)\cong k[X;\alpha]$ for some automorphism $\alpha$ by Property \ref{props: DVRs}(iii), and thus $\gr_{v}(F)\cong k[X,X^{-1};\alpha]$, and hence $\gr_{v}(F)$ is a domain. So since $\gr_{v^\otimes_K}(F_K)$ is free over $\gr_{v}(F)$, it follows that $\gr_{v}(F)$ has no zero divisors in $\gr_{v^\otimes_K}(F_K)$. 
\item Since $D$ is admissible, $\gr_{v}(D)\cong k[X;\alpha]$ is finite-dimensional over its centre, and hence so is $\gr_{v}(F) \cong k[X,X^{-1};\alpha]$.\qedhere
\end{enumerate}
\end{proof}

\begin{rk}\label{rk: identification of centres}
Continue to let $(D,v)$ be a complete discrete valuation ring and $(F,v)$ its (discretely valued) division ring of fractions.

In our context, we are often dealing with simple artinian rings of the form $Q=M_r(F)$ for some $r\in\mathbb{N}$. Then the natural diagonal inclusion map $d: F\hookrightarrow Q$ restricts to an isomorphism $Z(F) \cong Z(Q)$, and in what follows, we will frequently identify the centres $Z(F) = Z(Q)$, writing simply $Z$ for both.

Moreover, if $u = M_r(v)$ is the standard filtration on $Q$ induced by $F$, then $u\circ d = v$. It follows that $u|_Z = u\circ d|_Z = v|_Z$, and so we will be able to write all of these filtrations as $\varphi$ as above.
\end{rk}

\begin{lem}
Retain the above notation, and let $K/Z$ be a finite extension of fields, where $K$ is given the unique extended valuation $\varphi_K$. Define $Q_K = K\otimes_Z Q$ with filtration $u^\otimes_K = \varphi_K \otimes u$, and $F_K = K\otimes_Z F$ with filtration $v^\otimes_K = \varphi_K \otimes v$. Then there is an isomorphism of filtered rings $(Q_K, u^\otimes_K) \to (M_r(F_K), M_r(v^\otimes_K))$.
\end{lem}

\begin{proof}
On the level of rings, the isomorphism is
$$i: Q_K:=K\otimes_Z Q = K\otimes_Z M_r(F) \to M_r(K\otimes_Z F)=M_r(F_K),$$
where the isomorphism $K\otimes_Z M_r(F)\to M_r(K\otimes_Z F)$ is given by Lemma \ref{lem: Q_K = M_n(F_K)}.

The remainder of the proof is a matter of notation, but we spell it out for completeness:
\begin{itemize}
\item Write $K_n$, $F_n$ and $Q_n$ for the $n$th level sets of $(K, \varphi_K)$, $(F, v)$ and $(Q, u)$ respectively.
\item Write $\pi_A: K \otimes A \to K\otimes_Z A$ for each of $A = Q$, $F$ and $M_r(F)$, where the tensor product on the left-hand side is implicitly taken over $\mathbb{Z}$. Also write $\widetilde{\pi}_F = M_r(\pi_F) : M_r(K\otimes F)\to M_r(K\otimes_Z F)$.
\item Denote by $j$ the isomorphism $K\otimes M_r(F) \to M_r(K\otimes F)$. Note that $j$ respects additive subgroups in the sense that, if $A\leq K$ and $B\leq F$, then $j(A\otimes M_r(B)) = M_r(A\otimes B)$.
\end{itemize}

We are claiming that the $n$th level sets of $u^\otimes_K$ and of $M_r(v^\otimes_K)$ correspond under $i$. That is, recalling the description of the level sets in Definition \ref{def: filtered tensor product}, we want to show that
$$i\left(\sum_{s+t\geq n} \pi_Q(K_s \otimes Q_t)\right) = M_r\left( \sum_{s+t\geq n} \pi_F(K_s \otimes F_t)\right).$$
Now it is easy to check that
\begin{align*}
i\left(\sum_{s+t\geq n} \pi_Q(K_s \otimes Q_t)\right) &= \sum_{s+t\geq n} i\left(\pi_Q(K_s \otimes Q_t)\right)\\
&= \sum_{s+t\geq n} \widetilde{\pi}_F(j(K_s\otimes Q_t))&&\text{as } i\circ \pi_Q = \widetilde{\pi}_F \circ j\\
&= \sum_{s+t\geq n} \widetilde{\pi}_F(j(K_s\otimes M_r(F_t)))&& \text{by definition of } u = M_r(v)\\
&= \sum_{s+t\geq n} \widetilde{\pi}_F(M_r(K_s\otimes F_t)) && \text{as } j \text{ respects additive subgroups}\\
&= \sum_{s+t\geq n} M_r(\pi_F(K_s\otimes F_t))&&\text{by definition of } \widetilde{\pi}_F\\
&= M_r\left(\sum_{s+t\geq n} \pi_F(K_s\otimes F_t)\right)
\end{align*}
as required.
\end{proof}

\subsection{Admissible data}

Now let $Q$ be a simple artinian ring of characteristic $p$, and let $Z=Z(Q)$ be the centre of $Q$. (Much of what follows will continue to hold in the case where $Q$ is a simple artinian $\mathbb{Z}_p$-algebra, after appropriate modifications, but we will not need this generality: see \cite[\S 3.4]{jones-woods-3}.)

\begin{defn}\label{defn: admissible data}
An \emph{admissible datum} on $Q$ is a triple $(\sigma, \delta, u)$, where:

\begin{enumerate}[label=(AD\arabic*)]
\item $u$ is a standard filtration on $Q$.
\item The discrete valuation ring $D$ associated to $Q$ is \emph{admissible} in the sense of Definition \ref{defn: admissible}.
\item $\sigma$ is an automorphism of $Q$ satisfying $\sigma|_Z=\id_Z$.
\item $\delta$ is an inner $\sigma$-derivation defined by $\delta(q) = tq - \sigma(q)t$ for some $t\in Q$ satisfying $\sigma(t)=t$.
\item $[\sigma]\in \Out(Q)$ has finite order.
\item $\sigma$ is filtered with respect to $u$, and there exists $m\in\mathbb{N}$ such that $(\sigma^{p^m}-\id)(\O) \subseteq J(\O)^2$ and $\delta^{p^m}(\O)\subseteq J(\O)^2$, where $\O$ is the maximal order associated to $(Q, u)$.
\end{enumerate}
\end{defn}

\begin{rks}\label{rks: AD hypotheses}
The various constituent parts of this definition have several purposes, so we comment on them individually.

\begin{enumerate}[label=(AD\arabic*)]
\item This allows us to work with the well-behaved properties specific to standard filtrations, such as Theorem \ref{thm: simple artinian with compatible skew derivation}.

\item As mentioned in Remark \ref{rk: why admissible DVRs}, this will allow us to perform ``extension of scalars", and extend the standard filtration $v$ on $F$ to a well-behaved standard filtration on $F_K$.

\item This technical assumption will be necessary in what follows to ensure that our extension of scalars interacts nicely with skew power series rings: it follows from Corollary \ref{cor: when sigma does not fix Z large ring is prime iff small ring is sigma-prime} that this assumption does not affect the generality of our results.

\item In particular, $(\sigma, \delta)$ is a commuting skew derivation on $Q$. Also, by Property \ref{props: Sigma, Delta, crossed product} there are commuting skew derivations $(\sigma^{p^n}, \delta^{p^n})$ for all $n\in\mathbb{N}$, and $\delta^{p^n}$ is the inner $\sigma^{p^n}$-derivation defined by the element $t^{p^n}$.

\item We impose hypothesis (AD5) as we have already dealt with the case where (AD5) is \emph{not} satisfied in \cite[Theorem C(i)]{jones-woods-3}.

\item It follows from Lemma \ref{lem: criterion for delta positive degree} that $(\sigma^{p^m}, \delta^{p^m})$ is compatible with $u$, and hence from Lemma \ref{lem: going up compatibility} that $(\sigma,\delta)$ is quasi-compatible with $u$, so we can define the skew power series ring $Q^+[[x;\sigma,\delta]]$ as in Definition \ref{defn: bounded skew power series rings}.
\end{enumerate}
\end{rks}

\subsection{Admissible extensions}\label{subsec: admissible extensions}

Let $Q$ be a simple artinian ring of characteristic $p$ equipped with an admissible datum $(\sigma, \delta, u)$. In particular, we may define the skew power series ring $Q^+[[x; \sigma, \delta]]$ with respect to $u$.

Let $Z=Z(Q)$, and let $K/Z$ be a finite extension of fields. We will now construct a canonical admissible datum $(\sigma_K, \delta_K, u_K)$ on the scalar extension $Q_K=K\otimes_Z Q$. Our aim is to do so in a way that allows us both to define the extended skew power series ring $Q_K^+[[y; \sigma_K, \delta_K]]$ with respect to $u_K$, and to pass information between $Q_K^+[[y; \sigma_K, \delta_K]]$ and $Q^+[[x; \sigma, \delta]]$.

As before, we will write $\varphi$ for the naturally induced valuation on $Z$, $\varphi_K$ for its unique extension to $K$, and $u^\otimes_K = \varphi_K \otimes u$ for the corresponding filtration on $Q_K$.

\begin{defn}\label{defn: admissible extension} 
Let $K/Z$ be a finite extension of fields, and suppose $(\sigma, \delta, u)$ and $(\sigma_K, \delta_K, u_K)$ are admissible data on $Q$ and $Q_K$ respectively. 

We say that $(\sigma_K, \delta_K, u_K)$ is an \emph{admissible extension} of $(\sigma, \delta, u)$ if:

\begin{enumerate}[label=(AE\arabic*)]
\item $\sigma_K = \id \otimes \sigma$ and $\delta_K = \id \otimes \delta$.

Note: these tensor products are well-defined as both $\sigma$ and $\delta$ are $Z$-linear by (AD3--4).
\item For all $q\in Q$, if $u(q) \geq 0$, then $u_K(q) \geq 0$.
\item For any sequence $(q_i)_{i\in\mathbb{N}}$ with $q_i\in Q_K$, $(u_K(q_i))$ is bounded below if and only if $(u^\otimes_K(q_i))$ is bounded below.
\end{enumerate}

We write $(\sigma, \delta, u)\leq (\sigma_K, \delta_K, u_K)$ to mean that $(\sigma_K, \delta_K, u_K)$ is an admissible extension of $(\sigma, \delta, u)$.

Finally, we say that $(\sigma_K, \delta_K, u_K)$ is a \emph{strong} admissible extension of $(\sigma, \delta, u)$ if the admissible data satisfy (AE1),(AE2) and the additional strengthened hypothesis
\begin{itemize}
\item[(AE3$'$)]$u_K$ is topologically equivalent to $u^\otimes_K$.
\end{itemize}
\end{defn}

\begin{rks}\label{rks: AE hypotheses}
As before, these hypotheses are complex and are playing several roles at once, so we comment on how they will be used.

\begin{enumerate}[label=(AE\arabic*)]
\item This ensures that the inclusion map $Q^+[[x; \sigma, \delta]] \to Q_K^+[[y; \sigma_K, \delta_K]]$ (once it is defined in Theorem \ref{thm: extending scalars for skew power series rings}) will be a ring homomorphism.

Viewing $Q \subseteq Q_K$ under the canonical inclusion map: it also ensures that, if $\sigma^n$ is an inner automorphism of $Q$ (defined by conjugation by $a\in Q^\times$), then $\sigma_K^n$ is an inner automorphism of $Q_K$ (also defined by conjugation by $a$). Moreover, if we are given $t\in Q$ such that $\sigma(t) = t$, and $\delta(q) = tq - \sigma(q)t$ for all $q\in Q$, then we also have $\sigma_K(t) = t$ and $\delta_K(q) = tq - \sigma_K(q)t$ for all $q\in Q_K$. 
\item Write $\O$ and $\O_K$ for the maximal orders associated to $(Q, u)$ and $(Q_K, u_K)$: this assumption ensures that $a\in \O^\times$ (as required for \eqref{SFOH}) implies $a\in \O_K^\times$ after scalar extension. 
\item This ensures that the skew power series ring $Q_K^+[[y; \sigma_K, \delta_K]]$ is the same regardless of whether it is defined with respect to $u^\otimes_K$ (allowing us to realise $Q_K^+[[y; \sigma_K, \delta_K]]$ as $K\otimes_Z Q^+[[x; \sigma, \delta]]$) or with respect to $u_K$ (allowing us to use the well-behaved properties of standard filtrations).
\item[(AE3$'$)] It is clear that this condition implies (AE3), and there are cases where we explicitly need this condition in \S \ref{sec: FOH}, most notably Lemma \ref{lem: existence of zeta in Z(OK)^times}. This is strong enough for our current purposes, but we are uncertain whether we could replace (AE3) by (AE3$'$) or vice-versa in general.
\end{enumerate}
\end{rks}

\begin{rks}\label{rks: more AE stuff}
$ $

\begin{enumerate}
\item
In the definition of an admissible extension, we do not assume that $(Q,u)\to (Q_K,u_K)$ is continuous (unless we are assuming (AE3$'$) of course). However, in all the examples we construct this property will also hold, so we could include this assumption without affecting our results at all, and in a future work, this assumption may become necessary, so we remark on it here.
\item
We do not know whether the property of being a \emph{strong} admissible extension is transitive, even under our specialist hypotheses: topological equivalence is not preserved by tensor products in general, and the algebraic properties of the filtered localisation procedure are poorly understood: see step \textbf{(e)} of \cite[Procedure 4.1.2, Proposition 4.1.3]{jones-woods-3} and the following discussion.
\end{enumerate}
\end{rks}

\begin{lem}\label{lem: admissible data properties}
The relation $\leq$ on admissible data is transitive. That is, with notation as above: if $(\sigma, \delta, u)\leq (\sigma_K, \delta_K, u_K)$ and $(\sigma_K, \delta_K, u_K)\leq (\sigma_L, \delta_L, u_L)$, then $(\sigma, \delta, u)\leq (\sigma_L, \delta_L, u_L)$.
\end{lem}

\begin{proof}
Conditions (AE1) and (AE2) are clear from the definitions, so it remains to show (AE3), for which we set up some notation. 

Firstly, extend $\varphi := u|_Z$ uniquely to $\varphi_K$ on $K$ and $\varphi_L$ on $L$ using Property \ref{props: extension of commutative valuations}(i), and then define the tensor product filtrations $w_{L,K} = \varphi_L \otimes u_K$ on $Q_L$, $w_{L,Z} = \varphi_L \otimes u$ on $Q_L$, and $w_{K,Z} = \varphi_K \otimes u$ on $Q_K$. Also choose finite filt-bases $\{\alpha_1, \dots, \alpha_m\}$ for $L/K$ and $\{\beta_1, \dots, \beta_n\}$ for $K/Z$, which is possible by Property \ref{props: extension of commutative valuations}(iii). That is:
\begin{align*}
\varphi_L(\alpha_1 k_1 + \dots + \alpha_m k_m) = \min_{1\leq i\leq m} \{\varphi_L(\alpha_i) + \varphi_K(k_i)\},\\
\varphi_K(\beta_1 z_1 + \dots + \beta_n z_n) = \min_{1\leq j\leq n} \{\varphi_K(\beta_j) + \varphi(z_j)\},
\end{align*}
for all choices of $k_1, \dots, k_m \in K$ and $z_1, \dots, z_n \in Z$.

Now take a sequence $(x_i)_i$ of elements of $Q_L$.

The set $(u_L(x_i))_i$ is bounded below if and only if $(w_{L,K}(x_i))_i$ is bounded below, using (AE3) for the extension $L/K$. By Lemma \ref{lem: filt-free module under tensor product}, we see that $(Q_L, w_{L,K})$ is filt-free of finite rank over $Q_K$ with filt-basis $\{\alpha_1\otimes 1, \dots, \alpha_m\otimes 1\}$, and so for each $i$ there exist unique $\hat{q}_{i,1}, \dots, \hat{q}_{i,m} \in Q_K$ such that $x_i = \alpha_1\otimes \hat{q}_{i,1} + \dots + \alpha_m\otimes \hat{q}_{i,m}$ and
$$w_{L,K}(x_i) = \min_{1\leq j\leq m}\{\varphi_L(\alpha_j) + u_K(\hat{q}_{i,j})\}.$$
As the set $\{\alpha_1, \dots, \alpha_m\}$ is finite, $(w_{L,K}(x_i))_i$ is bounded below if and only if the set $(u_K(\hat{q}_{i,j}))_{i,j}$ is bounded below.

Similarly using (AE3) for the extension $K/Z$: $(u_K(\hat{q}_{i,j}))_{i,j}$ is bounded below if and only if $(w_{K,Z}(\hat{q}_{i,j}))_{i,j}$ is bounded below. By Lemma \ref{lem: filt-free module under tensor product}, we see that $(Q_K, w_{K,Z})$ is filt-free of finite rank over $Q$ with filt-basis $\{\beta_1\otimes 1, \dots, \beta_n\otimes 1\}$, and so for each $i$ and $1\leq j\leq m$ there exist unique $q_{i,j,1}, \dots, q_{i,j,n} \in Q$ such that $\hat{q}_{i,j} = \beta_1\otimes q_{i,j,1} + \dots + \beta_n\otimes q_{i,j,n}$ and
$$w_{K,Z}(\hat{q}_{i,j}) = \min_{1\leq i\leq n} \{\varphi_K(\beta_i) + u(q_{i,j,k})\}.$$
As the set $\{\beta_1, \dots, \beta_n\}$ is finite, $(w_{K,Z}(q_{i,j}))_{i,j}$ is bounded below if and only if the set $(u(q_{i,j,k}))_{i,j,k}$ is bounded below.

On the other hand, we know by Lemma \ref{lem: transitivity of filt-free extensions} that $\{\alpha_j \beta_k\}$ is a filt-basis for $L/Z$, and hence $\{\alpha_j \beta_k \otimes 1\}$ is a filt-basis for $(Q_L, w_{L,Z})$ over $Q$ by another application of Lemma \ref{lem: filt-free module under tensor product}. So we may write each $x_i = \sum_{j,k} \alpha_j \beta_k \otimes q_{i,j,k}$, and then calculate
$$w_{L,Z}(x_i) = \min_{j,k}\{ \varphi_L(\alpha_j \beta_k) + u(q_{i,j,k})\},$$
and as the filt-basis $\{\alpha_j \beta_k\}$ is finite, the set $(w_{L,Z}(x_i))_i$ is bounded below if and only if the set $(u(q_{i,j,k}))_{i,j,k}$ is bounded below. This establishes (AE3) for $L/Z$.
\end{proof}

We now begin to think about how to construct an admissible extension. The following lemma demonstrates that (AE1) does much of the work for us.

\begin{lem}\label{lem: aut extension to Q_K}
Given a finite extension of fields $K/Z$ and an admissible datum $(\sigma, \delta, u)$ on $Q$, define $\sigma_K = \id\otimes \sigma$ and $\delta_K = \id\otimes \delta$. Then $\sigma_K$ and $\delta_K$ satisfy (AD3--5).
\end{lem}

\begin{proof}
Follows from Remark \ref{rks: AE hypotheses} for (AE1).\end{proof}

We must now define a standard filtration $u_K$ on $Q_K$ such that $(\sigma_K, \delta_K, u_K)$ satisfies (AD2) and (AD6).

\begin{thm}\label{thm: scalar extension for D}
Let $Q$ be a simple artinian ring of characteristic $p$ with centre $Z$, and let $(\sigma, \delta, u)$ be an admissible datum on $Q$.

Write $(D, v)$ for the associated discrete valuation ring, $F$ for the (valued) division ring of fractions of $D$, and $\O = M_n(D)$ for the associated maximal order, so that $Q = M_n(F)$ and $u = M_n(v)$. Fix a finite extension of fields $K/Z$, define $F_K = K\otimes_Z F$ and $Q_K = K\otimes_Z Q$, and let $\varphi_K$ be the extension of $\varphi := v|_Z$ to $K$. Then:
\begin{enumerate}[label=(\roman*)]
\item There exists a standard filtration $v_K$ on $F_K$, topologically equivalent to $v^\otimes_K := \varphi_K\otimes v$, induced from an admissible discrete valuation ring. The natural inclusion map $(D, v)\to (F_K, v_K)$ is continuous, and has image inside the associated maximal order $\mathcal{B}_K$ of $F_K$.
\item The filtration $u_K = M_n(v_K)$ on $Q_K$ is standard, has associated maximal order $\O_K = M_n(\mathcal{B}_K)$, and is topologically equivalent to the tensor product filtration $u^\otimes_K := \varphi_K \otimes u$ on $Q_K$.
\end{enumerate}
\end{thm}

\begin{proof}
$ $

\begin{enumerate}[label=(\roman*)]
\item
By Lemma \ref{lem: tensor product filtration on F_K}: $(F_K, v^\otimes_K)$ is a simple artinian filtered $\mathbb{F}_p$-algebra, its associated graded ring $\gr_{v^\otimes_K}(F_K)$ is Noetherian, $\gr_{v^\otimes_K}(F_K)$ is finitely generated over the graded subring $A := Z(\gr_v(F))$, and $A$ contains no nonzero zero divisors in $\gr_{v^\otimes_K}(F_K)$. Moreover, since we know that $\gr_{v^\otimes_K}(F_K)\cong\gr_{\varphi_K}(K) \otimes_{\gr_\varphi(Z)} \gr_v(F)$ by Lemma \ref{lem: associated graded of tensor product}, it is clear that $A$ is central in $\gr_{v^\otimes_K}(F_K)$.

Therefore, we can apply Theorem \ref{thm: filtered localisation} to see that $F_K$ carries a complete filtration $v_K$ such that 
\begin{enumerate}[label=(\alph*)]
\item $v_K$ is a standard filtration,
\item $\gr_{v_K}(F_K)$ is finitely generated over its centre,
\item for all $q\in F_K$, if $v^\otimes_K(q) \geq 0$ then $v_K(q) \geq 0$,
\item $v_K$ is topologically equivalent to $v^\otimes_K$.
\end{enumerate}

Since $w_K(d)=v(d)\geq 0$ for all $d\in D$, it follows that the image of $D$ in $F_K$ is contained in $\mathcal{B}_K$. Moreover, since $v_K$ is standard, $\mathcal{B}_K\cong M_r(D')$ for some DVR $D'$, and $\gr(D')\cong k_{D'}[Y,\gamma]$. So since $\gr_{v_K}(F_K)\cong M_r(\gr(Q(D')))\cong M_r(\gr(D')_{Y})$ is finitely generated over its centre, it follows that $D'$ is admissible as required.
\item By (i), $v_K$ is a standard filtration on $F_K$. In other words, by Definition \ref{defn: standard filtrations}: there is an identification $F_K = M_r(F')$ and $v_K = M_r(v')$, where $(F', v')$ is a filtered division ring, and $v'$ is induced from a complete discrete valuation ring $D'$ whose division ring of fractions is $F'$.

Since we may canonically identify $Q_K=K\otimes_Z Q=K\otimes_Z M_n(F) = M_n(F_K)$, it follows that we can write $Q_K = M_{nr}(F')$, and $u_K := M_n(v_K) = M_{nr}(v')$ becomes a standard filtration, with valuation ring $M_n(\mathcal{B}_K)=M_{nr}(D')$. The topological equivalence of $u_K = M_n(v_K)$ with $u^\otimes_K = M_n(v^\otimes_K)$ also follows from part (i).\qedhere
\end{enumerate}
\end{proof}

We summarise the main findings of the theorem using the following commutative diagram of standard filtered rings:

\begin{equation}\label{eqn: commutative diagram of filtered rings}
\begin{aligned}
\xymatrix@C=5pc{
&(Q,u)\ar@{-->}[rr]&&(Q_K,u_K)\\
&(\O,u)\ar@{-->}[rr]\ar[u]&&(\O_K, u_K)\ar@{-->}[u]\\
(F,v)\ar@{-->}[rr]\ar[uur]&& (F_K,v_K)\ar@{-->}[uur]\\
(D,v)\ar@{-->}[rr]\ar[uur]\ar[u]&& (\mathcal{B}_K,v_K)\ar@{-->}[u]\ar@{-->}[uur]
}
\end{aligned}
\end{equation}

In this diagram, all diagonally oriented maps are in fact diagonal embeddings $A\to M_n(A)$, all horizontal maps are continuous, and all vertical maps are inclusions of valuation rings. Note that, as usual in the theory of central simple algebras, even though $F$ is a division ring, $F_K$ may not be.

In the following corollaries, we assume the setup of Lemma \ref{lem: aut extension to Q_K} and Theorem \ref{thm: scalar extension for D}.

\begin{cor}\label{cor: existence of strong extensions}
$(\sigma_K, \delta_K, u_K)$ is an admissible datum on $Q_K$.
\end{cor}

\begin{proof}
Hypotheses (AD1--5) follow easily from Lemma \ref{lem: aut extension to Q_K} and Theorem \ref{thm: scalar extension for D}, so we need only check (AD6).

Choosing $m$ as in (AD6) for $(\sigma, \delta, u)$, and applying Lemma \ref{lem: criterion for delta positive degree}, we see that $\deg_u(\sigma^{p^m} - \id) \geq 1$ and $\deg_u(\delta^{p^m}) \geq 1$. Hence for all $k\geq 1$, $\deg_u(\sigma^{p^{m+k}} - \id) = \deg_u((\sigma^{p^m} - \id)^{p^k}) \geq p^k$, and likewise $\deg_u(\delta^{p^{m+k}}) \geq p^k$.

Suppose $\{\alpha_1,\dots,\alpha_s\}$ is a filt-basis for $K/Z$: then given any $r=\alpha_1\otimes q_1+\dots+\alpha_s\otimes q_s\in Q_K$ (where $q_1, \dots, q_s \in Q$) and any $N\in\mathbb{N}$, we can calculate easily that
\begin{align*}
\sigma_K^{p^N}(r) - r &=\alpha_1\otimes(\sigma^{p^N}(q_1)-q_1)+\dots+\alpha_r\otimes(\sigma^{p^N}(q_s)-q_s),\\
\delta_K^{p^N}(r) &=\alpha_1\otimes \delta^{p^N}(q_1)+\dots+\alpha_s\otimes \delta^{p^N}(q_s).
\end{align*}
But $\{\alpha_1\otimes 1, \dots, \alpha_s\otimes 1\}$ is a filt-basis for $(Q_K, u_K^{\otimes})$ over $(Q, u)$ by Lemma \ref{lem: filt-free module under tensor product}, and so
\begin{align*}
u_K^{\otimes}(r) &= \min_{1\leq i\leq s} \{\varphi_K(\alpha_1) + u(q_i)\},\\
u_K^{\otimes}(\sigma_K^{p^N}(r) - r) &= \min_{1\leq i\leq s} \{\varphi_K(\alpha_1) + u(\sigma^{p^N}(q_i) - q_i)\} &&\geq \min_{1\leq i\leq s}\{\varphi_K(\alpha_1) + u(q_i) + \deg_u(\sigma^{p^N} - \id)\},\\
u_K^{\otimes}(\delta_K^{p^N}(r)) &= \min_{1\leq i\leq s} \{\varphi_K(\alpha_1) + u(\delta^{p^N}(q_i))\} &&\geq \min_{1\leq i\leq s}\{\varphi_K(\alpha_1) + u(q_i) + \deg_u(\delta^{p^N})\}.
\end{align*}
It follows that $u_K^{\otimes}(\sigma_K^{p^{m+k}}(r) - r) \geq u_K^{\otimes}(r) + p^k$ and $u_K^{\otimes}(\delta^{p^{m+k}}(r)) \geq u_K^\otimes(r) + p^k$ for all $r\in Q_K$.

But by Theorem \ref{thm: scalar extension for D}, we know that $u_K^\otimes$ is topologically equivalent to $u_K$, so there exist $N_1,N_2\in\mathbb{N}$ such that 
\begin{itemize}
\item $u_K^\otimes(r)\geq N_1$ for all $r\in\O_K$, and

\item if $u_K^\otimes(r)\geq N_2$ for some $r\in Q_K$, then $r\in J(\O_K)^2$.
\end{itemize}

It follows that $N_2>N_1$, so choose $k\geq \log_p(N_2-N_1)$. Then for all $r\in\O_K$, $u_K^\otimes(r)\geq N_1$, and $$u_K^{\otimes}(\sigma_K^{p^{m+k}}(r) - r) \geq u_K^{\otimes}(r) + p^k\geq N_1+(N_2-N_1)=N_2$$ so it follows that $r\in J(\O_K)^2$, so $(\sigma_K^{p^{m+k}}-\id)(\O_K)\subseteq J(\O_K)^2$, and a similar argument shows that $\delta^{p^{m+k}}(\O_K)\subseteq J(\O_K)^2$.\end{proof}

\begin{cor}
$(\sigma, \delta, u)\leq (\sigma_K, \delta_K, u_K)$ is a strong admissible extension.
\end{cor}

\begin{proof}
(AE1) is immediate from the definition of $\sigma_K$ and $\delta_K$. With notation as in Theorem \ref{thm: scalar extension for D}, and noting that $Q = M_n(F)$ and $u = M_n(v)$: (AE2) follows from the corresponding claims of Theorem \ref{thm: scalar extension for D}(i), and (AE3$'$) is given by Theorem \ref{thm: scalar extension for D}(ii).
\end{proof}

\subsection{Consequences for skew power series rings}

Let $Q$ be a simple artinian ring of characteristic $p$ with centre $Z$, and fix a finite extension of fields $K/Z$.

\begin{thm}\label{thm: extending scalars for skew power series rings}
Let $(\sigma, \delta, u)\leq (\sigma_K, \delta_K, u_K)$ be an admissible extension of admissible data on $Q$ and $Q_K$ respectively. Then the bounded skew power series rings $Q^+[[x;\sigma, \delta]]$ (defined with respect to $u$) and $Q_K^+[[y;\sigma_K, \delta_K]]$ (defined with respect to $u_K$) exist, and there is a canonical isomorphism of $K$-algebras
\begin{align*}
\Theta: K\otimes_Z Q^+[[x;\sigma,\delta]]&\to Q_K^+[[y;\sigma_K, \delta_K]]\\
\gamma \otimes \sum_{i=0}^\infty q_ix^i &\mapsto \sum_{i=0}^\infty (\gamma\otimes q_i) y^j.
\end{align*}

In particular, if $Q_K^+[[y; \sigma_K, \delta_K]]$ is prime (resp.\ simple) then $Q^+[[x; \sigma, \delta]]$ is prime (resp.\ simple).
\end{thm}

\begin{proof}
The existence of these rings follows from Remark \ref{rks: AD hypotheses} (AD6).

To show that $\Theta$ is a well-defined map, take $\gamma\in K$ and $\sum_{i\in\mathbb{N}} q_i x^i \in Q^+[[x]]$: then, by definition of $Q^+[[x]]$, we know that $(u(q_i))_{i \in \mathbb{N}}$ is bounded below. In order for $\sum_{i \in \mathbb{N}} (\gamma\otimes q_i) y^i$ to be a well-defined element of $Q_K^+[[y]]$, we need $(u_K(\gamma\otimes q_i))_{i \in \mathbb{N}}$ to be bounded below. But, by (AE3), this is equivalent to the claim that $(u^\otimes_K(\gamma\otimes q_i))_i$ is bounded below. (Here, as before, we have defined the field valuation $\varphi = u|_Z$, extended it uniquely to a complete valuation $\varphi_K$ on $K$, and set $u^\otimes_K = \varphi_K \otimes u$.)

So choose a finite filt-basis $\{\alpha_j\}_{j=1}^m$ for $K$ over $Z$ by Property \ref{props: extension of commutative valuations}(iii), and write $\gamma$ uniquely as $\gamma = \sum_j z_j \alpha_j$ for some $z_j\in Z$, so that $\varphi_K(\gamma) = \min_j\{u(z_j) + \varphi_K(\alpha_j)\}$. Then since $\{\alpha_j\otimes 1\}_j$ is a filt-basis for $Q_K$ by Lemma \ref{lem: filt-free module under tensor product}, it follows from Definition \ref{defn: filt-free} that 
$$u^\otimes_K(\gamma\otimes q_i) = u^\otimes_K\left(\sum_{1\leq j\leq m} \alpha_j\otimes z_jq_j\right) \geq \min_{1\leq j\leq m} \{ \varphi_K(\alpha_j) + u(z_j) + u(q_i)\}$$
for each $q_i$. As the filt-basis is finite, it follows that $(u^\otimes_K(\gamma\otimes q_i))_{i \in \mathbb{N}}$ is bounded below if and only if $(u(q_i))_{i \in \mathbb{N}}$ is bounded below.

Hence $\sum_{i=0}^{\infty} (\gamma\otimes q_i) y^i$ is a well-defined element of $Q_K^+[[y]]$, and it is clear that the map sending $\gamma \otimes \sum_{i=0}^\infty q_ix^i$ to $\sum_{i=0}^\infty (\gamma\otimes q_i) y^j$ defines a $K$-linear map $\Theta: K\otimes_Z Q^+[[x]]\to Q_K^+[[y]]$. Moreover, for any $r=\sum_{i=0}^{\infty}  q_i x^i\in Q^+[[x;\sigma,\delta]]$, we have $\Theta(r)=\sum_{i=0}^{\infty}  q_i y^i$, so it is clear that $\Theta$ restricts to a ring homomorphism from $Q^+[[x;\sigma,\delta]]$ to $Q^+[[y;\sigma_K,\delta_K]]$ by (AE1).

Since $\sigma_K|_K=\id_K$ and $\delta_K|_K = 0$ by (AE1) and (AD3--4), it follows that $K$ is central in $Q^+[[y;\sigma_K,\delta_K]]$. Therefore, for any $\alpha,\beta\in K$, $r,s\in Q^+[[x;\sigma,\delta]]$, $$\Theta((\alpha\otimes r)(\beta\otimes s))=\Theta(\alpha\beta\otimes rs)=\alpha\beta\Theta(rs)=\alpha\beta\Theta(r)\Theta(s)=\alpha\Theta(r)\beta\Theta(s)=\Theta(\alpha\otimes r)\Theta(\beta\otimes s),$$ and it follows that $\Theta$ is a ring homomorphism. To prove that it is an isomorphism, it remains to prove that it is bijective. To this end, we note that it may be written as
$$\Theta: \sum_{j=1}^m \left( \alpha_j \otimes \sum_{i=0}^\infty q_{i,j} x^i\right) \mapsto \sum_{i=0}^{\infty} \left( \sum_{j=1}^m (\alpha_j \otimes q_{i,j}) \right) y^i$$
where $\{\alpha_j\}_{j=1}^m$ is the filt-basis chosen above. Now if
$$\sum_{i=0}^{\infty} \left( \sum_j (\alpha_j \otimes q_{i,j}) \right) y^i = 0,$$
then each $y^i$-coefficient must satisfy $\sum_j (\alpha_j \otimes q_{i,j}) = 0$: hence, as $\{\alpha_j\otimes 1\}_{j=1}^m$ is a right $Q$-module basis for $Q_K$ by Lemma \ref{lem: filt-free module under tensor product}, we must have $q_{i,j} = 0$ for each $j$. It follows that $\Theta$ is injective.

To show that $\Theta$ is surjective, choose arbitrary $r_i\in Q_K$ such that $(u_K(r_i))$ is bounded below: we will show that $\sum_{i=0}^\infty r_i y^i$ is in the image of $\Theta$. To this end, write $r_i = \sum_j \alpha_j \otimes q_{i,j}$ for $q\in Q$, using the module basis described above. It will suffice to show that $\sum_{i=0}^\infty q_{i,j} x^i$ is a well-defined element of $Q^+[[x]]$, i.e.\ that $(u(q_{i,j}))_i$ is bounded below: but this follows once more from (AE3).

So $\Theta$ is a ring isomorphism, and it follows from Proposition \ref{propn: simplicity under tensor product} that if $Q^+[[y;\sigma_K,\delta_K]]$ is prime (resp.\ simple), then $Q^+[[x;\sigma,\delta]]$ is prime (resp.\ simple).
\end{proof}

Using this result, our strategy for proving statement (a) of Theorem \ref{letterthm: FOH implies prime} is clear: if $(\sigma,\delta,u)$ satisfies \eqref{FOH}, then by definition there exists an admissible extension $(\sigma_K,\delta_K,u_K)$ satisfying \eqref{SFOH}. The remaining work is to show that this can be assumed to satisfy hypothesis (a) or (b) of Theorem \ref{letterthm: SFOH implies prime}: then we will be able to conclude that $Q_K^+[[y;\sigma_K,\delta_K]]$ is prime, and then apply Theorem \ref{thm: extending scalars for skew power series rings}.

For statement (b) of Theorem \ref{letterthm: FOH implies prime}, we want to prove that if \eqref{FOH} is \emph{not} satisfied, then $Q^+[[x;\sigma,\delta]]$ is a simple ring. We will assume for contradiction that $Q^+[[x;\sigma,\delta]]$ is \emph{not} a simple ring: most of \S \ref{sec: FOH} is then dedicated to constructing a finite extension $K/Z$ using the two-sided ideal structure of $Q^+[[x;\sigma,\delta]]$ and an admissible extension $(\sigma, \delta, u) \leq (\sigma_K, \delta_K, u_K)$ satisfying \eqref{SFOH}.

\section{The finite order hypothesis}\label{sec: FOH}

Let $Q$ be a simple artinian ring of characteristic $p$, and let $(\sigma, \delta, u)$ be an admissible datum on $Q$. Fix $t\in Q$ such that $\sigma(t) = t$ and $\delta(q) = tq - \sigma(q)t$ for all $q\in Q$ using (AD4), and assume further that $t\in\O^\times$. (By Lemma \ref{lem: t in Q^times} and (AD6), this is equivalent to the assumption that $u(t^N) = 0$ for all $N\in\mathbb{N}$.) We are now in the setup of Theorem \ref{letterthm: FOH implies prime}. The following subsections are dedicated to a number of simplifications of this setup: in particular, we will show that $(\sigma, \delta, u)$ satisfies \eqref{FOH}. 

We fix the following notation for this section:
\begin{itemize}
\item $Z$ is the centre of $Q$, and $\O$ is the maximal order associated to $Q$.
\item Given any unit $U\in Q^\times$, write $\varphi_U$ for the map $\varphi_U(q) = UqU^{-1}$ for all $q\in Q$. If $K/Z$ is a finite extension of fields, and $U\in Q_K^\times$, then write $\varphi_{U,K}$ for the map $\varphi_{U,K}(q) = UqU^{-1}$ for all $q\in Q_K$.
\end{itemize}

\subsection{The Iwasawa condition}\label{subsec: Iwasawa condition}

Many of our arguments become far more straightforward if we assume the additional condition that $t=-1$, or equivalently $\delta=\sigma-\id$. We refer to this condition as the \emph{Iwasawa condition}, and we call the corresponding bounded skew power series ring $Q^+[[x;\sigma,\delta]]$ a skew power series ring of \emph{Iwasawa type}.

We begin with the following simplification: we show that, under the assumption that $t\in\O^\times$, the Iwasawa condition is never far away.

Set $\sigma_0:=\varphi_t^{-1}\circ\sigma$ and $\delta_0 := -t^{-1}\delta = \sigma_0-\id$, and note that conditions (AD1--6) are trivially satisfied for $(\sigma_0, \delta_0)$, so $(\sigma_0, \delta_0, u)$ is still an admissible datum on $Q$.

If $(\sigma,\delta)$ happens to be compatible with $u$, then so is $(\sigma_0, \delta_0)$, and it follows from \cite[Theorem B]{jones-woods-2} that $Q^+[[x;\sigma,\delta]] = Q^+[[y;\sigma_0,\delta_0]]$ as filtered rings, where $y=-t^{-1}x$. In general, we only know by (AD6) that we can choose $m\in\mathbb{N}$ such that $(\sigma^{p^m},\delta^{p^m})$ is compatible with $u$. Fixing any such $m$: since $\sigma(t)=t$, we know that $\sigma_0^{p^m}=\varphi_t^{-p^m}\circ\sigma^{p^m}=\varphi_{t^{p^m}}^{-1}\circ\sigma^{p^m}$ and $\delta^{p^m}(q)=t^{p^m}q-\sigma^{p^m}(q)t^{p^m}$, from which it follows by the same reasoning that
$$Q^+[[x^{p^m};\sigma^{p^m},\delta^{p^m}]]=Q^+[[y^{p^m};\sigma_0^{p^m},\delta_0^{p^m}]],$$
which we denote $R$. Next, by Theorem \ref{thm: crossed product after localisation}, we have crossed product decompositions
\[
Q^+[[x; \sigma, \delta]] = \bigoplus_{i=0}^{p^m - 1} R g^i, \qquad Q^+[[y; \sigma_0, \delta_0]] = \bigoplus_{i=0}^{p^m - 1} R h^i,
\]
where $g = x-t$ and $h = y+1 = -t^{-1}g$. This is a \emph{diagonal change of basis} \cite[\S 1]{passmanICP}: in other words, since $-t$ is a unit in $R$, we have $R(-t) = R$, and so
$$Q^+[[x; \sigma, \delta]] = \bigoplus_{i=0}^{p^m - 1} R g^i = \bigoplus_{i=0}^{p^m - 1} R (-th)^i = \bigoplus_{i=0}^{p^m - 1} R h^i = Q^+[[y; \sigma_0, \delta_0]].$$
In particular, $Q^+[[x; \sigma, \delta]]$ is simple (resp.\ prime) if and only if $Q^+[[y; \sigma_0, \delta_0]]$ is simple (resp.\ prime).

Next, we check how this equality interacts with \eqref{FOH}. But if $K/Z$ is a finite extension of fields, and $(\sigma, \delta, u) \leq (\sigma_K, \delta_K, u_K)$, then it is easy to check that $(\sigma_0, \delta_0, u)\leq (\sigma_{0,K}, \delta_{0,K}, u_K)$, where
\begin{align}\label{eqn: formula for sigma_(0,K)}
\sigma_{0,K} := \id \otimes \sigma_0 = \id \otimes \varphi_t^{-1} \circ \sigma = \varphi_{t,K}^{-1} \circ (\id \otimes \sigma) = \varphi_{t,K}^{-1} \circ \sigma_K
\end{align}
(as $K$ is central in $Q_K$), and $\delta_{0,K} := \id_K \otimes \delta_0 = \sigma_{0,K} - \id_{Q_K}$ is still of Iwasawa type, hence we may pass between admissible extensions of $(\sigma, \delta, u)$ and $(\sigma_0, \delta_0, u)$. Moreover, by our assumption $t\in \O^\times$, equation \eqref{eqn: formula for sigma_(0,K)} implies that $(\sigma_{0,K},\delta_{0,K},u_K)$ satisfies \eqref{SFOH} if and only if $(\sigma_K, \delta_K, u_K)$ does, and hence $(\sigma_0, \delta_0, u)$ satisfies \eqref{FOH} if and only if $(\sigma, \delta, u)$ does.

Hence, in proving Theorem \ref{letterthm: FOH implies prime}, it will suffice to replace $(\sigma, \delta)$ by $(\sigma_0, \delta_0)$. So until the end of the section, we make the following standing assumptions:

\textbf{Setup.} Let $Q$ be a simple artinian ring of characteristic $p$, and let $(\sigma, \delta, u)$ be an admissible datum on $Q$, where $\delta = \sigma - \id$.

\subsection{Central scaling}\label{subsec: central scaling}

To start, we know using (AD5) that there exists an element $a\in Q^\times$ such that $\sigma^n = \varphi_a$ for some $n$. We aim to show further (after appropriate ``admissible" extensions to $\sigma_K$, $Q_K$, $\O_K$) that we can assume $a\in\O_K^\times$, and that $n=p^\ell$ for some $\ell\in\mathbb{N}$. This requires several steps, and at each stage we will need to retain tight control over the conjugating element $a$ and the power to which we must raise $\sigma$.

Let $n$ be the order of $[\sigma]\in\Out(Q)$, and let $a\in Q^\times$ such that $\sigma^n = \varphi_a$.

We aim to show that, under appropriate conditions, we may find an admissible extension $(\sigma, \delta, u) \leq (\sigma_K, \delta_K, u_K)$ to an appropriate field with the following properties:
\begin{itemize}
\item $[\sigma_K] \in \Out(Q_K)$ has order $n'$, and
\item there exist a natural number $\ell$ and $c\in \O_K^\times$ such that $\sigma_K^{p^\ell n'} = \varphi_{c,K}$.
\end{itemize}

\textit{Strategy.} The case $Z\cap J(\O) = 0$ will be easy to deal with separately using Corollary \ref{cor: algebraic with small centre}, so we restrict to the case $Z\cap J(\O) \neq 0$. Note also that, if we already have $a\in \O^\times$, then there is nothing to prove: we simply take $K = Z$ (so $n' = n$) and $\ell = 0$. More generally, if there exists an element $z\in Z^\times$ with $u(z) = u(a)$, then we can replace $a$ with $c:=z^{-1}a$: we still have that $\sigma^n = \varphi_c$, and now $u(c)=0$, so if we are able to ensure that the skew derivation is compatible with $u$, we may simply apply Lemma \ref{lem: a regular}(ii) to see that $c\in\O^\times$.

In general, there is no such $z\in Z^\times$, and so our approach will be to extend the centre $Z$ to a finite extension field $K$ in which we can find an appropriate element $z\in K^\times$.

\begin{lem}\label{lem: central scaling}
If $Z\cap J(\O)\neq 0$, then there exist a finite extension of fields $K/Z$, an admissible datum $(\sigma_K, \delta_K, u_K)$ on $Q_K$ (with associated maximal order $\O_K$) which is a strong admissible extension of $(\sigma,\delta,u)$, a natural number $\ell$ and an element $c\in \O_K^\times$ such that $\sigma_K^{p^\ell n} = \varphi_{c,K}$. Moreover, $K\cap J(\O_K) \neq 0$.
\end{lem}

\begin{proof}
Choose an element $z\in Z\cap J(\O)$ of minimal positive degree, say $C$. Let $K$ be the extension of $Z$ obtained by adjoining a $C$'th root of $z$, and let $\zeta_0$ be this root, i.e.\ $\zeta_0^{C}=z$.

By Corollary \ref{cor: existence of strong extensions}, there exists a strong admissible extension $(\sigma_K, \delta_K, u_K)$ of $(\sigma, \delta, u)$, and $\sigma_K^n = \varphi_{a,K}$ by Definition \ref{defn: admissible extension} and Remark \ref{rks: AE hypotheses} (AE1). By (AD6) for $(\sigma_K, \delta_K, u_K)$, we can fix $\ell$ sufficiently high to ensure that $(\sigma_K^{p^\ell}, \delta_K^{p^\ell})$ is compatible with $u_K$.

Next, set $b = a^{p^\ell}$. We will show that we can find $\zeta\in K^\times$ such that $\zeta^{-1}b \in \O_K^\times$. The result will then follow by taking $c = \zeta^{-1}b$.

In the case where $u(b)=0$, then $u(b^{-1}) = -u(b) = 0$ by Lemma \ref{lem: a regular}(ii), and so both $b$ and $b^{-1}$ are elements of $\O$. Since $\O\subseteq \O_K$ by (AE2), we have $b\in\O_K^{\times}$, and so we may take $\zeta = 1$.

In the case where $v:=u(b)\neq 0$, we have $u(b^{\pm C}) = \pm Cv$ by Lemma \ref{lem: a regular}(ii) (using the compatibility of $(\sigma^{p^\ell}, \delta^{p^\ell})$ and $u$). We also have $u(z^{\pm v}) = \pm Cv$ since $u|_Z = \varphi$ is a valuation on $Z$ (see e.g.\ the discussion of \S \ref{subsec: DVRs and scalar extensions}). Moreover, by Lemma \ref{lem: u-regular elements}, $u(z^{-v} b^C) = u(z^v b^{-C}) = 0$, and so $z^{-v} b^C \in \O^\times$, which implies $z^{-v} b^C \in\O_K^{\times}$ as above. But by definition, this is equal to $(\zeta_0^{-v}b)^C$, and if $(\zeta_0^{-v}b)^C\in \O_K^\times$ then we must have $\zeta_0^{-v}b\in \O_K^\times$ by another application of Lemma \ref{lem: a regular}(ii) (this time using the compatibility of $(\sigma_K^{p^\ell}, \delta_K^{p^\ell})$ and $u_K$). So in this case we take $\zeta = \zeta_0^v$.

For the final claim, simply note by (AE3$'$) that there exists $T\in \mathbb{N}$ such that $J(\O)^T \subseteq J(\O_K)$, so $z^T \in Z\cap J(\O_K) \subseteq K\cap J(\O_K)$.
\end{proof}

Unfortunately, this is not quite enough, as the order of $[\sigma_K]$ may be less than $n$, so we will need to repeat this process. In doing so, we may lose property (AE3$'$) (see Remark \ref{rks: more AE stuff}.2), but this will not cause us any problems.

\begin{propn}\label{propn: unit in finite extension}
If $Z\cap J(\O)\neq 0$, then there exist a finite extension of fields $L/Z$, an admissible datum $(\sigma_L, \delta_L, u_L)$ on $Q_L$ (with associated maximal order $\O_L$) which is an admissible extension of $(\sigma,\delta,u)$, a natural number $\ell$ and an element $c\in \O_L^\times$ such that $\sigma_L^{p^\ell n'} = \varphi_{c,L}$, where $n'$ is the order of $[\sigma_L] \in \Out(Q_L)$.
\end{propn}

\begin{proof}
Let $K_0 := Z$ and $(\sigma_0, \delta_0, u_0) := (\sigma, \delta, u)$. Set $n(0) := n$ to be the order of $[\sigma] \in \Out(Q)$. By repeated application of Lemma \ref{lem: central scaling}, we obtain, for all $i\geq 1$:
\begin{itemize}
\item a finite extension of fields $K_i / K_{i-1}$,
\item an admissible datum $(\sigma_i, \delta_i, u_i)$ on each $Q_i := Q_{K_i}$ which is a strong admissible extension of $(\sigma_{i-1}, \delta_{i-1}, u_{i-1})$,
\item a natural number $n(i)$, the order of $[\sigma_i] \in \Out(Q_i)$,
\item a natural number $\ell(i)$ such that $\sigma_i^{p^{\ell(i)} n(i-1)}$ is conjugation by an element of $\O_i^\times$.
\end{itemize}
But by (AE1), we must have that $n(i)$ divides $n(i-1)$ for each $i\geq 1$, and hence there exists some $j$ such that $n(j-1) = n(j)$. Now simply take $L = K_j$, $\ell = \ell(j)$ and $n' = n(j)$, and use Lemma \ref{lem: admissible data properties} to see that $(\sigma_L, \delta_L, u_L) := (\sigma_j, \delta_j, u_j)$ is an admissible extension of $(\sigma, \delta, u)$.
\end{proof}

\begin{propn}\label{propn: SFOH1}
Assume $Q^+[[x;\sigma,\delta]]$ is not a simple ring. Then there exist a finite extension of fields $K/Z$, an admissible datum $(\sigma_K, \delta_K, u_K)$ on $Q_K$ (with associated maximal order $\O_K$) which is an admissible extension of $(\sigma,\delta,u)$, a natural number $\ell$ and an element $c\in \O_K^\times$ such that $\sigma_K^{p^\ell n'} = \varphi_{c,K}$, where $n'$ is the order of $[\sigma_K] \in \Out(Q_K)$.
\end{propn}

\begin{proof}
If $Z\cap J(\O) \neq 0$, then the result follows immediately from Proposition \ref{propn: unit in finite extension}. If $Z\cap J(\O) = 0$ then, noting that $x-t$ is already invertible in $Q^+[[x; \sigma, \delta]]$ by our assumption $t\in \O^\times$, it follows from Corollary \ref{cor: algebraic with small centre}(iv) that $a\in\O^{\times}$, so taking $K:=Z$ and $c:=a$ we are done.
\end{proof}

\subsection{Algebraic elements}\label{subsec: inner automorphisms}

Assume the setup of \S \ref{subsec: Iwasawa condition}. Here, the additional assumption that $t=-1$ now becomes essential.

 Let $n$ be the order of $[\sigma]\in\Out(Q)$: we now assume further that there exist a natural number $\ell$ and an element $a\in \O^\times$ such that $\sigma^{p^\ell n} = \varphi_a$. We aim to show that, under appropriate conditions, we may find an admissible extension $(\sigma, \delta, u) \leq (\sigma_K, \delta_K, u_K)$ to an appropriate field, a natural number $L$ and an element $c\in 1 + J(\O_K)$ such that $\sigma_K^{p^L n} = \varphi_{c,K}$.

\textit{Strategy.} In this case, if we already have $a\in 1 + J(\O)$, then there is nothing to prove: we simply take $K = Z$ and $L = \ell$. More generally, if there exists an element $\zeta \in Z(\O)^\times$ with $a\in \zeta + J(\O)$, then we can replace $a$ with $c := \zeta^{-1} a$. But in general, there is no such $\zeta\in Z(\O)^\times$: we will first need to raise $a$ to an appropriate $p$'th power to ensure that it is central mod $J(\O)$, and then extend $Z$ to a finite extension field $K$ in which we can find an appropriate element $\zeta\in Z(\O_K)^\times$.

The first step is easy. Write $n = p^k d$, where $p\nmid d$, and fix $m\geq k + \ell$ such that $(\sigma^{p^m}, \delta^{p^m})$ is compatible with $u$ using (AD6). Define the element $b = a^{p^{m-k-\ell}}\in \O^\times$ throughout: it follows from an easy calculation that $\sigma^{p^m d} = \varphi_b$ and that the order of $[\sigma^{p^m}] \in \Out(Q)$ is $d$, and we also have $\deg_u(\varphi_b) \geq d \geq 1$ by Lemma \ref{lem: a regular}(i), so that $b$ is central mod $J(\O)$.

Writing $\overline{\,\cdot\,} : \O\to \O/J(\O)$ for the natural quotient map: we know that $u(bqb^{-1} - q) \geq 1$ for all $q\in \O$, i.e.\ $bqb^{-1} - q \in J(\O)$ for all $q\in \O$, and so $\overline{b} \in Z(\overline{\O})$. If we additionally have $b\in\overline{Z(\O)}$ (a subfield of $Z(\overline{\O})$), then we will be able to find $\zeta$; however, this may not be true, and in general $Z(\overline{\O})$ may be much larger than $\overline{Z(\O)}$. So we proceed with the second step as follows.

The following result is a direct application of Proposition \ref{propn: central polynomial}.
\begin{propn}\label{propn: algebraic with large centre}
Suppose $Q^+[[x;\sigma,\delta]]$ is not a simple ring. Then $\overline{b}$ is algebraic over $\overline{Z(\O)}$.
\end{propn}

\begin{proof}
Firstly, $Q^+[[x;\sigma,\delta]]$ is not a simple ring. Recalling from Proposition \ref{propn: gr(Q)} the isomorphism of graded rings $\gr_u(Q)\cong C[Y, Y^{-1}; \alpha]$, we see that $\gr_u(Q)$ contains the homogeneous unit $Y$ of degree 1, and $\gr_u(Q)_0 \cong C$ is simple. Hence we may apply \cite[Theorem 5.3.7(i)]{jones-woods-3} to see that $Q^+[[x^{p^m};\sigma^{p^m},\delta^{p^m}]]$ is not a simple ring.

For ease of notation, let $y:=x^{p^m}$, and $\tau:=\sigma^{p^m}$. Then since $\delta=\sigma-\id$, we know that $\delta^{p^m}=\tau-\id$, and $Q^+[[x^{p^m};\sigma^{p^m},\delta^{p^m}]]=Q^+[[y;\tau,\tau-\id]]$. 

Since this ring is not simple, it contains a proper, nonzero two-sided ideal $I$. Set $Y:=y+1$, and note that $Y$ is a unit in $Q^+[[y; \tau, \tau - \id]]$, so $Q^+[[y; \tau, \tau - \id]]$ may be identified with the localisation $Q^+[[y; \tau, \tau - \id]]_{Y}$.

Since the order of $[\tau] \in \Out(Q)$ is $d$, we know by Proposition \ref{propn: central polynomial} that there exists a monic polynomial $0\neq f(X)\in Z[X]$ such that $f(b^{-1}Y^d)\in I$; and by dividing by a nonzero coefficient of minimal $u$-value, we can assume without loss of generality that all coefficients lie in $\O$, and that one coefficient has $u$-value $0$. Explicitly, if $f(X) = z_0 + z_1 X + \dots + z_\ell X^\ell$, then
$$f(b^{-1}Y^d) = z_0 + z_1 b^{-1}Y^d + \dots + z_\ell b^{-\ell}Y^{\ell d} \in I.$$
Since $b^{-1}\in\O$, if we expand this expression as a polynomial in $y$, we see that the constant term is equal to $\alpha:=z_0+z_1b^{-1}+\dots+z_{\ell}b^{-\ell}$, and all higher order terms have coefficients in $\O$. In other words, we may write it as $\alpha+hy$, where $h\in\O[[y;\tau,\tau-\id]]$.

Note that, since each $z_i\in Z(\O)$ and $\overline{b} \in Z(\overline{\O})$, we have $\overline{\alpha} \in Z(\overline{\O})$. So we now split our proof into two cases according to whether $\alpha\in J(\O)$ or $\alpha\not\in J(\O)$.

As $\overline{\O}$ is a matrix ring over a division ring by the results of \S \ref{subsec: DVRs}, its centre is a field, and so if we are in the latter case, then $\overline{\alpha} \neq 0$ in $Z(\overline{\O})$ implies that $\overline{\alpha}$ is a unit in $\overline{\O}$. But this implies that $\alpha$ is a unit in $\O$. Hence $\alpha+hy\in I$ is a unit in $Q^+[[y;\tau,\tau-\id]]$, which contradicts that $I$ is a proper ideal of $Q^+[[y;\tau,\tau-\id]]$.

Therefore we must be in the former case where $\alpha\in J(\O)$, and hence $$0=\overline{\alpha}=\overline{z_0}+\overline{z_1}\overline{b}^{-1}+\dots+\overline{z_{\ell-1}}\overline{b}^{-(\ell-1)}+\overline{z_{\ell}}\overline{b^{-\ell}}.$$
As one of the $z_i$ has value $0$, we have $\overline{z_i} \neq 0$, and so this is a polynomial equation for $\overline{b}^{-1}$ with coefficients in $\overline{Z(\O)}$. Hence both $\overline{b}^{-1}$ and $\overline{b}$ are algebraic over $\overline{Z(\O)}$.
\end{proof}

Let $\overline{f}(X) \in \overline{Z(\O)}[X]$ be the minimal polynomial for $\overline{b}$, and lift it to a monic polynomial $f(X) \in Z(\O)[X]$ whose nonzero coefficients all have $u$-value $0$. Let $K$ be the finite extension of $Z$ generated by all roots of $f$, and we can define an admissible datum $(\sigma_K, \delta_K, u_K)$ on $Q_K$ (with associated maximal order $\O_K$) which is a strong admissible extension of $(\sigma, \delta, u)$ by Corollary \ref{cor: existence of strong extensions}.

\begin{rk*}
Adopt the temporary notation $\pi: \O_K \to \O_K/J(\O_K)$, which we will write as a superscript for readability.

In the below lemma, some caution is required. There is a natural inclusion map $\O\to \O_K$ by (AE2), which must preserve units, and it also preserves topological nilpotents by (AE3$'$) and a filt-basis argument. But again as in \ref{rks: more AE stuff}.2, it is not well-behaved with respect to algebraic properties of the filtrations: for instance, we cannot ensure that $J(\O) \subseteq J(\O_K)$. Hence in general there may be no corresponding inclusion map $\overline{\O} \to \O_K^\pi$, and hence no clear relationship between their various subrings.
\end{rk*}

\begin{lem}\label{lem: existence of zeta in Z(OK)^times}
Continue to assume $Q^+[[x; \sigma, \delta]]$ is not simple, and retain the notation above. Then for sufficiently large $T\in\mathbb{N}$, there exists $\zeta \in Z(\O_K)^\times$ such that $b^{p^T} \equiv \zeta \bmod J(\O_K)$.
\end{lem}

\begin{proof}
First note that all the roots of $f$ in $K$ must lie in $\O_K$. Indeed, if $\alpha\in K$ and $z_i \in Z(\O)\subseteq Z(\O_K)$, then $u_K(z_i\alpha^i)\geq iu_K(\alpha)$ for each $i$. Hence if we suppose that $u_K(\alpha) < 0$ and $$z_0 + z_1 \alpha + \dots + z_{N-1} \alpha^{N-1} + \alpha^N = 0,$$ then in particular, $$u_K(z_0+z_1\alpha+\dots+z_{N-1}\alpha^{N-1})\geq (N-1)u_K(\alpha),$$ so $$Nu_K(\alpha)=u_K(-\alpha^{N})=u_K(z_0+z_1\alpha+\dots+z_{N-1}\alpha^{N-1})\geq (N-1)u_K(\alpha),$$ a contradiction.

In particular, $f(X)$ splits over $K\cap \O_K \subseteq Z(\O_K)$. Write $\mathrm{Frob}$ for the Frobenius automorphism of $K$ sending $\beta\in K$ to $\beta^p\in K$, which acts on elements of the polynomial ring $K[X]$ by sending $\beta_0 + \beta_1 X + \dots + \beta_e X^e$ to $\beta_0^p + \beta_1^p X + \dots + \beta_e^p X^e$: then $\mathrm{Frob}^T(f)$ also splits over $Z(\O_K)$, and its reduction $\mathrm{Frob}^T(f)^\pi$ splits over $Z(\O_K)^\pi$.

Now we will choose $T$ such that $J(\O)^{p^T} \subseteq J(\O_K)$ (using (AE3$'$) for the strong extension). By our choice of $f$, we have $\overline{f}(\overline{b}) = 0 \implies f(b) \in J(\O)$, and so $f(b)^{p^T} \in J(\O_K)$. But since $f(b)^{p^T} = \mathrm{Frob}^T(f)(b^{p^T})$ (as the coefficients of $f$ are central in $\O$), we see that $(b^{p^T})^\pi$ is a root of $\mathrm{Frob}^T(f)^\pi$, and so it must lie in $Z(\O_K)^\pi$.

This allows us to choose $\zeta\in Z(\O_K)$ such that $b^{p^T} \equiv \zeta \bmod J(\O_K)$. Now we need only remark that $b^{p^T} \in \O^\times \subseteq \O_K^\times \implies (b^{p^T})^\pi = \zeta^\pi \in (\O_K^\pi)^\times \implies \zeta \in \O_K^\times$, as $\O_K$ is a complete local ring.
\end{proof}

\begin{propn}\label{propn: not simple implies c=1 (mod J(O))}
Continue to assume $Q^+[[x; \sigma, \delta]]$ is not simple, and retain the notation above. Then there exist a natural number $L$ and an element $c\in 1 + J(\O_K)$ such that $\sigma_K^{p^L n} = \varphi_{c,K}$.
\end{propn}

\begin{proof}
By Lemma \ref{lem: existence of zeta in Z(OK)^times}, for sufficiently large $T$, there exists $\zeta\in Z(\O_K)^\times$ such that $b^{p^T} \equiv \zeta \bmod J(\O_K)$. Define $c := \zeta^{-1}b^{p^T} \in \O_K^{\times}$: then $c-1=\zeta^{-1}(b^{p^T}-\zeta)\in J(\O_K)$, and setting $L = T+m$, we have $\sigma_K^{p^L n} = \varphi_{c,K}$.\end{proof}

\subsection{Convergence of inner automorphisms}\label{subsec: convergence of inner automorphisms}

Continue to assume the setup of \S \ref{subsec: Iwasawa condition}. Let $n = p^k d$ be the order of $[\sigma]\in\Out(Q)$, where $p\nmid d$. In light of \S \ref{subsec: inner automorphisms}, we can now assume further that there exist a natural number $\ell \geq k$ and an element $a\in 1 + J(\O)$ such that $\sigma^{p^\ell d} = \varphi_a$. We aim to show that $d = 1$.

Fix $m\geq \ell$ such that $(\sigma^{p^m}, \delta^{p^m})$ is compatible with $u$ using (AD6), and write $\tau = \sigma^{p^m}$. Define the element $c_1 = a^{p^{m-\ell}}\in 1 + J(\O)$: it follows from an easy calculation that $\tau^d = \varphi_{c_1}$ and that the order of $[\tau] \in \Out(Q)$ is $d$.

As $\gcd(d,p) = 1$, there must exist $r\geq 1$ such that $p^r \equiv 1 \bmod d$, i.e.\ $p^r = ds + 1$ for some $s \geq 0$. Fix these $r$ and $s$. Let $s_1 = s$, and for each $j\in\mathbb{N}$, define $$s_{j+1} := p^r s_j + s = (1 + p^r + p^{2r} + \dots + p^{jr})s$$ and $b_j := c_1^{s_j}$.

\begin{lem}\label{converge}
$ $

\begin{enumerate}[label=(\roman*)]
\item For each $j\geq 1$, $\varphi_{b_j} = \tau^{p^{jr} - 1}$.
\item $u(b_i - b_j) \geq p^{jr}$ for all $i > j \geq 1$. In particular, the sequence $(b_j)_{j \geq 1}$ converges in $Q$.
\item Let $b$ be the limit of $b_j$. Then $b \equiv b_j \bmod J(\O)^{p^{jr}}$ for all $j$. In particular, $b\in 1 + J(\O)$.
\item For all $j\geq 1$ and all $q\in \mathcal{O}$, $\varphi_b(q) \equiv \varphi_{b_j}(q) \bmod J(\O)^{p^{j r}}$.
\end{enumerate}
\end{lem}

\begin{proof}
$ $

\begin{enumerate}[label=(\roman*)]
\item We can calculate
$$s_j = (1 + p^r + p^{2r} + \dots + p^{(j-1) r})s = \dfrac{p^{j r} - 1}{p^r - 1}s = \frac{p^{j r} - 1}{d},$$
and so $\varphi_{b_j} = (\varphi_{c_1})^{s_j} = (\tau^d)^{s_j} = \tau^{p^{j r} - 1}.$
\item Since $c_1\in 1 + J(\O)$, we likewise have $b_j\in 1 + J(\O)$ for all $j\geq 1$. In particular, $u(b_j) = u(b_j^{-1}) = 0$ for all $j \geq 1$, and so $u(b_{j + 1} - b_j) = u(b_{j + 1} b_j^{-1} - 1) = u(c_1^{s_{j + 1} - s_j} - 1)$ by Lemma \ref{lem: u-regular elements}. But $s_{j+1} - s_j = p^{j r}s$, and $c_1^{p^{j r}s} - 1 = (c_1^s - 1)^{p^{j r}}$ as we are in characteristic $p$. It follows that $u(b_{j + 1} - b_j)\geq p^{j r}$, and so (by the properties of filtrations) $u(b_i - b_j) \geq p^{j r}$ for all $i > j$. Hence $(b_j)_{j\geq 1}$ is a Cauchy sequence in $Q$, and it converges as $Q$ is complete (Definition \ref{defn: standard filtrations}).
\item Follows immediately from (ii).
\item By (iii), $b_j^{-1} b = 1 + x_j$ and $(b_j^{-1} b)^{-1} = 1 + y_j$ for some $x_j, y_j \in J(\O)^{p^{j r}}$. Now it is easy to see that $(b_j^{-1} b)q(b_j^{-1} b)^{-1} \equiv q \bmod J(\O)^{p^{jr}}$ for any $q\in \O$, so $bqb^{-1} \equiv b_j q b_j^{-1} \bmod J(\O)^{p^{jr}}$ as required. \qedhere
\end{enumerate}
\end{proof}

\begin{propn}\label{propn: tau is inner}
There exists $c\in 1 + J(\O)$ such that $\tau = \varphi_c$.
\end{propn}

\begin{proof}
Let $j\geq 1$ be arbitrary. Since $(\tau-\id)(\O)\subseteq J(\O)$ and we are in characteristic $p$, it follows that $$(\tau^{p^{j r}}-\id)(\O)=(\tau-\id)^{p^{j r}}(\O)\subseteq J(\O)^{p^{j r}}.$$ So, by Lemma \ref{converge}(i), $(\varphi_{b_j}\circ\tau-\id)(\O)\subseteq J(\O)^{p^{j r}}$. Rephrasing this: $\tau(q) \equiv \varphi_{b_j}^{-1}(q) \mod J(\O)^{p^{j r}}$ for all $j \geq 1$, $q\in\O$. But combining this with Lemma \ref{converge}(iv) we see that $\tau(q) \equiv \varphi_{b}^{-1}(q) \mod J(\O)^{p^{j r}}$ for all $j$, and so we must have $\tau = \varphi_{b^{-1}}$.

But $b\in 1 + J(\O)$ by Lemma \ref{converge}(iii), so $c:=b^{-1}\in\O^{\times}$ and $c-1=-c(b-1)\in J(\O)$ as required.\end{proof}

As $\tau = \sigma^{p^m}$ is inner, it follows that the order of $[\sigma]\in\Out(Q)$ must divide $p^m$. It follows immediately that:

\begin{cor}\label{cor: FOH in cases 1,2}
$d = 1$ and $\sigma^{p^\ell}=\varphi_c$.\qed
\end{cor}

\subsection{Proof of remaining main theorems}\label{subsec: proof of remaining theorems}

\begin{proof}[Proof of Theorem \ref{letterthm: FOH implies prime}]
The simplifications of \S \ref{subsec: Iwasawa condition} show that we may assume $\delta = \sigma - \id$. Let $p^s$ be the order of $\sigma|_Z$, and set $(\tau, \varepsilon) = (\sigma^{p^s}, \delta^{p^s})$.

\begin{enumerate}[label=(\alph*)]
\item Suppose that $(\tau, \varepsilon, u)$ satisfies \eqref{FOH}. Then there exist a finite extension of fields $K/Z$ and an admissible extension $(\tau, \varepsilon, u) \leq (\tau_K, \varepsilon_K, u_K)$ on $Q_K$ which satisfies \eqref{SFOH}. That is, there exist $a\in \O_K^\times$ and $\ell\in\mathbb{N}$ such that $\tau_K^{p^\ell}=\varphi_{a,K}$. In particular, the order of $[\tau_K]\in \Out(Q_K)$ is a power of $p$.

If $Q^+[[y_K;\tau_K,\varepsilon_K]]$ is a simple ring, then $Q^+[[y;\tau, \varepsilon]]$ is a simple ring by Theorem \ref{thm: extending scalars for skew power series rings}, and thus $Q^+[[x^{p^s};\sigma^{p^s},\delta^{p^s}]]$ is a simple ring. Applying Corollary \ref{cor: when sigma does not fix Z part 2}, it follows that $Q^+[[x;\sigma,\delta]]$ is simple, and hence prime.

Therefore, we can assume that $Q^+[[y_K;\tau_K,\varepsilon_K]]$ is not a simple ring, so it follows from Proposition \ref{propn: not simple implies c=1 (mod J(O))} that there exists an admissible extension $(\tau_F,\varepsilon_F, u_F)$ of $(\tau_K,\delta_K,u_K)$, an element $c\in 1+J(\O_F)$ and an integer $\ell\in\mathbb{N}$ such that $\tau_F^{p^\ell}=\varphi_{c,F}$. Hence, for any $M \geq \ell$, we have $u_F(c^{p^{M-\ell}} - 1) = u_F((c-1)^{p^{M-\ell}}) > 0$. Applying Theorem \ref{letterthm: SFOH implies prime}(b) with $t = -1$, this means that $Q^+[[y_F; \tau_F, \varepsilon_F]]$ is prime. The same argument as in the previous paragraph will now show that $Q^+[[x;\sigma,\delta]]$ is prime as required.

\item 
Let us suppose for contradiction that $Q^+[[x; \sigma, \delta]]$ is not a simple ring, and hence $Q^+[[y; \tau, \varepsilon]]$ is not simple by Corollary \ref{cor: when sigma does not fix Z part 2}. We will prove that $(\tau, \varepsilon, u)$ satisfies \eqref{FOH}.

We perform the following series of reductions. Write $n$ for the order of $[\tau] \in \Out(Q)$, and choose $a\in Q^\times$ such that $\tau^n = \varphi_a$, as in \S \ref{subsec: central scaling}.
\begin{itemize}
\item By Proposition \ref{propn: SFOH1}, after replacing $Q$ and $(\tau, \varepsilon, u)$ by some appropriate $Q_K$ and $(\tau_K, \varepsilon_K, u_K)$, we may assume that there exists a natural number $\ell$ and an element $b\in \O_K^\times$ such that $\tau_K^{p^\ell n} = \varphi_{b,K}$. This puts us into the situation of \S \ref{subsec: inner automorphisms}.
\item By Proposition \ref{propn: not simple implies c=1 (mod J(O))}, after replacing $Q_K$ and $(\tau_K, \varepsilon_K, u_K)$ once more by some appropriate $Q_L$ and $(\tau_L, \varepsilon_L, u_L)$, we may assume that there exist a natural number $M$ and an element $c\in 1 + J(\O_L)$ such that $\tau_L^{p^L n} = \varphi_{c,L}$. This puts us into the situation of \S \ref{subsec: convergence of inner automorphisms}.
\item Corollary \ref{cor: FOH in cases 1,2} now shows that, in fact, for this $c\in 1 + J(\O_L) \subseteq \O_L^\times$, there exists a natural number $M$ such that $\tau_L^{p^M} = \varphi_{c,L}$. But this is \eqref{SFOH}.\qedhere
\end{itemize}
\end{enumerate}
\end{proof}

\begin{proof}[Proof of Theorem \ref{letterthm: compatible Iwasawa SPSR is prime}]
The technical hypotheses here ensure that the data $(R, w_0)$ satisfies the (filt) hypothesis from \cite[\S 1.2]{jones-woods-3}. We now make significant use of the procedure outlined in \cite[\S 4]{jones-woods-3}.

Construct $(Q(R), w)$ via homogeneous localisation \cite[\S 3]{li-ore-sets} and complete to get a filtered artinian ring $(\widehat{Q}, w)$, both as detailed in \cite[Procedure 4.1.2]{jones-woods-3}. Note that $(\sigma, \delta)$ extends to $\widehat{Q}$ by \cite[Lemma 4.3.1]{jones-woods-3}, and $\delta$ is still defined by $\delta(q)=tq-\sigma(q)t$ for all $q\in\widehat{Q}$.

We are now in the setup of \cite[\S 4.5]{jones-woods-3}, so let $M_1$ be a maximal ideal of $\widehat{Q}$, let $N$ be the intersection of its (finite) $\sigma$-orbit, and write $Q = \widehat{Q}/N$ as described there. Since $\delta$ is inner, we have $\delta(N) \subseteq N$ \cite[Lemma 4.5.4]{jones-woods-3}), and hence $(\sigma, \delta)$ descends from $\widehat{Q}$ to $Q$. From \cite[Lemma 4.5.6]{jones-woods-3} and the subsequent discussion, we have:
\begin{itemize}
\item $Q$ is a semisimple artinian ring with $p^\ell$ simple factors $Q_1, \dots, Q_{p^\ell}$.
\item These simple factors are permuted transitively by $\sigma$.
\item Each $Q_i$ admits a standard filtration $u_i$, and we can set $u := u_1 \times \dots \times u_{p^\ell}$. (Note that $\sigma$ does \emph{not} permute these transitively.)
\end{itemize}

Furthermore, each simple artinian factor $(Q_i,u_i)$ arises as a standard completion of $Q(R)$ in the same way as Theorem \ref{thm: filtered localisation}.1, and thus $\gr_{u_i}(Q_i)$ is finitely generated over its centre by Theorem \ref{thm: filtered localisation}.2. In particular, $Q_i=M_n(Q(D_i))$ for some complete DVR $D_i$, and $\gr(D_i)$ is finitely generated over its centre, i.e.\ $D_i$ is admissible.

Now, applying \cite[Theorem 4.5.7]{jones-woods-3}, we find that $(\sigma, \delta)$ is quasi-compatible with $u$, and so the restriction of $(\sigma^{p^\ell},\delta^{p^\ell} )$ to $Q_i$ remains quasi-compatible with $u_i$. In fact, it follows from \cite[Theorem 5.2.2]{jones-woods-3} shows that the restriction of $(\sigma^{p^m},\delta^{p^m})$ to $Q_i$ is compatible with $u_i$ for sufficiently high $m$.  

In particular, if $[\sigma^{p^\ell}]\in Out(Q_i)$ has finite order, then $\sigma^{p^\ell}|_{Z(Q_i)}$ has order $p^s$ by Lemma \ref{lem: identical on centre}, so we see that $(\sigma^{p^{\ell+s}},\delta^{p^{\ell+s}},u_i)$ satisfies (AD1--6) of Definition \ref{defn: admissible data}, and hence it defines an admissible datum on $Q_i$.

Moreover, since $t\in R^\times$ and $R\subseteq\O:=u^{-1}([0,\infty])$ by \cite[Theorem 4.1.1(ii)]{jones-woods-3}, we see that $t\in\O^\times$, and hence $u(t^N)=0$ for all $N\in\mathbb{N}$. Moreover, if $t_i$ is the image of $t$ under the projection to $Q_i$, then $u(t_i)=u(t_i^{-1})=0$, so the restriction of the skew derivation of $(\sigma^{p^\ell},\delta^{p^\ell} )$ to $Q_i$ satisfies the conditions of Theorem \ref{letterthm: FOH implies prime}.

We now analyse the result by cases.
\begin{itemize}
\item
If $[\sigma^{p^\ell}]\in \Out(Q_i)$ has infinite order, then $Q_i^+[[y;\sigma^{p^\ell},\delta^{p^\ell}]]$ is a simple ring by \cite[Theorem 5.4.2]{jones-woods-3}.
\item
If $[\sigma^{p^\ell}]\in \Out(Q_i)$ has finite order, then $Q_i^+[[y;\sigma^{p^\ell},\delta^{p^\ell}]]$ is a prime ring by Theorem \ref{letterthm: FOH implies prime}.
\end{itemize}
In either case, $Q_i^+[[y;\sigma^{p^\ell},\delta^{p^\ell}]]$ is a prime ring. Using \cite[Theorem 5.2.3]{jones-woods-3}, it follows that $Q^+[[x;\sigma,\delta]]$ and $R[[x;\sigma,\delta]]$ are prime.
\end{proof}

\begin{proof}[Proof of Theorem \ref{letterthm: primes in completely solvable Iwasawa algebras}]
Since $G/H$ is $p$-valuable and completely solvable, there exist a chain of normal subgroups $H=H_0\subseteq H_1\subseteq\dots\subseteq H_m=G$ such that $H_{i}/H_{i-1}\cong\mathbb{Z}_p$ for each $i>0$, i.e.\ $H_i/H_{i-1}$ it is generated by a single element $g_i\in H_i$. Therefore, $kH_i\cong kH_{i-1}[[x_i;\sigma_i,\delta_i]]$ where $\sigma_i$ is conjugation by $g_i$, and $\delta_i=\sigma_i-\id$, and hence $$R_i := kH_i/QkH_i\cong (kH_{i-1}/QkH_{i-1})[[x_i;\sigma_i;\delta_i]]$$
by \cite[3.14(iv)]{letzter-noeth-skew}. Let $w_0$ be the filtration on $kG$ given by its $p$-valuation \cite[Lemma 6.2]{ardakovInv}, so that $\gr_{w_0}(kG)$ is Noetherian and commutative, and let $w$ be the quotient filtration on $R := R_m$. Note that $\gr_w(R)$ is a quotient of $\gr_{w_0}(kG)$ \cite[Chapter I, \S 4, Theorem 4(1)]{LVO}, and $(\sigma_i, \delta_i)$ remains compatible with $w$.

\textbf{Case 1: $Q$ is the maximal ideal.} In this case, $R_0 = k$, and so $R$ is a skew power series ring over $k$, and we may calculate $\gr_w(R)$ explicitly as an iterated skew polynomial ring over $k$ \cite[Theorem D]{woods-SPS-dim}. This is a domain \cite[Theorem 1.2.9(i)]{MR}, and hence so is $R$.

\textbf{Case 2: $Q$ is a non-maximal prime ideal.} Then $R_0$ is a prime ring, so let us assume for induction that $R_{i-1}$ is prime for some $i\geq 1$. By our assumption that $Q$ is not maximal, we know that $\gr_w(R_{i-1})$ contains a non-nilpotent element of positive degree, by the argument of \cite[Example 4.12]{jones-woods-1}. So it follows from Theorem \ref{letterthm: compatible Iwasawa SPSR is prime} that $R_i\cong R_{i-1}[[x_i;\sigma_i;\sigma_i-\id]]$ is a prime ring. The result follows by induction.

\textbf{Case 3: $Q$ is a $G$-prime ideal.} In this case, by \cite[3.12]{letzter-noeth-skew}, there is a $G$-orbit of prime ideals $T_1, \dots, T_t$ of $kH$ such that $T_1 \cap \dots \cap T_t = Q$. Write $N_i$ for the normaliser in $G$ of $T_i$, and write $N = N_1 \cap \dots \cap N_t$: then $[G: N_i]$ is finite, and hence $[G:N]$ is finite. Hence we may find a normal subgroup $L\lhd G$, contained in $N$ but containing $H$, which is \emph{open} (i.e.\ has finite index in $G$) by \cite[Lemma 1.10]{woods-struct-of-G} and \cite[Proposition 5.9]{ardakovInv}.

Clearly $N_i/H$ and $L/H$ are still completely solvable, and it is easy to see that they are also $p$-valuable \cite[Chapitre III, D\'efinition 2.1.2]{lazard}. Now, by Case 2, we can see that both $T_ikN_i$ and $T_ikL$ are prime ideals in $kN_i$ and $kL$ respectively, and \cite[Lemma 5.1(a)]{ardakovInv} shows that
$$QkL = (T_1 \cap \dots \cap T_t)kL = T_1kL \cap \dots \cap T_tkL,$$
an intersection of a finite $G$-orbit of prime ideals of $kL$: hence $QkL$ is a $G$-prime ideal of $kL$. So it remains to show that $kG/QkG = (kL/QkL)*(G/L)$ is a prime ring. But by \cite[Corollary 14.8]{passmanICP}, this is prime if and only if $(kL/T_1kL)*(N_1/L) = kN_1/T_1kN_1$ is prime, which we have just proved.
\end{proof}

\bibliography{biblio}
\bibliographystyle{plain}

\end{document}